\documentclass[a4paper,10pt]{amsart}
\usepackage[english]{babel}
\usepackage{a4wide}
\usepackage[utf8]{inputenc}
\usepackage{csquotes}
\usepackage{comment}
\usepackage{tikz}
\usepackage{ulem}
\usetikzlibrary{fit,backgrounds, calc}
\usetikzlibrary{matrix, graphs, positioning} 

\title[Approximate Message Passing for non-Symmetric random matrices]{Approximate Message Passing for general non-Symmetric random matrices}
\author[Gueddari et al.]{Mohammed-Younes Gueddari, Walid Hachem, Jamal Najim}

\date{\today\\
CNRS, Laboratoire d'informatique Gaspard Monge (LIGM / UMR 8049),
Universit\'e Gustave Eiffel, France} 

\usepackage{amssymb,amsmath,amscd,amsfonts,amsthm,bbm,mathrsfs,dsfont}
\usepackage{booktabs}
\usepackage[colorlinks=true, linkcolor=blue, citecolor=blue]{hyperref}
\usepackage{mathabx}
\usepackage{graphicx}
\usepackage{color}
\usepackage{accents}
\usepackage{subcaption}
\usepackage{enumerate,enumitem}
\usepackage{tensor} 
\usepackage{authblk}

\DeclareUnicodeCharacter{00A0}{~}

\newtheorem{theorem}{Theorem}[section]
\newtheorem{proposition}[theorem]{Proposition}
\newtheorem{lemma}[theorem]{Lemma}

\newtheorem{definition}{Definition}[section]
\newtheorem{remark}[definition]{Remark}

\newtheorem{assumption}{Assumption A-\hspace{-0.15cm}}

\DeclareMathOperator{\AMP}{AMP}
\DeclareMathOperator{\AMPZ}{AMP-Z}
\DeclareMathOperator{\AMPW}{AMP-W}
\DeclareMathOperator{\DE}{DE}
\DeclareMathOperator{\GOE}{GOE}

\newcommand{\1}{\mathbbm 1}
\newcommand{\T}{\top}

\newcommand{\CC}{\mathbb{C}}
\newcommand{\PP}{{{\mathbb P}}} 
\newcommand{\EE}{{{\mathbb E}}} 
\newcommand{\NN}{{{\mathbb N}}} 
\newcommand{\RR}{{{\mathbb R}}}

\newcommand{\cA}{{\mathcal A}} 
\newcommand{\cB}{{\mathcal B}}

\newcommand{\cF}{{\mathcal F}} 
\newcommand{\cG}{{\mathcal G}}

\newcommand{\cK}{{\mathcal K}} 
\newcommand{\cL}{{\mathcal L}} 
 
\newcommand{\cN}{{\mathcal N}}

\newcommand{\cR}{{\mathcal R}} 
\newcommand{\cQ}{{\mathcal Q}} 
\newcommand{\cS}{{\mathcal S}} 
\newcommand{\cT}{{\mathcal T}}

\newcommand{\cU}{{\mathcal U}}

\newcommand{\tW}{{\widetilde W}}

\newcommand{\chR}{{\widecheck R}}

\newcommand{\bu}{\bs u} 
\newcommand{\bv}{\bs v} 

\newcommand{\bx}{\bs x} 
\newcommand{\bcx}{\bs x} 
\newcommand{\cx}{\check x} 

\newcommand{\tX}{{\widetilde X}}

\newcommand{\by}{\bs y}

\newcommand{\bZ}{\bs Z} 
 
\newcommand{\cZ}{{\widecheck Z}}

\newcommand{\balpha}{\bs \alpha}

\newcommand{\bG}{\bs G}

\newcommand{\bU}{\bs U}

\newcommand{\bm}{{\bs m}} 

\newcommand{\bz}{\bs z} 
\newcommand{\tbz}{\tilde{\bs z}}

\newcommand{\Cmom}{C_{\text{mom}}}

\newcommand{\bs}{\boldsymbol}
\DeclareMathOperator*{\diag}{diag}
\newcommand{\eqlaw}{\stackrel{\cL}{=}}  
\newcommand{\toprobalong}{\xrightarrow[n\to\infty]{\mathbb P}}

\newcommand{\tolong}{\xrightarrow[n\to\infty]{}}

\makeatletter
\renewcommand\paragraph{\@startsection{paragraph}{4}{\z@}%
  {-3.25ex\@plus -1ex \@minus -.2ex}%
  {1.5ex \@plus .2ex}%
  {\normalfont\normalsize\bfseries}}
\makeatother

\begin{document}

\maketitle

\begin{abstract}

    Approximate Message Passing (AMP) algorithms are a family of iterative
algorithms based on large random matrices with the special property of tracking
the statistical properties 
of their iterates. They are used in various fields such as Statistical
Physics, Machine learning, Communication systems, Theoretical ecology, etc.

In this article we consider AMP algorithms based on non-Symmetric random
matrices with a general variance profile, possibly sparse, a general covariance
profile, and non-Gaussian entries. We hence substantially extend the results on
Elliptic random matrices that we developed in~\cite{gueddari2024elliptic}. From
a technical point of view, we enhance the combinatorial techniques developed in
Bayati et al. \cite{Bayati_2015} and in \cite{HACHEM2024104276}.

    Our main motivation is the understanding of equilibria of large food-webs described by Lotka-Volterra systems of ODE, in the continuation of the works of~\cite{HACHEM2024104276}, Akjouj et al. \cite{akjouj2024equilibria} and~\cite{gueddari2024elliptic}, but the versatility of the model studied might be of interest beyond these particular applications.

\end{abstract}

\section{Introduction}

Approximate Message Passing (AMP) refers to a class of iterative algorithms that are built around a large random matrix, producing at each step a high-dimensional $\mathbb{R}^n$-valued random vector ($n \gg 1$) whose elements' empirical distribution can be identified as $n$ goes to infinity. These algorithms take the following form 
$$
\bs{x}^{t+1} = W h_t(\bs{x}^t) - \{\text{corrective term}\}\,,
$$  
where $\bs{x}^t = (x_i^t)$ is the $n \times 1$ vector at iteration $t$, $W$ is a $n \times n$ random matrix, and $h_t(\bs{x}^t) = (h_t(x_i^t))_i$ is a vector based on the so-called \emph{activation function} $h_t:\mathbb{R} \to \mathbb{R}$. The corrective term, known as the Onsager term, is carefully defined to facilitate the description of the statistical properties of $\bs{x}^t$ as $n \to \infty$. 

In the fields of machine learning and statistical estimation, AMP algorithms were originally developed for studying compressed sensing and sparse signal recovery problems \cite{Donoho_2009,bayati2011dynamics}. They have since found applications across various fields, including high-dimensional estimation \cite{desh-abb-mon-17,lel-mio-ptrf19}, communication theory \cite{bar-krz-it17,rush-gre-ven-17}, statistical physics \cite{mon-siam-19}, theoretical ecology \cite{akjouj2024equilibria,HACHEM2024104276,gueddari2024elliptic}, etc. AMP algorithms have undergone extensive recent developments and the goal of this article is to extend the AMP framework to general non-symmetric random matrices $W$. 

In general, the random matrix model $W$ may differ depending on the considered application, and most of AMP algorithms focus on symmetric matrices. For instance, in the problem of low-rank information extraction from noisy data matrix, the goal is to estimate the $n\times 1$ signal $\bs{x}^\star$ from noisy observations  
\begin{equation}
\label{eq:estimationProblem}
    Y = \sqrt{\lambda} \bs{x}^\star (\bs{x}^\star)^\top + W\,,
\end{equation}
where $W$ is a random matrix. In \cite{deshpande2014informationtheoreticallyoptimalsparsepca} and \cite{montanari2019estimationlowrankmatricesapproximate}, the authors develop an AMP algorithm involving a symmetric matrix $ W=\frac{1}{\sqrt{n}}G$ where $G$ is drawn from the Gaussian Orthogonal Ensemble (GOE$(n)$) to study the problem \eqref{eq:estimationProblem}. More precisely, each entry $G_{ij}\sim {\mathcal N}(0,1 + 1_{(i=j)})$, where $1_{(i=j)}$ equals one if $i=j$ and zero else, and all the entries on and above the diagonal are independent. The $1/\sqrt{n}$ normalization factor is standard in Random Matrix Theory and has the effect to ensure that the spectral norm of $W$ is $\mathcal{O}(1)$.

In \cite{jav-mon-13, beh-ree-pmlr22, gui-ko-krz-zde-(arxiv)22,pak-jus-krz-23},  
the authors develop an AMP
algorithm involving a symmetric random matrix $W$ with a block-wise variance
profile $S$ to study the problem \eqref{eq:estimationProblem} in the case of
an inhomogeneous noise. More precisely, $W$ is now written as 
\begin{equation}
\label{eq:example-VP} W = \frac 1{\sqrt{n}}\, S^{\odot 1/2} \odot G\,,
\end{equation}
where $G\sim\GOE(n)$ and $S$ is a symmetric, deterministic, block-constant matrix of 
non-negative elements. Matrix $S$ has a finite number of rectangular blocks 
which dimensions scale with $n$, the elements of $S^{\odot 1/2}$ are 
the square roots of those of $S$, and $\odot$ is the Hadamard or entry-wise 
product. In the recent paper \cite{bao2023leave}, Bao et al. consider an 
AMP algorithm based on Gaussian matrices with a variance profile and provide 
non-asymptotic results. 

Our main motivation to develop AMP algorithms associated to new matrix models comes from theoretical ecology and the study of large Lotka-Volterra systems of ODEs.
In such models, the random matrix $W$ is used to model the interactions
between $n$ living species that coexist within an ecosystem, and the time
evolution of the abundances is described by the multi-dimensional Lotka-Volterra
differential equation. In \cite{akjouj2024equilibria},
Akjouj \textit{et al.}~consider the GOE model for the matrix of interactions,
and use an AMP approach to describe the statistical properties of the
equilibrium point of the resulting Lotka-Volterra dynamical system when this
equilibrium is globally stable.  Dealing with a more realistic interaction
matrix model, \cite{HACHEM2024104276} considers a symmetric random matrix
with a variance profile as in~\eqref{eq:example-VP}, with the main difference
that the variance profile matrix $S$ can be sparse. Including correlations
between the elements of the interaction matrix is an important feature in
theoretical ecology.  In this direction, a non-symmetric elliptic matrix $W$ is
considered in \cite{gueddari2024elliptic}, where each entry pair
$(\sqrt{n}W_{ij},\sqrt{n} W_{ji})$ is a standard two-dimensional centered
Gaussian vector with a covariance $\rho\in[-1,1]$, and where all the different
pairs are independent. 

All these cases are particular cases of the model we study in this article.

\subsection{The random matrix model} 
\label{subsec:MatrixModel}

The model under investigation here combines an arbitrary variance profile, possibly sparse, with a correlation profile. To this end, we first introduce the notion of a $T$-correlated matrix. Let $[n]=\{1,\cdots, n\}$.

\begin{definition}
    \label{def:corr_matrix}
    Let $T = (\tau_{ij})_{1\leq i,j\leq n}$ be a symmetric $n\times n$ matrix with entries in $[-1, 1]$. The $n\times n$ random matrix $X$ is $T$-correlated if
    \begin{itemize}
        \item[-] Every entry $X_{ij}$ is centered random variable with variance $1$.
        \item[-] For $(i,j)\in [n]^2$, $i<j$, the covariance matrix of the pair $(X_{ij}, X_{ji})$ is
              $$ \begin{pmatrix}
                      1         & \tau_{ij} \\
                      \tau_{ji} & 1
                  \end{pmatrix}. $$
        \item[-] The random elements in the set $\{X_{ii}, (X_{ij}, X_{ji}),\,(i,j)\in [n]^2,\, i<j\}$ are independent.
    \end{itemize}
\end{definition}
\begin{remark}
    Notice that the diagonal elements of $T$ are not specified in this definition. A natural convention could be to set $\tau_{ii} = 1$, as it represents the correlation of $X_{ii}$ with itself, but their exact values (as long as it is bounded) have no impact on the presented results.
\end{remark}

Let $X$ be a ${\mathbb R}^{n\times n}$--valued $T$-correlated matrix and $S =
(s_{ij} )_{i,j\in[n]} $ be a deterministic $n\times
n$ matrix with non-negative elements. The random matrix model considered in
this paper is
\begin{equation}\label{def:W}
W=S^{\odot 1/2} \odot X = \left(\sqrt{s_{ij}} X_{ij}\right)_{1\le i,j\le n}\, .
\end{equation}
Notice that the entries need not to be Gaussian and contrary to \eqref{eq:example-VP}, the normalization is embedded into matrix $S$. We refer to $S$ as the \emph{variance profile} of matrix $W$ and to $T$ as its \emph{correlation profile}. Such a model is fairly general as it encompasses most of the classical random matrix models (Wigner, Elliptic, Circular models) and many important features required in the applications (sparsity, variance profile, etc.).

\subsection{A primer to Approximate Message Passing}

For a random matrix $W$ such that $\sqrt{n}W\sim \textrm{GOE}(n)$, an AMP algorithm starting at $\bs{x}^0 =(x_0, \cdots, x_0)^\top$ using a set of Lipschitz activation functions $(h_t)_{t\geq 0}$ is given by the following recursion equation; for all $t\geq 0$,
\begin{equation}\label{eq:AMP-GOE} \bs{x}^{t+1} =  W h_t(\bs{x}^t) - b_t h_{t-1}(\bs{x}^{t-1}) \quad \text{where} \quad b_t = \frac{1}{n}\sum_{i=1}^n h_t'(x_i^t)\,,
\end{equation}
with the convention that $h_{-1} \equiv 0$.

The crucial term in this recursion is the Onsager term, i.e.
``$\operatorname{ONS}_t := b_t h_{t-1}(\bs{x}^{t-1})$" that we subtract from
the power method iteration term at each step $t$. The effect of the Onsager
term is that for a fixed $t$ and as $n\to\infty$, it ``cancels'' the dependence
due to the repeated use of matrix $W$ at each iteration: 
$$ 
\bs{x}^{t+1}  = W h_t(Wh_{t-1}(W\cdots)-\operatorname{ONS}_{t-1}) -\operatorname{ONS}_t\,. 
$$
With the correction of the Onsager term, the asymptotic behavior of $\bs{x}^t$ is similar to the behavior of $ \tilde{\bs{x}}^t$ generated with the ``power method iteration" but with a new sampled independent random matrix $W^t$ at each step $t$, i.e.
$$
\tilde{\bs{x}}^{t+1} = W^t h_t(\tilde{\bs{x}}^t) \qquad \text{with} \qquad \sqrt{n}W^t \overset{i.i.d.}{\sim} \GOE(n)\,.
$$
Notice that in the latter case, it is easy to characterize the asymptotic behavior of the empirical distribution $\mu^{ \tilde{\bs{x}}^t}$ of the entries of the vector $\tilde{\bs{x}}^t=(\tilde{x}^t_i)$, 
$$ 
\mu^{\tilde{\bs{x}}^t} = \frac{1}{n} \sum_{i=1}^n \delta_{\tilde{x}^t_i}\,. 
$$
Roughly speaking $\mu^{\bs{x}^t}\approx \mu^{\tilde{\bs{x}}^t}$ as $n\to
\infty$. Beware however that the correlation between consecutive iterations
$\bs{x}^t$ and $\bs{x}^{t+1}$ differs from the correlation between iterates
$\tilde{\bs{x}}^t$ and $\tilde{\bs{x}}^{t+1}$ which turn out to be
asymptotically decorrelated. 

Given the iterates $\bs{x}^1=(x^1_i),\cdots, \bs{x}^t=(x^t_i)$ produced by \eqref{eq:AMP-GOE}, the main result associated to AMP is the description of the limiting distribution of 
$$
\mu^{(\bs{x}^1,\cdots, \bs{x}^t)}:=\frac 1n \sum_{i=1}^n \delta_{\left(x^1_i,\cdots, x_i^t\right)}
$$
as $n\to\infty$ in terms of a multivariate Gaussian vector whose covariance matrix is described by the Density Evolution Equations.

\subsection{Density Evolution Equations}

Density Evolution (DE) equations are a set of recursive equations that define a sequence of deterministic, symmetric, positive semi-definite matrices, which are central objects in the analysis of AMP algorithms. 
These matrices are covariance matrices associated to multivariate normal distributions which describe the asymptotic behavior of the AMP iterates (and their correlations) as $n$ goes to infinity. 

Given a set of activation functions $h_t:\RR \to \RR$ and a initial constant vector $\bs{x}^0 = x_0 \bs{1}_n\in \RR^n$, the Density Evolution equations associated to the AMP \eqref{eq:AMP-GOE} with $\sqrt{n}W\sim \textrm{GOE}(n)$ is a sequence of $t\times t$ matrices $(R^t)_{t\in \NN^\star}$ defined recursively as follows,

$$ R^1 = \left(h(x_0)\right)^2 \qquad \text{and} \qquad R^{t+1} = \EE \begin{bmatrix}
    h_t(x_0) \\
    h_t(Z_1)\\
    \cdots \\
    h_t(Z_t)
\end{bmatrix}\begin{bmatrix}
    h_t(x_0) &
    h_t(Z_1) &
    \cdots &
    h_t(Z_t)
\end{bmatrix}\,, $$
where $(Z_1, \cdots, Z_{t})\sim \cN_{t}(0,R^{t})$. Notice that in particular, the variances $\sigma_t^2 = \EE\,Z^2_t$ satisfy a simple recursion equation given by: 
\begin{equation}
\label{eq:DEVariances}
    \sigma_0^2 = h^2_0(x_0) \qquad \text{and} \qquad \sigma_{t+1}^2 = \EE\,h^2_t(\sigma_t \xi)\qquad \textrm{where}\quad \xi\sim \cN(0,1)\,.
\end{equation}
With the family of covariance matrices $(R^t)$ at hand, we can express the limiting statistical properties of measure $\mu^{(\bs{x}^1,\cdots, \bs{x}^t)}$ which captures both the asymptotic properties of the iterates $\bs{x}^t$ and the dependence between the iterates $\bs{x}^1,\cdots, \bs{x}^t$:
$$
\mu^{(\bs{x}^1,\cdots, \bs{x}^t)}\quad \xrightarrow[n\to\infty]{weak,L^2}\quad  \cN_{t}(0,R^{t})\, 
$$
in probability (see \cite{feng2022unifying} for sharper convergence results). Stated differently, for any test functions $\varphi:\RR^t\to \RR$ and $\psi:\RR\to \RR$,
\begin{equation}\label{eq:conv-AMP-GOE}
\frac{1}{n}\sum_{i=1}^n \varphi(x_i^1,\cdots,x_i^t) \toprobalong \EE\,\varphi(Z_1,\cdots, Z_t)  \qquad\textrm{and}\qquad 
\frac 1n \sum_{i=1}^n \psi(x_i^t) \toprobalong \EE\,\psi(\sigma_t \xi)\ ,
\end{equation}
where $\xi\sim \cN(0,1)$, $\xrightarrow[]{\PP}$ stands for the convergence in probability and $(\sigma_t)_{t\geq 0}$ is a sequence of positive numbers defined recursively by \eqref{eq:DEVariances}.

In \cite{gueddari2024elliptic}, we show that the DE equations used to study an
AMP with an elliptic matrix do not depend on the correlation coefficient, the
latter being included in the formulation of the AMP recursion, and more
specifically in the Onsager term. In \cite{HACHEM2024104276}, the case of
a symmetric random matrix with a general variance profile $S$ is handled. 

In the case of a general variance profile, the description of the asymptotic behavior of the iterates becomes more involved and instead of having a multivariate Gaussian vector $(Z_1,\cdots, Z_t)$ we have a family of $n$-dimensional vectors $(\bs{Z}^1,\cdots, \bs{Z}^t)$.

In the following definition, we give a general description of the DE equations associated to a variance profile matrix $S$.
We now consider that the activation function depend on an additional parameter $\bs{\eta}$ and we no longer express the dependence in $t$
using a subscript, it is now included in the arguments of function $h$.

\begin{definition}
    \label{def:de}
    Let $\bs{x}^0=(x^0_i)\in \RR^n$ and $\bs{\eta}=(\eta_i) \in \RR^n$ be two deterministic vectors, $S = \left(s_{ij}\right)_{1\leq i,j \leq n}$ a matrix with non-negative elements and $h:\RR^2\times \NN \to \RR$ an activation function. 
    \begin{enumerate}[label=\alph*)]
    \item \emph{Initialization.} For any $i\in [n]$, define the non-negative numbers $H_i^0$ and $R_i^1$ as
    \begin{equation*}
        H_i^0 := h^2(x_i^0, \eta_i, 0) \qquad \text{and} \qquad R_i^1 := \sum_{j = 1}^{n} s_{ij}H_j^0\,.
    \end{equation*}
    Let $Z_i^1\sim\cN(0,R^1_i)$, assume that for all $i\in [n]$, the $Z_i^1$'s are independent and set 
    $$\bs{Z}^1=(Z^1_i)_{i\in [n]}\ .$$ 

    \item \emph{Step 1.} Let $\bs{Z}^1=(Z^1_i)_{i\in [n]}$ be given and $i\in[n]$ be fixed. Let 
    $$
    H_i^1=\EE \begin{bmatrix}
            h(x^{0}_i, \eta_i, 0) \\
            h(Z^{1}_i, \eta_i, 1) 
            \end{bmatrix}  \left[h(x^{0}_i, \eta_i, 0)\,,\ h(Z^{1}_i, \eta_i, 1)\right]
            \qquad \textrm{and}\qquad R^2_i=\sum_{j=1}^n s_{ij} H_j^1\,.
    $$
Notice that the $1\times1$ upper left corner of $R_i^2$ coincides with $R_i^1$. Let $Z_i^2$ be such that $\vec{Z}_i^2:=(Z_i^1,Z_i^2)\sim\cN_2(0,R^2_i)$, and such that for all $i\in [n]$, the $\vec{Z}_i^2$'s are independent. Set $\bs{Z}^2=(Z^2_i)$. \\
    \item \emph{Step t.} Let the covariance matrix $R_i^t\in \RR^{t\times t}$ and the $\RR^n$ vectors $\bs{Z}^1,\cdots, \bs{Z}^t$ be given, where 
    $$
    \vec{Z}_i^t:=(Z_i^1,\cdots, Z_i^t) \sim {\mathcal N}_t(0,R_i^t)\ ,
    $$
    and where all the $\vec{Z}_i^t$'s are independent for $i\in [n]$. Let  
    $$
        H_i^t =  \EE \begin{bmatrix}
            h(x^{0}_i, \eta_i, 0) \\
            h(Z^{1}_i, \eta_i, 1) \\
            \vdots                \\
            h(Z^{t}_i, \eta_i, t)\end{bmatrix}
        \begin{bmatrix}
            h(x^{0}_i, \eta_i, 0) &
            h(Z^{1}_i, \eta_i, 1) &
            \cdots                & h(Z^{t}_i, \eta_i, t)\end{bmatrix}$$
            and $R_i^{t+1} =  \sum_{j=1}^n s_{ij}H_j^t$. Notice that the $t\times t$ upper left corner of matrix $R_i^{t+1}$ coincides with $R_i^t$.
            Let $Z_i^{t+1}$ be such that 
    $$
    \vec{Z}_i^{t+1}:=(Z_i^1,Z_i^2,\cdots, Z_i^{t+1}) \sim\cN_{t+1}(0,R^{t+1}_i)
    $$
    and such that for all $i\in [n]$, the $\vec{Z}_i^{t+1}$'s are independent. Set $\bs{Z}^{t+1}=(Z^{t+1}_i)$.
        \end{enumerate}
    Consider the sequence of $n$-dimensional Gaussian random vectors $\left(\bs{Z}^{t}\right)_{t\in \mathbb{N}}$. We denote 
    $$
    \left(\bs{Z}^1, \cdots, \bs{Z}^t\right) \sim \DE\left(S,h, \bs{x}^0, \bs{\eta}, t\right)\, .
    $$
    We also define $Z_i=(Z_i^t)_{t\ge 1}$. The sequences $\{Z_i\}_{i\in [n]}$ are centered, Gaussian, and independent. The notations $\bs{Z}^t$ and $\vec{Z}_i^t$ are described in Fig. \ref{fig:notations-Z}.
\end{definition}

\begin{figure}
    \begin{tikzpicture}

\node[anchor=east] at (-2.5, 0) {$
\left(\bs{Z}^1, \cdots, \bs{Z}^t\right) =$};

\matrix[matrix of math nodes,left delimiter={(},right delimiter={)}, column sep=.5cm] (m) {
    Z_1^1 & Z_1^2 & \cdots & Z_1^t \\
    Z_2^1 & Z_2^2 & \cdots & Z_2^t \\
    \vdots & \vdots &  & \vdots \\
    Z_i^1 & Z_i^2 & \cdots & Z_i^t \\
    \vdots & \vdots &  & \vdots \\
    Z_n^1 & Z_n^2 & \cdots & Z_n^t \\
};

\draw[thick, red, rounded corners] 
  ($(m-4-1.south west)+(-0.2,-0.1)$) 
  rectangle 
  ($(m-4-4.north east)+(0.2,0.1)$);

\draw[thick, blue, rounded corners] 
  ($(m-1-4.north west)+(-0.1,0.2)$) 
  rectangle 
  ($(m-6-4.south east)+(0.1,-0.2)$);

\node[anchor=west] at ($(m-4-3.east)+(1.5,0)$) {$\rightarrow\vec Z^{t}_i\in \RR^t$};
\node[anchor=north] at ($(m-4-4.south)+(0,-1.5)$) {$\substack{\downarrow\\ \bs{Z}^t\in \RR^n}$};
\end{tikzpicture}

\caption{The Gaussian matrix $(\bs{Z}^1,\cdots, \bs{Z}^t)$, the notations $\bs{Z}^t$ and $\vec{Z}^t_i$. Rows $Z_i=(Z_i^t,\, t\ge 1)$ are independent. The correlations within each row are described by the DE equations: $\vec{Z}_i^t\sim \cN_t(0,R_i^t)$, see Definition \ref{def:de}. }
\label{fig:notations-Z}
\end{figure}

\subsection{Main result (informal)}
As already mentioned, numerous studies \cite{Bayati_2015, pak-jus-krz-23, HACHEM2024104276, gueddari2024elliptic} have extended the AMP algorithm to cover more complex random matrix models $W$. For each new matrix model, two key questions must be addressed: 
\begin{enumerate}[label=\alph*)]
    \item How to define a proper Onsager term?
    \item What are the associated DE equations ?
\end{enumerate}
In this paper, we answer both questions for the matrix model described in Section~\ref{subsec:MatrixModel}. We show that the DE equations are given by Definition \ref{def:de}; in particular they only depend on the variance profile and not on the correlation profile. 
Let $W$ be given by \eqref{def:W}, $h:\RR^2\times \NN\to \RR$ an activation function, $\bs{x}^0,\bs{\eta}\in \RR^n$ deterministic vectors and 
$$
V = \left(S\odot S^\top\right)^{\odot 1/2} \odot T\,,$$
where $S$ and $T$ are respectively the variance and correlation profiles of the random matrix $W$, and $(\bs{Z}_1,\cdots, \bs{Z}_t)$ be given by the DE equations. We identify a possible Onsager term as
$$
\textrm{ONS}_t=  \diag\left(V \EE \frac{\partial h}{\partial x} (\bs{Z}^t,\bs{\eta},t)\right)h(\bs{x}^{t-1},\eta, t-1)\, ,
$$
and consider the AMP 
$$
\bs{x}^{t+1} = W h(\bs{x}^t, \bs{\eta}, t)  - \text{ONS}_t\,.
$$
We shall prove that for any appropriate test function $\varphi:\RR^{t+1}\to \RR$ and uniformly bounded sequence $(\beta^{(n)}_i)_{i\in [n]}$ of real numbers, the following convergence holds true
$$
\frac 1n \sum_{i=1}^n \left\{ \beta^{(n)}_i \varphi(\eta_i,x_i^1,\cdots, x_i^t) - \beta^{(n)}_i \varphi(\eta_i,\vec{Z}_i^t)\right\}
\xrightarrow[n\to\infty]{\mathbb{P}} 0\, ,
$$
where the $\vec{Z}_i^t$'s are defined in Definition \ref{def:de}. The formal assumptions and statement are provided in Section \ref{sec:main-assumption+result}.

\begin{remark}
    As a consequence of the variance profile structure, each $t$-uple $(x^1_i,\cdots, x^t_i)$ needs to be compared to $\vec{Z}_i^t$ in the convergence above, a situation substantially more complex than in \eqref{eq:conv-AMP-GOE}. 
\end{remark}

\subsection{Motivation from theoretical ecology} The analysis of large
ecological networks (foodwebs) and complex systems has garnered significant
attention in recent years, with numerous studies leveraging tools from random
matrix theory 

\cite{allesina2015stability,Bunin,cure2023antagonistic}. In this perspective,
large Lotka-Volterra (LV) models \cite{akjouj2024complex} describe the dynamics
of the vector of the species abundances $\bx(s) = (x_i(s))_{i\in [n]}$ for
$s\in[0,\infty)$ in a series of coupled differential equations where the
interactions are encoded by a random matrix $A$ whose entries $A_{ij}$'s
represent the effect of species $j$ on species $i$.  The more complex the
matrix model $A$, the better the modeling of the network. 

In a series of articles
\cite{akjouj2024equilibria,HACHEM2024104276,gueddari2024elliptic}, AMP
algorithms were designed in this context to analyze the statistical properties
of the globally stable equilibrium $\bx^\star$ (when it exists) of the vector
$\bx(s)$, depending on the random matrix model (symmetric models in
\cite{akjouj2024equilibria,HACHEM2024104276}, elliptic model in
\cite{gueddari2024elliptic}). More specifically, let $\bs{z}\in \mathbb{R}^n$
be the solution of the fixed-point equation:  
$$
    \bs{z} = \left(A-I_n\right) \bs{z}^+ + \bs{1}_n\, ,\qquad \bs{z}^+=\bs{z} \vee 0\,,
$$
which can be shown to be unique under a condition on $A$ 
(see \cite{akjouj2024equilibria} for details), then the equilibrium $\bs{x}^\star$ is given by $\bs{x}^\star=\bs{z}^+$. Extracting statistical information from $\bs{x}^\star$ is a non-trivial task as the dependence of $\bs{x}^\star$ to $A$ is highly non-linear. 
However this task can be performed by designing a specific AMP algorithm. 

In a foodweb, the effect $j\to i$ of species $j$ on species $i$ is a priori different from the effect $i\to j$. Moreover, recent empirical evidence \cite{busiello2017explorability} has shown that in a foodweb of size $n$ a given species only interacts with a small number $K_n\ll n$ of other species. One may want to go one step further in modelling foodwebs, and for instance consider block structures with subpopulations with homogeneous statistical features \cite{clenet2024impact}. 

All these desirable features naturally motivate the study of non-Symmetric and possibly sparse random matrices, with variance and correlation profiles. Such a model is at the heart of the AMP developed in this article. 

In a forthcoming work, we intend to design improved matrix models for foodwebs and to analyze via AMP techniques the equilibria of associated large LV models.

\subsection{Outline of the article}
In Section \ref{sec:main-assumption+result} we formally state the assumptions and the main result of the article, namely Theorem \ref{thm:main}, together with examples, an extension to non-centered random matrices, and open questions. 
The remaining sections are devoted to the proof of the main result (see also Section \ref{subsec:outline} for a precise roadmap of the proof). In Section \ref{sec:polyActiv}, we state a matrix AMP for polynomial activation functions, see Theorem \ref{thm:amp-vec}. 
Section \ref{sec:combinatorics} is the heart of the proof of Theorem \ref{thm:amp-vec}. It is based on combinatorial techniques which build upon \cite{Bayati_2015} and \cite{HACHEM2024104276}. In Section \ref{sec:generalActiv} we generalize the previous AMP for more general functions, and relax the assumption that matrix $W$ should have null diagonal (an assumption made to handle the combinatorics in the proof of Theorem \ref{thm:amp-vec}). 

\subsection{Notations} Denote by $|{\mathcal S}|$ the cardinality of a set
${\mathcal S}$.  We often (but not systematically) use bold letters for vectors
$\bs{a}=(a_i)_{i\in [n]},\bs{b}=(b_j)_{j\in [k]}$, etc. If $\bs{a}=(a_\ell)\in
\RR^q$ and $\bs{m}=(m_\ell)\in \NN^q$ is a multi-index, we denote by
$\bs{a}^{\bs{m}}=\prod_{\ell\in [q]} a_{\ell}^{m_{\ell}}$. 

Denote by $\bs{1}_n$ (or $\bs{1}$ if the context is obvious) the $n\times 1$
vector of ones and by $\bs{1}_{n\times p}$ the matrix $\bs{1}_{n\times p}=
\bs{1}_n \bs{1}_p^\top$ where matrix $A^\top$ stands for the transpose of $A$.
For $\bs{a}\in \RR^n$, $\diag(\bs{a})$ stands for the $n\times n$ diagonal
matrix with diagonal elements the $a_i$'s.  If $\bs{a}\in \RR^n$ is a vector,
$\|\bs{a}\|$ stands for its Euclidian norm and $\|\bs{a}\|_n :=
\|\bs{a}\|/\sqrt{n}$ for its normalized Euclidian norm. If $A$ is a matrix, $\|A\|$
stands for its spectral norm. 
 
If $f:\RR\to \RR$ and $\bs{a}=(a_i)_{i\in [n]}$ a vector, denote by $f(\bs{a})=\left( f(a_i)\right)_{i\in [n]}$ with obvious generalizations $f(\bs{a}, \bs{b})=(f(a_i,b_i))$ for $\bs{a},\bs{b}\in \RR^n$. Let $f(x,y,t)$ a real function with $(x,y,t)\in \RR^2\times \NN$, denote by $\partial f = \frac{\partial f}{\partial x}$. Let $\bs{a}\in \RR^n$ and $I\subset [n]$, then
$
\langle \bs{a} \rangle_n =\frac 1n \sum_{i\in [n]} a_i
$ and 
$
\langle \bs{a} \rangle_I =\frac 1{|I|} \sum_{i\in I} a_i
$. The empirical measures $\mu^{\bs{a}}$ and $\mu^{\bs{a}^1,\cdots, \bs{a}^t}$ of vector $\bs{a}=(a_i)_{i\in [n]}$ 
and vectors $\bs{a}^1,\cdots, \bs{a}^t$ in $\RR^n$ stand for
$$
\mu^{\bs{a}}=\frac 1n \sum_{i\in [n]} \delta_{a_i}\qquad \textrm{and}\qquad 
\mu^{\bs{a}^1,\cdots, \bs{a}^t}=\frac 1n \sum_{i\in [n]} \delta_{(a_i^1,\cdots, a_i^t)} \,,
$$
where $\delta_x$ is the Dirac measure on $\RR$ and $\delta_{(x^1,\cdots, x^t)}$, the Dirac measure on $\RR^t$. Convergence in probability is denoted by $\xrightarrow[]{\mathbb{P}}$. 

\section{AMP for general non-Symmetric random matrices}
\label{sec:main-assumption+result}

Assumptions are introduced in Section \ref{subsec:ass}. The main result, Theorem \ref{thm:main}, is stated in Section \ref{subsec:main}. In Section \ref{subsec:examples}, we provide two examples, one focusing on the correlation profile, the second on a sparse variance profile. 
In Section \ref{sec:extension}, we extend the AMP result to a non centered random matrix model.
Finally, we provide in Section \ref{subsec:outline} a detailed outline of the proof of the main theorem.

\subsection{The general framework of the AMP recursions} 

Let $X$ be a $n\times n$ $T$-correlated matrix and $S$ a $n\times n$ matrix with non negative coefficients. Recall the definition of $W=S^{\odot 1/2}\odot X$ in Eq. \eqref{def:W} and define matrix $V$ as follows
\begin{equation}
    \label{def:V}
         V = \begin{pmatrix} V_{ij} \end{pmatrix}_{i,j=1}^n
        = \bigl(S \odot S^\top\bigr)^{\odot {1/2}} \odot T\,.
\end{equation}
Notice that  
    $ \EE\left[W\odot W^\top\right] = V$.
    
Let $h : \RR^2\times \NN \to \RR$ be a measurable function such that for all
$(\eta,t)\in \RR\times \NN$, the derivative $\partial h(\cdot, \eta, t)$ exists
almost everywhere\footnote{Notice that if $h$ is Lipschitz with respect to the
first variable, then it is differentiable almost everywhere by Rademacher's
theorem.}. We denote as $\partial h$ any measurable function that coincides
with this derivative almost everywhere. For $\bs{x}, \bs{\eta}\in \RR^n$ and
$t\in \NN$, denote $h(\bs{x}, \bs{\eta}, t) = \left(h(x_i, \eta_i,
t)\right)_{i\in [n]}$.

\begin{definition}
    \label{def:AMP-recursive-scheme}
Let $X$ be a $n\times n$ $T$-correlated matrix following Definition
\ref{def:corr_matrix}, $W$, $V$ given by \eqref{def:W}, \eqref{def:V}, and
$\bs{x}^0,\bs{\eta}\in \RR^n$. Let $h : \RR^2\times \NN \to \RR$ a measurable
function such that $\partial h$ exists. 
    Let $\bs{Z}^1,\cdots, \bs{Z}^t$ be $\RR^n$-valued Gaussian vectors defined in Def. \ref{def:de}.
    Define the $\RR^n$-valued random sequence $(\bs{x}^t)_{t\geq 1}$ recursively as follows,
    \begin{equation}
        \label{def:AMPZ}
        \begin{cases}
        \bs{x}^{1} &= W h(\bs{x}^0, \bs{\eta} , 0)\,,\\
        \bs{x}^{t+1} &= W h(\bs{x}^t, \bs{\eta} , t) - \diag\left(V \EE \partial h(\bs{Z}^{t}, \bs{\eta}, t)\right) h(\bs{x}^{t-1}, \bs{\eta}, t-1)\quad\textrm{for}\quad  t\ge 1\,.
        \end{cases}
    \end{equation}
The following notation will be used in the sequel:
    \begin{equation}
        \label{eq:AMP_notation}
        \left(\bs{x}^t\right)_{t\ge 1}  = \AMPZ\left(X, S, h, \bs{x}^0, \bs{\eta}\right)\,,\quad \bs{x}^0, \bs{\eta}\in \RR^n\,.
    \end{equation}

\end{definition}
\begin{remark}
    The parameter $\bs{\eta}\in\RR^n$ which is fixed once for all in the recursions can be seen as an extra degree of freedom in the design of the algorithm.
\end{remark}

\begin{remark}[versatility]
\label{rmk:elliptic}
    Definition \ref{def:AMP-recursive-scheme} generalizes many frameworks found in the literature.  
    \begin{enumerate}[label=\alph*)]
    \item For a symmetric matrix $X$ where $T = \bs{1}_{n\times n}$ and $S=\frac{\bs{1}_{n\times n}}n$, one gets the AMP in  \cite{Bayati_2015}.
    \item By taking a sparse symmetric matrix $S$, one recovers the AMP in \cite{HACHEM2024104276}. 
    \item The elliptic AMP studied in \cite{gueddari2024elliptic} is obtained by taking $S=\frac{\bs{1}_{n\times n}}n$ and $T = \rho \bs{1}_{n\times n}$ for $\rho \in [-1, 1]$. In the latter, the AMP recursion writes
    \begin{equation*}
        \label{eq:AMP_elliptic}
        \bs{x}^{t+1} = W\ h\left(\bs{x}^t, \bs{\eta}, t\right) - \rho
        \left\langle \partial h\left(\bs{x}^t, \bs{\eta},t\right)\right\rangle_n
        h\left(\bs{x}^{t-1},\bs{\eta},  t-1\right)\,.
    \end{equation*}
    One can notice that the Onsager term is slightly different. We will come back to this later in Section~\ref{subsec:onsagers}.
    \end{enumerate}
\end{remark}

\subsection{Assumptions}
\label{subsec:ass}
We present hereafter the assumptions that will be used in the sequel, some of which already appeared in \cite{HACHEM2024104276}.

\begin{assumption}[moments]
    \label{ass:X}
    Let $T= (\tau_{ij})_{1\leq i,j\leq n}$ be a symmetric matrix with $\tau_{ij}\in [-1, 1]$ and $X$ a random $T$-correlated matrix following Definition \ref{def:corr_matrix}. For every $k\ge 1$ there exists a positive real number $\Cmom(k)>0$ such that 
    for every $n\ge 1$ and all $i,j\in [n]$ 
        $$
        \left( \EE \left| X_{ij} \right|^k \right)^{1/k} \leq \Cmom(k)\, . 
        $$
\end{assumption}

\begin{assumption}[variance profile] \label{ass:S}
Let $(K_n)$ a sequence of positive integers diverging to $+\infty$ and satisfying $K_n \leq n$. 
The deterministic $n\times n$ matrix $S=(s_{ij})_{1\le i,j\le n}$ has non-negative elements and satisfies the following: there exist positive constants $C_{\text{card}}, C_S, c_S>0$ such that for every $n\ge 1$ and all $i,j\in [n]$, 
$$
 \left| \left\{ j \in [n] \, : \, s_{ij} > 0 \right\} \right|
                  \ \leq\ C_{\text{card}}\, K_n \ ,\qquad s_{ij} \ \leq\  \frac{C_S}{K_n}\qquad\textrm{and}\qquad  \sum_{\ell = 1}^n s_{i\ell} \ \geq\ c_S\, .
$$
\end{assumption}

The following technical assumption ensures that the spectral norm of the matrix $W$ is almost surely bounded by a constant as $n$ goes to infinity.

\begin{assumption}[lower bound on the sparsity level]\label{ass:nu}
Let A-\ref{ass:X} and A-\ref{ass:S} hold for the random matrix $X$ and the variance profile $S$, and consider associated $\Cmom$ and $(K_n)$. There exist positive real numbers $\nu, C> 0$ such that for every $k,n\geq 1$
$$
\Cmom(k) \leq C\, k^{\nu/2}\, \qquad \text{and} \qquad K_n \ge  C\,\log^{(\nu \vee 1)}(n)\, .
$$
\end{assumption}

\begin{remark}[on Assumption A-\ref{ass:nu}]
\begin{enumerate}[label=(\alph*)]
    \item The moment condition $\Cmom(k) \leq C\, k^{\nu/2}$ is standard. For example, it is fulfilled with $\nu=1$ for subGaussian entries.
    \item Assumptions A-\ref{ass:S} and A-\ref{ass:nu} describe the sparsity level one can expect for matrix $W$. The sequence $K_n$ is an upper bound of the number of non-vanishing elements of $W$ per row. It must be at least logarithmic in $n$ (up to the power $\nu\vee 1$) but can be much smaller than $n$. 
    \item As will appear later in Proposition \ref{prop:spectralNormBound}, the logarithmic lower bound on $K_n$ and the upper bound for the moments of $X$'s entries are technical conditions needed for bounding the spectral norm of the random matrix $W$.
\end{enumerate}
\end{remark}

We also consider initial conditions for the initial vector $\bs{x}^0$ and for the parameter vector $\bs{\eta}\in \RR^n$.

\begin{assumption}[initial and parameter vectors]
    \label{ass:initial_point}
    Let $\bs{x}^0=(x^0_i)\in \RR^n$, $\bs{\eta}=(\eta_i)\in \RR^n$ be deterministic vectors and consider the sequences $(\bs{x}^0)_n$ and $(\bs{\eta})_n$. There exist two compact sets $\cQ_{x}\subset \RR$ and $\cQ_{\eta}\subset \RR$ such that
    \begin{equation*}
  \{ x_i^0\,,\ i\in [n]\,,\ n\ge 1\}\ \subset\ \cQ_{x}\qquad \textrm{and}\qquad \{ \eta_i\,,\ i\in [n]\,,\ n\ge 1\}\ \subset\ \cQ_{\eta}\,.
    \end{equation*}
    
\end{assumption}

\begin{assumption}[Regularity of the activation functions]
    \label{ass:activation_function}
    Let $h:\RR^2 \times \NN \to \RR$ be a measurable function. For every $t \in \NN$, there exists a positive number $L$ such that for every $x,y,\eta\in \RR$,
    \begin{equation*}
        \left|h(x, \eta, t) - h(y, \eta, t)\right|\quad  \leq\quad  L \left|x-y\right|\,.
    \end{equation*}
    For every $t\in \NN$, there exists a continuous non-decreasing function $\kappa : \RR_+ \to \RR_+$ with $\kappa(0)=0$ and a compact set $\cQ_{\eta}\subset \RR$ such that for every $x\in \RR$ and $\eta, \eta'\in \cQ_{\eta}$, 
    \begin{equation*}
        \left|h(x, \eta, t) - h(x, \eta^\prime, t)\right|\leq \kappa\left(\left|\eta - \eta^\prime\right|\right)\left(1+|x|\right).
    \end{equation*}
\end{assumption}

\begin{assumption}[non degeneracy condition over $h$]
\label{ass:nonDegen}
         Let $h:\RR^2 \times \NN \to \RR$ be a measurable function. There exist
        two compact sets $\cQ_{x}\subset \RR$ and $\cQ_{\eta}\subset \RR$ with the following properties:
        \begin{enumerate}
            \item 
        There exists a constant $c>0$ such that 
\begin{equation*}
    \inf_{x\in \cQ_{x}, \eta\in \cQ_{\eta}} h^2(x, \eta, 0) \quad \geq\quad  c\,.
\end{equation*}
    \item For every $t\ge 1$, there exist two positive real numbers $c_h(t), D_h(t)>0$ such that 
    $$ \inf_{\eta \in \cQ_{\eta}} \int_{-D_h(t)}^{D_h(t)} h^2(x, \eta, t) dx \quad \geq\quad  c_h(t)\,. $$
    \end{enumerate}
\end{assumption}

There are tight links between the assumptions. In particular, the parameter $\nu$ of A-\ref{ass:nu} controls the moments bounds $(C_{mom}(k))$ given by A-\ref{ass:X} and the sparsity level $K_n$ given by A-\ref{ass:S}, the compact sets $\cQ_{\eta}$ and $\cQ_{x}$ of A-\ref{ass:activation_function} and A-\ref{ass:nonDegen} will be given by A-\ref{ass:initial_point}.

\subsection{Main result}
\label{subsec:main}
Recall the definition of a \textit{pseudo-Lipschitz} function. A function $f:\RR^d\to\RR$ is said to be pseudo-Lipschitz (PL) if there exists a constant $L$ such that for all $\bx, \by\in \RR^d$ the following inequality is satisfied:
\begin{equation*}
    \left|f(\bx) - f(\by)\right| \ \leq\  L \left\lVert\bx - \by\right\rVert \left(1+ \left\lVert\bx\right\rVert+\left\lVert\by\right\rVert\right)\, .
\end{equation*}
We are now in position to state our main result. 
\begin{theorem}
    \label{thm:main}
    Let Assumptions A-\ref{ass:X} to A-\ref{ass:nonDegen} hold true, with associated $\nu$, $\cQ_{\eta}$ and $\cQ_{x}$. Consider the AMP 
    $$
    (\bs{x}^t)_{t\ge 1} = \AMPZ\left(X, S, h, \bs{x}^0, \bs{\eta}\right)
    $$ 
    as defined in Definition~\ref{def:AMP-recursive-scheme}, and the sequence of $n$-dimensional Gaussian random vectors $\left(\bs{Z}^t\right)_{t\in \NN}$  defined by the DE equations in Definition~\ref{def:de}:
    $$
    (\bs{Z}^1, \cdots, \bs{Z}^t) \sim \DE(S,h, \bs{x}^0, \bs{\eta}, t)\,.
    $$ 
    Let $t\ge 1$ and $\bs{\beta} = (\beta^{(n)}_i)\in \RR^n$ uniformly bounded, i.e. $\sup_n \max_{i\in [n]} |\beta^{(n)}_i|< \infty$. For any pseudo-Lipschitz test function $\varphi :\RR^{t+1} \to \RR$, it holds that
    \begin{equation*}
        \frac{1}{n} \sum_{i\in [n]} \beta^{(n)}_i  \left\{ 
        \varphi\left(\eta_i, x_i^1, \cdots, x_i^t\right) - \EE\left[\varphi\left(\eta_i, Z_i^1, \cdots, Z_i^t\right)\right] \right\}
        \quad \toprobalong\quad  0\, .
    \end{equation*}
\end{theorem}

\subsection{Alternative Onsager terms}\label{subsec:onsagers} It might be convenient to consider alternative Onsager terms in the AMP recursion and replace the diagonal matrix $\diag(V \EE \partial h(\bs{Z}^t, \bs{\eta}, t))$ by one of the two following terms 
    \begin{equation}
        \label{eq:variants}
        \diag\left(V \partial h(\bs{x}^t, \bs{\eta}, t)\right)
        \quad \text{or} \quad \diag\left(W\odot W^\top \partial h(\bs{x}^t, \bs{\eta}, t)\right)\,.
    \end{equation}

Depending on the context, it might be convenient to consider one of these three Onsager terms. 
    
For example, the Onsager term built upon $\diag\left(W\odot W^\top \partial h(\bs{x}^t, \bs{\eta}, t)\right)$ is better suited for the combinatorial arguments developed in Section~\ref{sec:combinatorics} as it directly involves the entries of matrix $W$, and the loss with respect to the original recursion should be asymptotically negligible since $\EE (W\odot W^\top)=V$. The Onsager term built upon $\diag\left(V \partial h(\bs{x}^t, \bs{\eta}, t)\right)$ naturally appears in \cite{akjouj2024equilibria,gueddari2024elliptic}.

    In this perspective we introduce new notations. Denote by 
 \begin{equation}
        \label{def:AMPW}
        \left(\bs{x}^t\right)_{t\ge 1}  := \AMPW\left(X, S, h, \bs{x}^0, \bs{\eta}\right),
    \end{equation}
    the recursive procedure defined by 
    \begin{equation*}
    \begin{cases}
    \bs{x}^1&=W h(\bs{x}^0, \bs{\eta} , 0)\,,\\
    \bs{x}^{t+1} &= W h(\bs{x}^t, \bs{\eta} , t) - \diag\left(W\odot W^\top \partial h(\bs{x}^{t}, \bs{\eta}, t)\right) h(\bs{x}^{t-1}, \bs{\eta}, t-1)\quad\textrm{for}\quad  t\ge 1\,.
    \end{cases}
    \end{equation*}
Similarly, denote by
   \begin{equation}
   \label{def:AMP}
        \left(\bs{x}^t\right)_{t\ge 1}  := \AMP\left(X, S, h, \bs{x}^0, \bs{\eta}\right) \,,
    \end{equation}
    the recursive procedure defined by 
    \begin{equation*}
    \begin{cases}
    \bs{x}^1&=W h(\bs{x}^0, \bs{\eta} , 0)\,,\\
    \bs{x}^{t+1} &= W h(\bs{x}^t, \bs{\eta} , t) - \diag\left(V \partial h(\bs{x}^t, \bs{\eta}, t)\right) h(\bs{x}^{t-1}, \bs{\eta}, t-1)\quad\textrm{for}\quad  t\ge 1\,.\\
    \end{cases}
    \end{equation*}

We believe that none of these three Onsager terms should change the general
asymptotics of the AMP. However, a complete proof of this fact is not yet
established.

\subsection{Examples of AMP}
\label{subsec:examples}
We provide hereafter two examples of matrix models where we work out the specific AMP recursion and DE equations. 
Both matrix models are of practical interest, with applications in fields 
such as theoretical ecology, where random matrices represent species interaction matrices in large 
ecological systems (see \cite{akjouj2024complex}).

\subsubsection*{Blockwise correlated random matrix}

This example generalizes the elliptic matrix model characterized 
by a single correlation coefficient $\rho$. Here, the matrix is allowed to have 
different correlation coefficients for each block. Let $n=n_1+n_2$, $X$ a $n\times n$ matrix partitioned into four submatrices: $X^{(11)}$, $X^{(12)}$, $X^{(21)}$, and $X^{(22)}$, 
of respective sizes $n_1 \times n_1$, $n_1 \times n_2$, $n_2 \times n_1$, and $n_2 \times n_2$:
$$
X = \begin{pmatrix}
    X^{(11)} & X^{(12)} \\
    X^{(21)} & X^{(22)} \\
\end{pmatrix}.
$$
Let $X^{(11)}$ and $X^{(22)}$ be (independent) elliptic random matrices with correlation coefficient $\rho_1$, 
while each entry in $X^{(12)}$ is correlated with its symmetrically corresponding entry in $X^{(21)}$ 
with a coefficient $\rho_2$. All the entries of the random matrix $X$ have variance $1$ and satisfy A-\ref{ass:X}. Consider the 
normalized version of $X$,
$$
W = \frac{X}{\sqrt{n}} \, .
$$
With our previous formalism, this model corresponds to choosing $X$ as a $T$-correlated matrix
and $W = S\odot X$ where $S$ (variance profile) and $T$ (correlation profile) are defined by
$$ 
S = \frac{\bs{1}_{n\times n} }{n}  \qquad \text{and} \qquad
T = \begin{pmatrix}
    \rho_1 \bs{1}_{n_1\times n_1} & \rho_2 \bs{1}_{n_1\times n_2} \\
    \rho_2 \bs{1}_{n_2\times n_1} & \rho_1 \bs{1}_{n_2\times n_2} \\
\end{pmatrix}.
$$
Let $r_n := \frac{n_1}n$, $I_1= \{1, \cdots, n_1\}$ and $I_2 = [n]\setminus I_1$, assume that $r_n\to r \in (0,1)$ and consider the following framework:
$
\bs{x}^0=x_0 \bs{1}_n$, the activation function $f:\mathbb{R}\to \mathbb{R}$ is Lipschitz.
Notice that $f$ satisfies A-\ref{ass:activation_function}, neither depends on $t$ nor on some extra parameter $\bs{\eta}$.  

Consider the recursion 
$(\bs{x}^t)_{t\in\NN} = \AMP\left(X, S, f, \bs{x}^0\right)$. In particular,
$$
\bs{x}^{t+1} = W f(\bs{x}^t) - \diag\left(Vf'(\bs{x}^t)\right) f(\bs{x}^{t-1}),
$$
where $V = T/n$. The Onsager term can be simplified here by writing $Vf'(\bs{x}^t)$ as
\begin{eqnarray*}
Vf'(\bs{x}^t) &=& \begin{pmatrix}
    r_n\rho_1 \langle f'(\bs{x}^t)\rangle_{I_1}\bs{1}_{n_1} + (1-r_n)\rho_2 \langle f'(\bs{x}^t)\rangle_{I_2}  \bs{1}_{n_1} \\
    r_n\rho_2 \langle f'(\bs{x}^t)\rangle_{I_1}\bs{1}_{n_2} + (1-r_n)\rho_1 \langle f'(\bs{x}^t)\rangle_{I_2}  \bs{1}_{n_2} \\
\end{pmatrix}\\
&=& \begin{pmatrix}
    r_n \rho_1 \bs{1}_{n_1} & (1-r_n)\rho_2 \bs{1}_{n_1}\\
    r_n \rho_2 \bs{1}_{n_2} & (1-r_n)\rho_1 \bs{1}_{n_2}
\end{pmatrix} \begin{pmatrix}
    \langle f'(\bs{x}^t)\rangle_{I_1} \\
    \langle f'(\bs{x}^t)\rangle_{I_2}
\end{pmatrix}.
\end{eqnarray*}
Thus 
$$
\bs{x}^{t+1} = W f(\bs{x}^t) - \left[\begin{pmatrix}
    r_n \rho_1 \bs{1}_{n_1} & (1-r_n)\rho_2 \bs{1}_{n_1}\\
    r_n \rho_2 \bs{1}_{n_2} & (1-r_n)\rho_1 \bs{1}_{n_2}
\end{pmatrix} \begin{pmatrix}
    \langle f'(\bs{x}^t)\rangle_{I_1} \\
    \langle f'(\bs{x}^t)\rangle_{I_2}
\end{pmatrix}\right]\odot f(\bs{x}^{t-1}),
$$
Notice that the Onsager term generalizes here the one obtained in the elliptic case 
(see Remark~\ref{rmk:elliptic}).

Not surprisingly (and as mentioned in \cite{gueddari2024elliptic} in the elliptic case), the DE equations do not depend on the correlation structure of $X$ and reduce to
$$ R^1 = f^2(x_0) \in \RR^{1\times 1}\,, \qquad R^{t+1} = \mathbb{E}\begin{bmatrix}
    f(x_0) \\
    f(Z_1) \\
    \cdots \\
    f(Z_t)
\end{bmatrix}\begin{bmatrix}
    f(x_0) &
    f(Z_1) &
    \cdots &
    f(Z_t)
\end{bmatrix} \in \RR^{(t+1)\times (t+1)},$$
where $(Z_1, \cdots, Z_t)\sim \cN_t(0, R^t).$ In this case, Theorem \ref{thm:main} implies that for any PL test function $\varphi:\RR^t\in \RR$ ur main theorem implies in this case that 
$$ \frac 1n \sum_{i\in [n]} \varphi(x^1_i,\cdots, x^t_i)  \toprobalong \EE \varphi(Z_1, \cdots, Z_t)\,. $$

\begin{remark}
    This example can easily be generalized to $K\times K$ blocks and $K$
correlation coefficients $\rho_1, \cdots, \rho_K$.
\end{remark}

\subsubsection{$d$-regular random matrix}

In this example, we consider a symmetric matrix $X$ where $X_{ij}$ are independent centered random
variables with variance $1$ up to the symmetry, i.e. $X$ is a $T$-correlated random matrix where 
$T=\bs{1}_{n\times n}$. Let Assumption \ref{ass:X} hold, let $d = d_n =\lfloor C\log^{(\nu\vee   1)}(n) \rfloor$ where $\nu>0$ is given by Assumption \ref{ass:nu}. Let $A$ be the $n\times n$ adjacency matrix of a $d$-regular non oriented graph, in particular
$$ 
\left|\{j \in [n] \ | \ A_{ij} = 1\}\right| = d \qquad \text{and} \quad \left|\{i \in [n] \ | \ A_{ij} = 1\}\right| = d\,,
$$
and consider the variance profile matrix $S= \frac{1}{d} A$. Let $f:\mathbb{R}\to \mathbb{R}$ a Lipschitz function (hence satisfying Assumption~\ref{ass:activation_function}) and set 
$$
W = S\odot X = \frac{1}{d} A\odot X\,,\qquad \bs{x}^0 = x_0 \bs{1}_n\qquad\textrm{and}\qquad  (\bs{x}^t)_{t\in\NN} = \AMP\left(X, S, f, \bs{x}^0\right)\,.
$$
Introducing the sets $I_k := \{j \in [n] \ | \ A_{kj} = 1 \}$ and the $n\times 1$ vector 
$\bs{v}=\left(\langle f'(\bs{x}^t)\rangle_{I_k}\,,\ k\in [n]\right)$, the recursion writes 
$$ 
\bs{x}^{t+1} = Wf(\bs{x}^t) - \bs{v} \odot f(\bs{x}^{t-1})\,. 
$$

Let us now simplify the Density Evolution equations 
defined \ref{def:de} for this particular case. We notice that $H_i^0 = (h(x_0))^2 =: H^0$ does not depend on
$i$, so $R^1_i = \sum_{j\in I_i} \frac{1}{d} H_i^0 = H^0 := R^1$ which is also independent of $i$ and $n$.
By induction, we can reduce DE equations to ``asymptotic'' DE equations, meaning that they do not
depend on $n$. In fact, if $R_i^t \in \RR^{t\times t}$ is independent of $i$, consider 
$(Z_i^1, \cdots, Z_i^t)\sim \cN_t(0, R_i^t)$, these $n$ $t$-dimensional random vectors have the same law. Now let
$i\in [n]$ and consider the value of $R^{t+1}_i$,
$$ R^{t+1}_i = \sum_{j\in I_i} \frac{1}{d} \EE \begin{bmatrix}
    f(x_0) \\
    f(Z_i^1)\\
    \cdots \\
    f(Z_i^t)
\end{bmatrix} \begin{bmatrix}
    f(x_0) &
    f(Z_i^1)&
    \cdots &
    f(Z_i^t)
\end{bmatrix} = \EE \begin{bmatrix}
    f(x_0) \\
    f(Z_1)\\
    \cdots \\
    f(Z_t)
\end{bmatrix} \begin{bmatrix}
    f(x_0) &
    f(Z_1)&
    \cdots &
    f(Z_t)
\end{bmatrix}$$
where $(Z_1, \cdots, Z_t) \sim \cN_t(0,R^t)$, thus $R_{i}^{t+1}$ is also independent of $i$ and $n$ and 
we recover the ``asymptotic" DE equations. Our main theorem implies in this case that 
$$ \mu^{\bs{x}^1, \cdots, \bs{x}^t} \toprobalong \mathcal{L}(Z_1, \cdots, Z_t).$$

\subsection{Extension to non-centered random matrices}\label{sec:extension}

We have considered so far an AMP algorithm with a centered random matrix. We extend our AMP result to consider a non-centered matrix model. More precisely, we add to our centered random matrix model a deterministic rank-one perturbation - notice that our result could easily be generalized to any finite-rank perturbation.

Let $W$ be a random matrix model as in Theorem~\ref{thm:main}, with variance profile $S$ and correlation profile $T$. Let $\bu,\bv\in \RR^n$ two deterministic vectors satisfying $\lVert \bu\rVert, \lVert \bv\rVert = \mathcal{O}(n^{-1})$. Consider the following matrix model,
\begin{equation}
    A = \lambda \bu \bv^\top + W.
\end{equation}
Before stating the AMP recursion based on matrix $A$, we adapt the Density Evolution equations introduced in Definition~\ref{def:de}.
In this section, we shall use the notation $h_t(x,\eta) $ instead of $ h(x,\eta, t)$ as simplification of the notations.

\begin{definition}
    \label{def:noncenteredde}
    Let $\bs{x}^0=(x^0_i)\in \RR^n$, $\bs{\eta}=(\eta_i) \in \RR^n$, $\bs{\bu}=(u_i) \in \RR^n$ and $\bs{\bv}=(v_i) \in \RR^n$ be deterministic vectors, $S = \left(s_{ij}\right)_{1\leq i,j \leq n}$ a matrix with non-negative elements and $h:\RR^2 \times \NN \to \RR$ an activation function.
    \begin{enumerate}[label=\alph*)]
    \item \emph{Initialization.} For any $i\in [n]$, define the positive numbers $H_i^0$, $R_i^1$ and $\mu_1$ as
    \begin{equation*}
        H_i^0 := \left(h_0(x_i^0, \eta_i)\right)^2,\qquad R_i^1 := \sum_{j = 1}^{n} s_{ij}H_j^0\qquad \text{and} \qquad \mu_1:= \lambda \left\langle \bv, h_0(\bx^0, \bs{\eta})\right\rangle\,.
    \end{equation*}
    Let $Z_i^1\sim\cN(0,R^1_i)$, assume that for all $i\in [n]$, the $Z_i^1$'s are independent and set 
    $$\bs{Z}^1=(Z^1_i)_{i\in [n]}\ .$$ 

    \item \emph{Step 1.} Let $i\in[n]$ be fixed. Given $Z^1_i$, let 
    \begin{eqnarray*}
    H_i^1&=&\EE \begin{bmatrix}
            h_0\left(x^{0}_i, \eta_i\right) \\
            h_1\left(Z^{1}_i + \mu_0 u_i, \eta_i\right) 
            \end{bmatrix}\begin{bmatrix}
            h_0\left(x^{0}_i, \eta_i\right) &
            h_1\left(Z^{1}_i + \mu_0 u_i, \eta_i\right) 
            \end{bmatrix}\ ,\\
            R^2_i&=&\sum_{j=1}^n s_{ij} H_j^1\quad \text{and} \quad \mu_2 = \lambda \EE\left[\left\langle \bv, h_1\left(\bZ^{1} + \mu_1 \bu, \bs{\eta}\right)\right\rangle\right]\,.
    \end{eqnarray*}
Let $(Z_i^1, Z_i^2)\sim\cN_2(0,R^2_i)$, denote by $\vec{Z}_i^2=(Z_i^1,Z_i^2)$. 
Assume that for all $i\in [n]$, the $\vec{Z}_i^2$'s are independent. Set $\bs{Z}^2=(Z^2_i)$. \\
    \item \emph{Step t.} Let $i\in [n]$ be fixed. Given $(\bs{Z}^1,\cdots, \bs{Z}^t)$ and $\vec{Z}^t_i=(Z_i^1,\cdots, Z_i^t)$, let      
     $$
        H_i^t =  \EE \begin{bmatrix}
            h_0(x^{0}_i, \eta_i) \\
            h_1(Z^{1}_i + \mu_1 u_i, \eta_i) \\
            \vdots                \\
            h_t(Z^{t}_i + \mu_t u_i, \eta_i)\end{bmatrix}\begin{bmatrix}
            h_0(x^{0}_i, \eta_i) &
            h_1(Z^{1}_i + \mu_1 u_i, \eta_i) &
            \cdots                &
            h_t(Z^{t}_i + \mu_t u_i, \eta_i)\end{bmatrix}\, .\\
     $$
    Denote
    $$
    R_i^{t+1} =  \sum_{j=1}^n s_{ij}H_j^t \qquad \textrm{and}\qquad  \mu_{t+1} = \lambda \EE\left[\left\langle \bv, h_t\left(\bZ^{t} + \mu_t \bu, \bs{\eta}\right)\right\rangle\right]\,.
    $$
     Let $(Z_i^1, Z_i^2,\cdots Z_i^{t+1})\sim\cN_{t+1}(0,R^{t+1}_i)$, denote by $\vec{Z}_i^{t+1}=(Z_i^1,Z_i^2,\cdots, Z_i^{t+1})$. Assume that for all $i\in [n]$, the $\vec{Z}_i^{t+1}$'s are independent. Set $\bs{Z}^{t+1}=(Z^{t+1}_i)$. \\
    \end{enumerate}
    Consider the sequence of $n$-dimensional Gaussian random vectors $\left(\bs{Z}^{t}\right)_{t\in \mathbb{N}}$. We denote 
    $$
    \left(\bs{Z}^1, \cdots, \bs{Z}^t\right) \sim \widetilde{\DE}\left(h, \bs{x}^0, S, t,\bu, \bv\right)\, .
    $$

\end{definition}
We are now in position to state the AMP recursion. 
\begin{equation}
\label{eq:ampnoncentered}
   \bs{x}^{t+1} = A h_t(\bs{x}^t, \bs{\eta}) - \diag\left(V \EE \partial h_t(\bs{Z}^{t} + \mu_t \bu, \bs{\eta})\right) h_{t-1}(\bs{x}^{t-1}, \bs{\eta}),
\end{equation}
where $\bs{Z}^t$ and $\mu^t$ are defined as in Definition~\ref{def:noncenteredde}.

The following theorem describes the asymptotic behavior of $\left(\bs{x}^t\right)_{t\in\NN}$ when $n$ goes to infinity.
\begin{theorem}
\label{th:noncentered}
    Let Assumptions A-\ref{ass:X} to A-\ref{ass:nonDegen} hold true, with associated $\nu$, $\cQ_{\eta}$ and $\cQ_{x}$.
    Consider the AMP sequence $(\bs{x}^t)_t$ defined in \eqref{eq:ampnoncentered}.
    Consider the sequence $n$-dimensional Gaussian random vectors $\left(\bs{Z}^t\right)_{t\in \NN}$ and the scalars $(\mu_t)_t$ defined by the DE equations in Definition~\ref{def:noncenteredde}.
    
    Let $t\ge 1$ and $\bs{\beta} = (\beta^{(n)}_i)\in \RR^n$ uniformly bounded, i.e. $\sup_n \max_{i\in [n]} |\beta^{(n)}_i|< \infty$. For any pseudo-Lipschitz test function $\varphi :\RR^{t+1} \to \RR$, it holds that
    \begin{equation*}
        \frac{1}{n} \sum_{i\in [n]} \beta^{(n)}_i  \left\{ 
        \varphi\left(\eta_i, x_i^1, \cdots, x_i^t\right) - \EE\left[\varphi\left(\eta_i, Z_i^1+\mu_1 u_i, \cdots, Z_i^t+\mu_t u_i\right)\right] \right\}
        \quad \toprobalong\quad  0\, .
    \end{equation*}
\end{theorem}
This theorem can be seen as a corollary to Theorem~\ref{thm:main}, the proof is provided in Appendix \ref{app:noncenteredproof}.

\subsection{Open questions}  
\begin{enumerate} 
    \item Currently, the sparsity level is of order $\log^{\nu\vee 1}(n)$. Would it be possible to lower this level, and to dissociate the sparsity assumption from the parameter $\nu$ which is associated to the moments of the matrix entries?
    \item Would it be possible to improve the convergence in probability in Theorem \ref{thm:main} to an almost sure convergence?
    \item Our current assumptions over the entries of the matrix necessitate all the moments. Would it be possible by truncation techniques to lower this assumption?
    \item Would it be possible to establish the counterpart of Theorem \ref{thm:main} for AMP schemes \eqref{def:AMPW} or \eqref{def:AMP}?
\end{enumerate}

\subsection{Outline of the proof}
\label{subsec:outline}
Building on the methods developed in \cite{Bayati_2015} and \cite{HACHEM2024104276}, we start by analyzing a particular case of the Approximate Message Passing (AMP) algorithm with polynomial activation functions (Section~\ref{subsec:polyActivScalar}), which motivates the adoption of combinatorial techniques. In our setting, the variance profile is non-symmetric, and the matrix contains correlations between symmetric entries, necessitating modifications to the combinatorial approaches used in both \cite{Bayati_2015} and \cite{HACHEM2024104276} to fit our case. The combinatorial heart of the proof is presented in Section \ref{sec:combinatorics}. We then use density arguments to extend the results to non-polynomial activation functions that exhibit at most polynomial growth (Section~\ref{subsec:generalActiv}). 

It should be noted that the combinatorial methods in \cite{Bayati_2015} and \cite{HACHEM2024104276} rely on the assumption of a zero-diagonal variance profile, i.e., $S_{ii} = 0$ for all $i \in [n]$, which simplifies the derivations. We adopt this assumption in Sections \ref{subsec:polyActivScalar}, \ref{subsec:polyActivMatrix} and \ref{subsec:generalActiv} and then lift it via a perturbation argument in Section~\ref{subsec:nonzeroDiag}. Unless otherwise specified, we assume that the matrix $S$ has a zero-diagonal, implying, without loss of generality, that the random matrix $X$ also has a zero diagonal $X_{ii} = 0$.

\begin{figure}[h!]
\centering
\begin{tikzpicture}[
    node distance=2cm and 2cm,
    every node/.style={draw, rectangle, rounded corners, text width=4cm, align=left, minimum height=1cm}
]

\node (block4)  {\textbf{Matrix AMP-W}\\
Polynomial activation and test functions.\\
Combinatorial arguments.\\
Zero-diagonal assumption (A-\ref{ass:diag-nulle}).\\\textbf{[Section~\ref{subsec:polyActivMatrix}]}.
};
\node (block5) [right=of block4] {\textbf{Polynomial AMP-W}\\
Polynomial activation functions.\\
Zero-diagonal assumption (A-\ref{ass:diag-nulle}).\\\textbf{[Section~\ref{subsec:polyActivScalar}]}.};
\node (block6) [below=of block5] {\textbf{AMP-Z}\\
General activation functions approximation by polynomials.\\
Full diagonal - we lift (A-\ref{ass:diag-nulle}).\\
\textbf{[Section \ref{subsec:generalActiv}]}.};
\node (block7) [left=of block6] {\textbf{AMP-Z}\\
Main theorem. \\ \textbf{[Theorem~\ref{thm:main}]}.};

\node (block1) [above left= 1cm and -1cm of block5] {From polynomial to general test functions with at most polynomial growth.\\ \textbf{[Lemma~\ref{lem:polyActivation}]}.};
\node (block2) [below right=0cm and 0.5cm of block5] {Diagonal perturbation technique to lift (A-\ref{ass:diag-nulle}).\\ \textbf{[Section~\ref{subsec:nonzeroDiag}]}.};

\draw[->, very thick] (block4) -- (block5);
\draw[->, very thick] (block5) -- (block6);
\draw[->, very thick] (block6) -- (block7);

\draw[->, thick] (block1) -- (block5);
\draw[->, thick] (block2) -- (block6);

\end{tikzpicture}
\caption{Proof steps.}
\label{fig:proof-graph}
\end{figure}
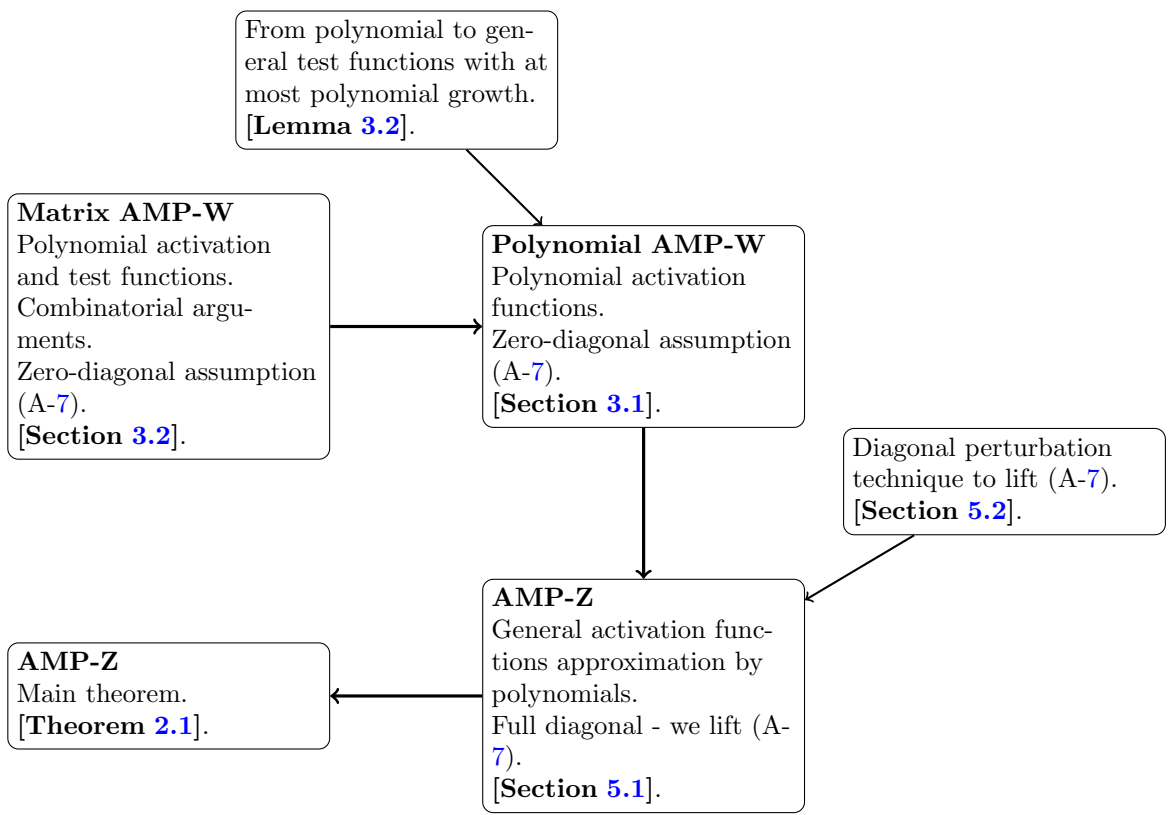

\section{AMP and Matrix AMP for polynomial activation functions}
\label{sec:polyActiv}

We present hereafter the AMP algorithm for polynomial activation functions, a suitable framework to establish the proof by combinatorial techniques, see \cite{Bayati_2015,HACHEM2024104276}. In Section \ref{subsec:polyActivScalar}, we state Theorem \ref{thm:amp-pscal} for iterates that are $\RR^n$-valued.
 
In Section \ref{subsec:polyActivMatrix}, we state a result for iterates that are $\RR^{n\times q}$-valued, a more general result that will imply Theorem \ref{thm:amp-pscal}. The extension to general pseudo-Lipschitz functions will be performed in Section~\ref{subsec:generalActiv}. 

The following technical assumption (to be lifted in Section \ref{subsec:nonzeroDiag}) will be used hereafter.

\begin{assumption}[variance profile with vanishing diagonal]\label{ass:diag-nulle}
    The deterministic $n\times n$ matrix $S=(s_{ij})_{1\le i,j\le n}$ has non-negative elements with null elements on the diagonal:
    $$
    S_{ii}=0 \qquad \textrm{for}\quad i\in [n]\, .
    $$
\end{assumption}

\begin{remark}
Assumption A-\ref{ass:diag-nulle} is very convenient to establish the statistical properties of the AMP iterates for polynomial activation functions, as the proof relies on combinatorial techniques. The fact that the diagonal of the variance profile $S$ is zero substantially simplifies the combinatorics. This assumption is relaxed in Theorem \ref{thm:main} by means of perturbation arguments (see Section~\ref{subsec:nonzeroDiag}).
\end{remark}

\subsection{AMP for polynomial activation functions}
\label{subsec:polyActivScalar}

Let $d\ge 1$ be a fixed positive integer independent from $n$. For every integer $t\ge 1$, consider a uniformly bounded triangular array of real coefficients 
\begin{equation}
        \label{bnd-ag}
\Big( \alpha_{\ell}(i, t, n)\,,\ \ell\le d\,,\ i\in [n]\,,\ n\ge 1\Big)\qquad \textrm{with}\qquad 
\sup_n \max_{\ell\leq d} \max_{i\in[n]} \left| \alpha_{\ell}(i,t,n) \right| < \infty\,.
\end{equation}
The following function will play a key role in the sequel:
\begin{eqnarray}
p:\mathbb{R} \times [n] \times \mathbb{N}&\to& \RR\,, \label{eq:polynomials}\\
(u,i,t)\quad &\mapsto & p(u,i,t) =\sum_{\ell = 1}^{d} \alpha_\ell (i,t,n) u^\ell\,.  \nonumber
\end{eqnarray}
Function $p$ is a polynomial in $u$ with degree bounded by $d$. It depends on $n$ via the coefficients $\alpha_\ell (i,t,n)$. To lighten the notations, we drop the dependence of $\alpha_\ell(i,t,n)$ in $n$ and simply write $\alpha_\ell(i,t)$ and do not indicate the dependence of $p$ in $n$.

Following Definition \eqref{def:AMPW}, let $\bs{\cx}^{0} \in \mathbb{R}^n$ be deterministic and define 
\begin{equation*}
    \left(\bs{\cx}^t\right)_{t\ge 1} = \AMPW\left(X, S,p, \bs{\cx}^0 \right)\,,\quad \bs{\cx}^0\in \RR^n\, ,
\end{equation*}
that is 
\begin{equation}
    \label{amp-posc}
    \bs{\cx}^{t+1} =
    W p(\bs{\cx}^{t},\cdot,t) -
    \diag \Bigl(W\odot W^\top \partial p(\bs{\cx}^{t},\cdot,t) \Bigr)
    p(\bs{\cx}^{t-1},\cdot,t-1) ,
\end{equation}
where $p(\bs{x},\cdot,t) = \left[p(x_i,i,t)\right]_{i=1}^n$ and $\partial p(\bs{x},\cdot,t)=  \left[\partial p(x_i,i,t)\right]_{i=1}^n$ for any $\bs{x}\in\mathbb{R}^n$.

We now present the AMP result for polynomial activation
functions. 
\begin{theorem}
    \label{thm:amp-pscal}
    
    Let A-\ref{ass:X}, A-\ref{ass:S} and A-\ref{ass:diag-nulle} hold true. Let $d\ge 1$ be fixed, $(\alpha_{\ell})$ and $p$ given by \eqref{bnd-ag} and \eqref{eq:polynomials}. Let $\bs{\cx}^0=(\check x^0_i)\in \RR^n$. Assume that there exists a compact set $\cQ_{\bs{\cx}}\subset \RR$ such that $\cx^0_i\in \cQ_{\bs{\cx}}$. Consider
    \begin{equation*}
    \left(\bs{\cx}^t\right)_{t\ge 1} = \AMPW\left(X, S,p, \bs{\cx}^0 \right)\,.
\end{equation*}
    Let $(\bs{\cZ}^1, \cdots, \bs{\cZ}^t) \sim \DE(S,p, \bs{\cx}^0, t)$ and denote by
    $\chR_i^t$ the covariance matrix of vector $\left(\cZ^1_i, \cdots, \cZ^{t}_i\right)$.
    Then for all $t,m\ge 1$
    \begin{equation}
        \label{mom-cx}
       \sup_n \max_{i\in[n]} \| \chR^{t}_i \| < \infty \qquad \textrm{and} \qquad
        \sup_n \max_{i\in[n]} \EE |\cx^{t}_i|^m < \infty\,.
    \end{equation}
    Given $t\ge 1$, let $d'\ge 1$ be fixed and consider function $\psi_n: \RR^t \times [n] \to \RR$, a multivariate polynomial with bounded degree: 
    $$
   \psi_n(x_1,\cdots, x_t, \ell) = \sum_{d_1+\cdots +d_t\le d'} \beta_n(d_1,\cdots, d_t,\ell) \prod_{i\in[t]}x_i^{d_i}\,,
    $$
    with $$
    \sup_{n\ge 1}\sup_{\ell\in[n]}\sup_{d_1+\cdots+d_t\le d'} \left| \beta_n(d_1,\cdots, d_t, \ell)\right| <\infty\, .
    $$ 
    Let $\cS^{(n)} \subset [n]$ be such that $| \cS^{(n)} | \leq C K_n$ where $K_n$ is given by A-\ref{ass:S}. Then,
    \begin{subequations}
        \label{cvg-pscal}
        \begin{align}
             & \frac{1}{K_n} \sum_{i\in \cS^{(n)}}
            \left\{ \psi_n(\cx^{1}_i,\ldots,\cx^{t}_i,i) -
            \EE\psi_n(\cZ^{1}_i, \ldots, \cZ^{t}_i,i) \right\} \toprobalong 0 \,,
            \quad \text{and} \label{cvg-small}             \\
             & \frac{1}{n} \sum_{i\in[n]}
            \left\{ \psi_n(\cx^{1}_i,\ldots,\cx^{t}_i,i) -
            \EE\psi_n(\cZ^{1}_i, \ldots, \cZ^{t}_i,i) \right\} \toprobalong 0 \,.\label{cvg-big} 
        \end{align}
    \end{subequations}
\end{theorem}

\begin{remark}
    In this theorem, both the activation function and the test function used in the convergence formulation are polynomials. The general case for the activation function will be addressed later in Section~\ref{subsec:generalActiv}. Regarding the test functions, we extend this result in the following lemma to encompass general continuous functions that grow at most polynomially near infinity. Notice also that Assumption A-\ref{ass:nu} is not needed when dealing with AMP sequences having polynomial activation functions, this assumption is purely technical and is used when a comparison between two AMP sequences is provided.
\end{remark}

\begin{remark}
    The interesting regime in \eqref{cvg-small} is $|\cS^{(n)}|\sim K_n$. If $|\cS^{(n)}|\ll K_n$ then \eqref{cvg-small} is trivial in the sense that one can easily prove that both terms  
    $$
    \frac{1}{K_n} \sum_{i\in\cS^{(n)}}   
        \psi_n\left(\cx_i^1, \cdots, \cx_i^t, i\right) \qquad \textrm{and}\qquad 
        \frac{1}{K_n} \sum_{i\in\cS^{(n)}}    \EE\left[\psi_n\left(\cZ_i^1, \cdots, \cZ_i^t, i\right)\right] 
    $$
    converge to zero\footnote{By $|\cS^{(n)}|\sim K_n$, we mean that there exist $c,C>0$ such that 
    $cK_n\le |\cS^{(n)}| \le C K_n$ and by $|\cS^{(n)}|\ll K_n$, we mean that $|\cS^{(n)}|/ K_n\to 0$.}.
\end{remark}
\begin{lemma}
    \label{lem:polyActivation}
    Let $\bs{\check x}^0$ and $\bs{\eta}$ satisfy A-\ref{ass:initial_point}. Let $(\bs{\cx}^t)_{t\in\NN}$ and $(\bs{\cZ}^t)_{t\in\NN}$ as in Theorem~\ref{thm:amp-pscal}. Let $t,m\ge 0$ be fixed integers and let $\varphi : \cQ_{\eta} \times \RR^t \to \RR$ be a continuous function such that $$
    \left|\varphi(\alpha, u_1, \cdots, u_t)\right|\leq C\left(1 + |u_1|^m + \cdots + |u_t|^m\right)\,.
    $$ 
    For any sequence $(\beta^{(n)}_i\in \RR\,,\ i\in [n]\,,\ n\ge 1)$ such that 
    $
    \sup_n \max_{i\in [n]} |\beta^{(n)}_i|< \infty
    $, the following convergence holds:
    $$\frac{1}{n} \sum_{i\in[n]}
            \beta^{(n)}_i\varphi(\eta_i,\cx^{1}_i,\ldots,\cx^{t}_i) -\frac{1}{n} \sum_{i\in[n]}
            \beta^{(n)}_i\EE\varphi(\eta_i, \cZ^{1}_i, \ldots, \cZ^{t}_i) \toprobalong 0.$$
\end{lemma}
\begin{proof}
    Define the two $t+2$ dimensional random measures $\mu_n$ and $\nu_n$ as follows

    \begin{equation*}
        \mu_n = \frac{1}{n}\sum_{i\in [n]} \delta_{(\beta_i, \eta_i, \cx_i^1, \cdots, \cx_i^t)} \quad \text{and} \quad \nu_n = \mathcal{L}\left(\beta_\theta, \eta_\theta, \cZ_\theta^1, \cdots, \cZ^t_\theta\right),
    \end{equation*}
    where $\theta\sim \mathcal{U}\left([n]\right)$ is independent. Consider the function $\psi(\beta, \eta, x_1, \cdots, x_t) = \beta \varphi\left(\eta, x_1, \cdots, x_t\right)$, and recall that $(\beta_i)$, $(\eta_i)$ and the covariance matrices $(R_i^t)$ are bounded, thus by some slight modification to Lemma~\ref{lem:momentsTogeneral} we get the desired result.
    
\end{proof}

\subsection{Matrix AMP for polynomial activation functions}
\label{subsec:polyActivMatrix}
In order to prove Theorem~\ref{thm:amp-pscal}, we  need to study a
matrix version of the AMP algorithm where the iterates $\bs{\cx}^{t}$
 are $\RR^{n\times q}$--valued matrices, $q\ge 1$ being a
fixed integer.
Using this framework, we  only need to express the
convergence result in Theorem~\ref{thm:amp-pscal} using test functions
acting only on the $t^{\text{th}}$ iterates instead of all previous
iterates. Consider the function
\begin{equation}\label{eq:f-poly}
f: \RR^q \times [n] \times \NN \longrightarrow \RR^q \,,\qquad f(\bu,l,t) = \begin{pmatrix} f_1(\bu,l,t) \\ \vdots \\ f_q(\bu,l,t)
        \end{pmatrix}\ ,
\end{equation}
where each component $f_r$ is a polynomial in $\bu \in \mathbb{R}^q$, with degree bounded by $d$, written as
\[
    f_r(\bu,\ell,t) = \sum_{\substack{\bs{i}=(i_1,\cdots, i_q)\in \NN^q\\i_1+\cdots+i_q \leq d }}
    \alpha_{\bs{i}}(r,\ell,t) \bs{u}^{\bs{i}}\,,
\]
(recall the notation $\bs{u}^{\bs{i}}=\prod_{s\in [q]}u_s^{i_s}$). Given a deterministic $n$-uple
$(\bcx^{0}_1,\ldots, \bcx^{0}_n)$ where $\bcx^{0}_i$ is a $q$-dimensional vector, the AMP iterates are recursively defined for $t\geq 1$ as follows:
\begin{equation}
    \label{ampol}
    x^{t+1}_{i}(r) =
    \sum_{\ell\in[n]} W_{i\ell} f_r(\bcx^{t}_\ell,\ell,t)
    - \sum_{s\in[q]} f_s(\bcx^{t-1}_i,i,t-1)
    \sum_{\ell\in[n]} W_{i\ell}W_{\ell i}
    \frac{\partial f_r}{\partial x(s)}
    (\bcx^{t}_\ell,\ell,t) \,,
\end{equation}
for $r\in [q]$ and $f(\cdot, \cdot, -1) \equiv 0$. We  denote such a sequence by
$$ \left(\bs{x}^t\right)_{t\ge 1} = \AMPW_{q}\left(X, S, f, \bs{x}^0\right)\,,\quad \bs{x}^0\in \RR^{n\times q}\,. $$

\subsubsection*{DE Equations for matrix AMP}
\label{subsubsec:MatrixDE}
Similarly to the DE equations for standard AMP introduced 
in Definition~\ref{def:de}, we introduce here a
$\left(\mathbb{R}^{q}\right)^n$-valued sequence
of Gaussian random vectors $(U^t)_{t\in\mathbb{N}^\star}$
defined by
\[
    U^{t} = \begin{bmatrix} (U^{t}_1)^\top \\ \vdots \\ (U^{t}_n)^\top
    \end{bmatrix},
\]
where $\{U_i^t\}_{i\in [n]}$ are $\mathbb{R}^q$-valued independent Gaussian
random vectors, $U^t_i\sim \cN\left(0,Q_i^t\right)$ and
the $q\times q$ matrices $Q_i^t$ are defined recursively in $t$ by
\begin{equation}\label{eq:def-Q-poly}
    Q^{t+1}_i =  \sum_{\ell\in[n]} s_{i\ell}
    \EE f(U^{t}_\ell, \ell, t) f( U^{t}_\ell, \ell, t)^\T
    \qquad \text{for} \quad i \in [n],
\end{equation}
with the convention that $U^0 := \bs{x}^0$. We  denote
\begin{equation}
\label{eq:MatrixDE}
   U^t\sim \DE_q\left(S,f, \bs{x}^0, t\right).
\end{equation}

The following Theorem
is the key component to the proof of Theorem~\ref{thm:amp-pscal}.
\begin{theorem}
    \label{thm:amp-vec}
    Let Assumptions A-\ref{ass:X} and A-\ref{ass:S} hold true and $q\ge 1$ be fixed. 
   Let $f$ be defined by \eqref{eq:f-poly} and $\bs{x}^0\in \RR^{n\times q}$. Assume that for each $t \ge 1$, there exists a
    constant $C = C(t) > 0$ such that
    \begin{equation}
        \label{alpha-vec}
        \left| \alpha_{i_1,\ldots,i_q}(r,l,t) \right| \leq C,
    \qquad \text{and}\qquad
        \sup_n \max_{i\in[n]} \left\| \bcx^{0}_{i} \right\|  < \infty\ .
    \end{equation} 
    Consider the
    iterative algorithm
    $ \left(\bs{x}^t\right)_{t\ge 1} = \AMPW_q\left(X, S, f, \bs{x}^0\right)$,
    and let $Q_i^t$ and $U^t$ be defined by \eqref{eq:def-Q-poly}--\eqref{eq:MatrixDE}. 
    Then we have,
    \begin{equation}
        \label{Qfini}
        \forall t > 0\,, \qquad
        \sup_n \max_{i\in[n]} \| Q^{t}_i \| < \infty .
    \end{equation}
    Moreover,
    \begin{equation}
        \label{Exbnd}
        \forall t > 0\,, \ \forall \bm \in \NN^q\,, \qquad
        \sup_n \max_{i\in[n]} \EE | (\bcx^t_i)^\bm | < \infty.
    \end{equation}
    Let $\psi : \RR^q \times [n] \to \RR$ be such that
    $\psi(\cdot, l)$ is a multivariate polynomial with a bounded
    degree and bounded coefficients as functions of $(l,n)$.
    Let $\cS^{(n)} \subset [n]$ be a non empty set such that
    $| \cS^{(n)} | \leq C K_n$. Then,
    \begin{subequations}
        \label{cvg-vec}
        \begin{align}
             & \frac{1}{K_n} \sum_{i\in \cS^{(n)}}
            \psi(\bcx^{t}_i,i) - \EE\psi(U^{t}_i,i) \toprobalong 0 
            \quad \text{and}                       \\
             & \frac{1}{n} \sum_{i\in[n]}
            \psi(\bcx^{t}_i,i) - \EE\psi(U^{t}_i,i) \toprobalong 0 \,.\label{cvg-vec-big}
        \end{align}
    \end{subequations}
\end{theorem}

\begin{remark}
    In this theorem, and particularly in the convergence described in \eqref{cvg-vec-big}, the result is not explicitly stated for all iterations from $1$ to $t$, as was done in \eqref{cvg-big}. Consequently, Matrix AMP can be interpreted as a more compact formulation of the ``standard" AMP. This distinction is further elucidated in the subsequent proof.
\end{remark}

\begin{proof}[Proof of Theorem~\ref{thm:amp-pscal}]
    Theorem~\ref{thm:amp-pscal} can be deduced from Theorem~\ref{thm:amp-vec}
    by adequately choosing $q$ as well as a precise construction of
    the activation function $f$ using the $\mathbb{R}$-valued polynomials
    $p$. Define the sequence $\left(\bs{\cx}^t\right)_{t\ge 1}$ as follows, 
    \begin{equation}
    \label{eq:cx_def}
        \left(\bs{\cx}^t\right)_{t\ge 1} = \AMPW\left(X, p, \bs{\cx}^0, S \right)\, .
    \end{equation}
    We shall establish the convergence \eqref{cvg-big} for each $t$ and prove that for all multivariate polynomials $\psi$ we have 
    $$\frac{1}{n} \sum_{i\in[n]}
            \left\{ \psi(\cx^{1}_i,\ldots,\cx^{t}_i,i) -
            \EE\psi(\cZ^{1}_i, \ldots, \cZ^{t}_i,i) \right\} \toprobalong 0 \,.$$
            where $(\bs{\cZ}^1, \cdots, \bs{\cZ}^t) \sim \DE(S,p, \bs{\cx}^0, t)$.
    To this end, let $\tau>0$ be fixed and chose $q=\tau$, construct the sequence $(\bx^{t})_{1\leq t \le \tau}$ of $\RR^{\tau\times \tau}$-valued matrices
    such that 
    \begin{equation*}
        \begin{split}
            \bx_i^1 &= \begin{pmatrix}
                \cx_i^1 & 0 & \cdots & 0
            \end{pmatrix}^\top, \\
            \bx_i^2 &= \begin{pmatrix}
                \cx_i^1 & \cx_i^2 & \cdots & 0
            \end{pmatrix}^\top, \\
            &\vdots \\
            \bx_i^\tau &= \begin{pmatrix}
                \cx_i^1 & \cx_i^2 & \cdots & \cx_i^\tau
            \end{pmatrix}^\top. \\
        \end{split}
    \end{equation*}
    Now using the polynomials $p$, we construct the function $f : \RR^\tau \times [n] \times \NN \to \RR^\tau$ 
    such that for all $i\in [n]$ and $0\leq \ell\leq \tau-1$ we have
    \begin{equation*}
        f(\bx, i, \ell) = \begin{pmatrix}
            p(x_i^0, i, 0) & p(\bx_i(1), i, 1) & \cdots & p(\bx_i(\ell), i, \ell) & 0 & \cdots & 0
        \end{pmatrix}^\top.
    \end{equation*}
    For $\ell\ge \tau$, we set $$f(\bx, i, \ell)=(0\ \cdots\ 0)\, .$$
    
In order to apply apply Theorem~\ref{thm:amp-vec}, we show that the sequence $(\bx^t)$ is given by
\begin{equation}
\label{eq:tprime}
    (\bx^{t})_{t\ge 1} = \AMPW_\tau\left(X,S, f, \bs{x}^0\right).
\end{equation}
Let $t \in [\tau-1]$. By definition, for $r\in [\tau]$ and $i\in [n]$ we have 
$$x^{t+1}_i(r)=\left\{\begin{array}{cc}
    \cx_i^{r}&  \text{if}\  r\leq t+1\,,\\
     0& \text{if} \ r>t+1 \,.
\end{array}\right. $$
In addition, by Eq.~\eqref{eq:cx_def} we know that 
$$\cx_i^{r} = \sum_{\ell \in [n]}W_{i\ell} p(\cx_\ell^{r-1},\ell, r-1) - \sum_{\ell\in [n]} W_{i\ell} W_{\ell i}\partial p(\cx_\ell^{r-1},\ell, r-1)p(\cx_i^{r-2},i, r-2),  $$
which implies that for $r\leq \tau+1$,
\begin{align*}
    x^{\tau+1}_i(r) &= \sum_{\ell \in [n]}W_{i\ell} p(x_\ell^{\tau}(r-1),\ell, r-1) \\&\quad- \sum_{\ell\in [n]} W_{i\ell} W_{\ell i}\partial p(x_\ell^{\tau}(r-1),\ell, r-1)p(x_i^{\tau-1}(r-2),i, r-2)\ ,\\
    &= \sum_{\ell \in [n]}W_{i\ell} p(x_\ell^{\tau}(r-1),\ell, r-1) \\&\quad - \sum_{s\in [t]}p(x_i^{\tau-1}(s-1),i, s-1)\sum_{\ell\in [n]} W_{i\ell} W_{\ell i}\partial p(x_\ell^{\tau}(r-1),\ell, r-1)\delta_{s,r-1}\ ,\\
    &=\sum_{\ell \in [n]}W_{i\ell} f_r(x_\ell^{\tau},\ell, \tau) - \sum_{s\in [t]}f_s(x_i^{\tau-1},i, \tau-1)\sum_{\ell\in [n]} W_{i\ell} W_{\ell i} \frac{\partial}{\partial x(s)}f_{r}(x_\ell^{\tau},\ell, \tau)\ ,
\end{align*}
which is precisely the recursion in \eqref{eq:tprime}. 

We can now apply the result of Theorem~\ref{thm:amp-vec} to the sequence $(\bs{x}^{t})$, which implies that for all polynomial test functions $\psi(.,\ell):\RR^\tau\to \RR$ we have 
\begin{equation*}
    \frac{1}{n} \sum_{i\in [n]} \psi(\bx_i^\tau,i) - \EE \psi(U_i^\tau,i) \toprobalong 0, \quad \forall \tau\in \NN\,,
\end{equation*}
which yields 
\begin{equation}
\label{eq:xU}
    \frac{1}{n} \sum_{i\in [n]} \psi(\cx_i^1, \cdots, \cx_i^\tau,i) - \EE \psi(U_i^\tau,i) \toprobalong 0, \quad \forall \tau\in \NN\,,
\end{equation}
where the $U^\tau$ is $(n\times \tau)$-dimensional random matrix with law $\DE_q(S,f,\bx^0,\tau)$, the latter is defined in \eqref{eq:MatrixDE}. Denote the columns of $U^\tau$ by $Z^1, \cdots, Z^\tau \in \RR^n$, then it is clear that $(Z^1, \cdots, Z^\tau)\sim \DE(p,\cx^0, S, \tau)$. The convergence in \eqref{eq:xU} becomes 
$$\frac{1}{n} \sum_{i\in [n]} \psi(\cx_i^1, \cdots, \cx_i^\tau,i) - \EE \psi(Z_i^1, \cdots, Z_i^\tau,i) \toprobalong 0, \quad \forall \tau\in \NN\,.$$
with $(Z^1, \cdots, Z^\tau)\sim \DE(S,p,\cx^0, \tau)$. Convergence \eqref{cvg-big} is established. One can prove similarly \eqref{cvg-small}, which concludes the proof of Theorem~\ref{thm:amp-pscal}.

\end{proof}

\section{Proof of Theorem~\ref{thm:amp-vec}: A combinatorial approach}
\label{sec:combinatorics}

Taking polynomial activation functions in Theorem~\ref{thm:amp-vec} is fundamental, as all iterations $\bx^t$ can be 
written as multinomials on the entries of the matrix $W$ and the initial point's coordinates $x^0_i(s)$. This makes 
the analysis purely combinatorial. At the first and second iterations $t=1,2$, and given simple polynomial activation functions 
$f_r(u, \ell, 1) = f_r(u, \ell, 0) = u(1)^{m}$, one can write
\begin{align*}
    & x_i^{1}(r) = \sum_{\ell \in [n]} W_{i\ell}x^0_\ell(1)^{m}, \\
    & x_i^{2}(r) = \sum_{\ell , \ell_1, \cdots, \ell_m\in [n]} W_{i\ell}W_{i\ell_1}\cdots W_{i\ell_m} \left(x_{\ell_1}^0(1)\cdots x_{\ell_m}^0(1)\right)^m - \{\text{Onsager}\}.
\end{align*}
We already notice that by the second iteration $t=2$, the exact expression for $x_i^2$ as a multinomial expansion in terms of the entries of
matrix $W$ becomes increasingly complex. We hence need to find an alternative indexation scheme for 
the summation above, properly suited to extract the desired information and establish Theorem~\ref{thm:amp-vec}. We follow the combinatorial approach initiated in \cite{Bayati_2015}. This approach is based on the introduction of ``non-backtracking" trees associated to ``non-backtracking" iterations.

\subsection{Strategy of proof}

To prove that the AMP iterations have the simple deterministic equivalent described in Theorem~\ref{thm:amp-vec} we first approximate the moments of $\bs{x}^t \in \RR^{n\times q}$ with the moments of simpler objects $\bs{z}^t$ called the ``non-backtracking'' iterations, these are generated with the same matrix $W$ used in the recursion~\eqref{def:AMPZ}, with a slightly different recursion scheme where the Onsager term is removed.
$$ \EE(\bs{x}^t_i)^\bm \approx \EE(\bs{z}^t_i)^\bm, \quad \forall \bm \in \NN^q,$$
this is done in (Proposition~\ref{prop:xtzt}) section~\ref{subsec:xtzt}. We then show a universality property of the iterations $\bs{z}^t$ in (Proposition~\ref{prop:ztzttilde}) section~\ref{subsec:zt}. More specifically, we show that if $\bs{\tilde{z}}^t$ is another non-backtracking iteration sequence generated using another matrix $\tilde{W}$ satisfying the same assumptions as $W$ but does not have the same distribution, then 
$$ \EE(\bs{z}^t_i)^\bm \approx \EE(\bs{\tilde{z}}^t_i)^\bm,\quad \forall \bm \in \NN^q . $$
This means that we can reduce our problem to an AMP constructed using a Gaussian matrix. Hence, without loss of generality we can suppose that $W$ is Gaussian. Moreover, we approximate the non-backtracking iterations $\bs{z}^t$ with another non-backtracking iterations $\bs{y}^t$, but this time, in the recursion formula of $\bs{y}^t$, at each step $t$ we independentally pick a new random matrix $W^t\eqlaw W$ which is Gaussian,
$$ \EE(\bs{\tilde{z}}^t_i)^\bm \approx \EE(\bs{y}^t_i)^\bm, \quad \forall \bm \in \NN^q . $$
this is done in (Proposition~\ref{prop:ztyt}) section~\ref{subsec:ztyt}. $\bs{x}^t$ is now reduced to its simplest form $\bs{y}^t$. Finally, we show in (Proposition~\ref{prop:ytut}) section~\ref{subsec:xtut} that 
$$ \EE(\bs{y}^t_i)^\bm \approx \EE(U^t_i)^\bm, \quad \forall \bm \in \NN^q . $$
which is relatively easy given that $\bs{y}^t$ are Gaussian. This finishes the proof of Theorem~\ref{thm:amp-vec}.

The proof of all these steps follows the combinatorial approach described in both \cite{Bayati_2015} and \cite{HACHEM2024104276} and thus we begin by presenting the framework of ``non-backtracking'' trees in section~\ref{subsec:treeStructure}. Notice that that the key difference between prior research and our approach is that the matrix $W$ is no longer symmetric, and exhibits some correlations between its entries.

\subsection{Description of the tree structure}
\label{subsec:treeStructure}

The proof of Theorem~\ref{thm:amp-vec} follows a combinatorial approach which aims
at studying the moments of the AMP iterates. In order to simplify the expression of 
these moments, we use \textit{planted and labeled trees} to index the sums in these 
expressions. We first define \textit{planted trees} and then describe its labeling.
\begin{definition}[Planted trees] 
We recall the following definition from graph theory.
\begin{itemize}
    \item A rooted tree $T = (V(T), E(T))$ at $\circ \in V(T)$, where $V(T)$ and $E(T)$ denote 
    respectively the set of vertices and edges, is said to be panted if the root $\circ$ has degree $1$.
    \item We consider that all the edges are oriented towards the root, we say that $v\in V(T)$ is the 
    parent of $u$ if the edge $(u\to v)$ is in $E(T)$, in this case, we use the notation $\pi(u) = v$, 
    we also say that $u$ is a child of $v$.
    \item We denote by $L(T)$ the set of leaves of $T$, i.e. vertices $v\in V(T)$ with no children.
    \item Given a vertex $v\in V(T)$, we denote by $|v|$ its distance to the root $\circ$.
    \item Finally, we define a \textit{path} starting at $v_1$ and ending at $v_k$ as a sequence of 
    vertices $(v_1, v_2, \cdots, v_k)$ such that $v_i = \pi(v_{i+1})$ for all $i\in [k-1]$.
\end{itemize}    
\end{definition}
We fix a integer $d,t \in \NN$, throughout this proof we consider the class of planted trees $(T, \circ)$
of depth at most $t$ such that for each vertex $v$, $v$ can have at most $d$ children.

We denote $$\NN_{\le d}^{\,q} := \{(a_1, \cdots, a_q) \in \NN^q\,,\  a_1+\cdots+a_q \leq d\}\ ,
$$
where $q$ is also a fixed integer.
\begin{definition}[Labeled and planted trees] We now describe the labeling of the trees.
A labeling of a tree $T$, is a triplet of functions $(\ell, r, c)$ such that 
$$ \ell : V(T) \to [n],  \quad r : V(T)\setminus \{\circ\} \to [q], \quad c : L(T)\to \NN_{\le d}^{\, q}.$$
\begin{itemize}
    \item For each vertex $u\in V(T)$, $\ell(u)$ is called the \textit{type} of $u$.
    \item For each vertex $u\in V(T)$ except the root, $r(u)$ is called the \textit{mark} of $u$.
    \item For each vertex $u\in V(T)$ which is not a leaf, we denote by $u[i]$ the number of 
    children of $u$ that have mark $i\in [q]$. We use the same notation to describe $c(u)$ for 
    $u\in L(T)$; $c(u) = \left(u[1], \cdots, u[q]\right) \in \NN_{\le d}^{\, q}$. In what follows, this notation is 
    used instead of $c(u)$.
    \item For a non-maximal leaf $u\in L(T)$, i.e. such that $|u|$ is less than the depth of $T$, 
    we set $u[1]=\cdots u[q] = 0.$
\end{itemize}
We denote by $\overline\cT^t$ the set of planted and labeled trees, with depth $t$ at most.
\end{definition}

\subsubsection*{Non-backtracking trees} One class of planted and labeled trees that is particularly adapted
to our specific study, is the class of trees satisfying the \textit{non-backtracking} condition, we recall
here the definition that can be found in \cite{Bayati_2015}. A non-backtracking tree is a planted and labeled 
tree $T$ such that for each path $(u_1=\circ, u_2 \cdots, u_k)$ in $T$ the types 
$(\ell(u_i), \ell(u_{i+1}), \ell(u_{i+2}))$ are distinct for each $i\in [k-2]$. We denote the class of these 
trees as $\cT^t$. In addition, we introduce the following classes of trees, for given integers $i, j$ and $r$,
we denote by,
\begin{itemize}
    \item $\cT^t_{i\to j}(r) \subset \cT^t$ the subset of trees in $\cT^t$ for
          which the type of the root is $i$, the type of the child $v$ of the root
          satisfies $\ell(v) \not\in \{ i, j \}$, and the mark of $v$ is $r(v) = r$.
    \item $\cT^t_{i}(r) \subset \cT^t$ the subset of trees in $\cT^t$ for
          which the type of the root is $i$, the type of the child $v$ of the root
          satisfies $\ell(v) \neq i$, and the mark of $v$ is $r(v) = r$.
\end{itemize}

We can already use these trees to create the following objects. For a matrix $W\in \RR^{n\times n}$, 
a vector $x\in \RR^n$ and a family of real numbers $\balpha = \left\{\alpha_{\iota} (r, \ell, s) \mid \iota \in \NN_{\le d}^{\, q}, (r, \ell, s)\in [q]\times [n] \times [t]\right\}$, we define,
\begin{align*}
    W(T)                  & := \prod_{(u\to v) \in E(T)} W_{\ell(v) \ell(u)}\ , \\
    \Gamma(T, \balpha, t) & := \prod_{(u\to v) \in E(T)}
    \alpha_{u[1],\ldots,u[q]}\left(r(u), \ell(u), t - |u|\right)\ ,                       \\
    x(T)                  & := \prod_{v\in L(T)} \prod_{s\in[q]}
    \left( x_{\ell(v)}(s)\right)^{v[s]}\ .
\end{align*}

To better illustrate the concepts previously defined, we present a simple example of a tree and 
demonstrate how it indexes the tree quantities $W$, $\Gamma$, and $x$.

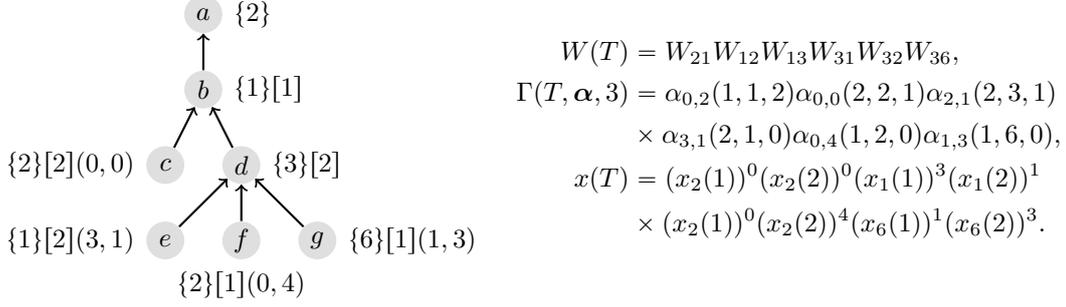
\begin{figure}[ht]
\begin{minipage}{0.45\textwidth}
    \begin{tikzpicture}
  [level distance=10mm,
   every node/.style={fill=gray!25,circle,inner sep=1pt, minimum size=0.5cm},
   level 1/.style={sibling distance=20mm,nodes={fill=gray!25}},
   level 2/.style={sibling distance=10mm,nodes={fill=gray!25}},
   level 3/.style={sibling distance=10mm,nodes={fill=gray!25}},
   edge from parent/.style={draw, <-, thick},  
   ]
   
  \node[fill=gray!25,circle,inner sep=1pt] (node0) {$a$}
     child {node (node1) {$b$}
     child {node (node2) {$c$}}
     child {node (node3) {$d$}
        child {node (node4) {$e$}}
        child {node (node5) {$f$}}
        child {node (node6) {$g$}}
        }
     };
     
    \begin{scope}[every node/.style={}] 
    \node[right=0.25mm of node0] {$\{2\}$};
    \node[right=0.25mm of node1] {$\{1\} [1]$};
    \node[left=0.25mm of node2] {$ \{2\} [2] (0, 0)$};
    \node[right=0.25mm of node3] {$\{3\}[2]$};
    \node[left=0.25mm of node4] {$\{1\} [2] (3, 1)$};
    \node[below=0.25mm of node5] {$\{2\} [1] (0, 4)$};
    \node[right=0.25mm of node6] {$\{6\} [1] ( 1, 3)$};
  \end{scope}
  
\end{tikzpicture}
\end{minipage}
\begin{minipage}{0.45\textwidth}
\begin{align*}
    W(T) &= W_{21} W_{12} W_{13} W_{31} W_{32} W_{36},\\
    \Gamma(T, \bs{\alpha}, 3) &= \alpha_{0,2}(1,1,2)\alpha_{0,0}(2,2,1)\alpha_{2,1}(2,3,1)\\
    &\times\alpha_{3,1}(2,1,0)\alpha_{0,4}(1,2,0)\alpha_{1,3}(1,6,0),\\
    x(T) &= (x_2(1))^0(x_2(2))^0(x_1(1))^3(x_1(2))^1\\
    &\times (x_2(1))^0(x_2(2))^4(x_6(1))^1(x_6(2))^3.\\
\end{align*}
\end{minipage}
\caption{Example of a tree $T\in\overline\cT^3$ for parameters $q=2$, $d=4$, $t=3$ and $n=6$. The types are written between braces, the marks are between brackets and leafs info is between parentheses. In this example, $T$ is not a non-backtracking tree because of the two paths $(a\leftarrow b \leftarrow c)$ and $(b\leftarrow d \leftarrow e)$.}
\end{figure}

\subsection{Non-backtracking iterations}
\label{subsec:zt}

The non-backtracking iterations $(\bz^t)_t$, are defined recursively similarly to $(\bx^t)_t$ but minus the Onsager term and with a slight change in the contributing terms from the previous iteration. Recall that the purpose of having the Onsager term is to eliminate components that induce non-Gaussian behavior in the iterates in the high dimensional regime. Basically, non-backtracking iterations evolve purposefully getting rid of parts that are source non-Gaussian behavior. In particular we do not need to have a corrective term.\\

Given any $i,j\in
    [n]$ with $i\neq j$, we initialize the non-backtracking sequence with $\bz^0_{i\to j} := \bcx^0_i$. We then define recursively $\bz^{t+1}_{i\to j}$ using the previous iterations as follows 
\begin{equation}
    \label{zij(r)}
    z^{t+1}_{i\to j}(r) = \sum_{\ell\in[n]\setminus\{j\}}
    W_{i\ell} f_r(\bz^t_{\ell\to i}, \ell, t), \quad \forall r \in [q],
\end{equation}
the case $l= i$ is excluded because $W_{ii} = 0$.
In addition, we also define the vectors $(\bz^t)_t$ by
\begin{equation}
    \label{zi(r)}
    z^{t+1}_{i}(r) = \sum_{\ell\in[n]} W_{i\ell} f_r(\bz^t_{\ell\to i}, \ell, t), \quad \forall r \in [q].
\end{equation}

We provide here a non-recursive formulation of $z_{i\rightarrow j}^t$ and $z_i^t$ described as sums
indexed by trees in $\cT^t_{i\to j}(r)$ and $ \cT^t_{i}(r)$.
\begin{lemma}[Lemma 1 of \cite{Bayati_2015}]
    \label{ztree}
    For all integers $t\in \NN$, $i,j\in [n]$ and $r\in [q]$, we have,
    \begin{align}
        z^t_{i\to j}(r) & = \sum_{T\in \cT^t_{i\to j}(r)} W(T) \Gamma(T,\balpha, t)
        x(T), \nonumber                                                               \\
        z^t_{i}(r)      & = \sum_{T\in \cT^t_{i}(r)} W(T) \Gamma(T,\balpha, t) x(T) .
        \nonumber
    \end{align}
    Here $x(T) := \bs{x}^0(T)$, we drop the superscript from this notation.
\end{lemma}
Note that this lemma is purely structural, the proof is not impacted by our specific
variance and correlation profiles.

To simplify the notations in the following proofs we introduce the following sets,
\begin{equation}
\label{def:calKcalC}
    \cK = \left\{ (i,j)\in [n]\times [n]\,,\  \ s_{ij} > 0 \right\} \quad \text{and} \quad \mathcal{C} = \left\{ (i,j)\in [n]\times [n]\,,\  \ \tau_{ij} \neq 0 \right\}.
\end{equation}
We also define the row and column sections of $\cK$,
\begin{equation}
\label{def:calKcalC2}
     \cK_i = \left\{ j \in [n]\,,\  \ s_{ij} > 0 \right\} \quad \text{and} \quad \cK^j = \left\{ i \in [n]\,,\ \ s_{ij} > 0 \right\}.
\end{equation}

The next proposition shows that in the large dimensional regime, the
moments of a vector $\bz^t_i$ issued from the non-backtracking
iterations depend for large $n$ only on the first two moments of the elements
of $W$.
\begin{proposition}
    [adaptation of Proposition 1 of \cite{Bayati_2015}]
    \label{z-tz}
    \label{prop:ztzttilde}
    Let $\tX$ be a random matrix satisfying A-\ref{ass:X}, with distribution not
    necessarily identical to its analogue $X$. Assume that $W$ fulfills A-\ref{ass:S}. Let $\tW$ be the
    matrix constructed similarly to $W$, but with the $X_{ij}$ replaced with the
    $\tX_{ij}$.
    Starting with the set of $\RR^q$--valued vectors $\{ \tbz^0_{i\to j}, \ i,j \in
        [n], \ i\neq j \}$ given as $\tbz^0_{i\to j} = \bcx^0_i$, define the vectors
    $\tbz_i^t \in \RR^q$ by the recursion~\eqref{zij(r)} and the
    equation~\eqref{zi(r)}, where $W$ is replaced with $\tW$.  Then, for each
    $t\geq 1$ and each $\bm \in \NN^q$,
    \begin{equation*}
        \left| \EE (\bz^t_i)^\bm - \EE (\tbz^t_i)^\bm \right| = \mathcal{O}\left(\frac{1}{\sqrt{K_n}}\right).
    \end{equation*}

\end{proposition}
\begin{proof}
    Without loss of generality, we restrict the proof to the case where the multi-index $\bm$
    satisfies
    $$ \bm(s) =\left\{
    \begin{array}{ll}
          0 & \text{if} \ s\neq r, \\
          m & \text{if} \ s=r,
    \end{array}
    \right.$$
    for some integer $m > 0$. By Lemma~\ref{ztree}, we have
    \[
        \EE (z^t_i(r))^m = \sum_{T_1,\ldots, T_m \in \cT^t_{i}(r)}
        \left( \prod_{k=1}^m \Gamma(T_k,\balpha, t) \right)
        \EE \left[ \prod_{k=1}^m W(T_k) \right]
        \prod_{k=1}^m x(T_k) .
    \]
    For a tree $T$ and $j,\ell \in [n]$, define
    \[
        \vec{\varphi}_{\ell j}(T) = \left| \left\{ (u\to v) \in E(T), \ ( \ell(u),\ell(v) ) =
        (j, \ell) \right\} \right|.
    \]
    Based on the definition of $W(T)$,  $\vec{\varphi}_{\ell j}(T)$ counts the number of edges in the tree $T$ that represent the $(\ell,j)$ matrix entry $W_{\ell j}$. We also define $\varphi_{j\ell}$ for $j < \ell$ as

    \[
        \varphi_{j\ell}(T) = \vec{\varphi}_{j\ell}(T)+ \vec{\varphi}_{\ell j}(T),
    \]
    this quantity represents the total number of edges in the tree $T$ that represent either $W_{j\ell}$ or $W_{\ell j}$.
    We know that there is an integer constant $C_E = C_E(d,t,m)$ that bounds the total number
    of edges in the trees $T_1,\ldots, T_m \in \cT^t_{i}(r)$, thus
    \[
        \sum_{k\in[m]} \sum_{j<\ell} \varphi_{j\ell}(T_k) \leq C_E = m \frac{d^t - 1}{d - 1} .
    \]
    $C_E$ is simply the maximum number of edges in the $m$-tuple of trees $T_1, \cdots, T_m$.
    Given an integer $\mu \in [C_E]$, recall that $\cK$ is introduced in \eqref{def:calKcalC}, define
    \begin{align*}
        \cA_i(\mu) := \Bigl\{ & (T_1,\ldots, T_m)\,,\  T_k \in \cT^t_{i}(r) \
        \text{for all} \ k \in [m] ,                                                                               \\
                             & \forall j < \ell, \ \sum_{k\in[m]} \varphi_{j\ell}(T_k) \neq 1 ,                          \\
                             & \forall j , \ell, \ \sum_{k\in[m]} \vec{\varphi}_{j\ell}(T_k) > 0 \ \Rightarrow (j,l) \in
        \cK,                                                                                                       \\
                             & \sum_{k\in[m]} \sum_{j<\ell} \varphi_{j\ell}(T_k) = \mu \Bigr\}.
    \end{align*}
    Since the elements of $W$ beneath the diagonal are centered and
    independent, then,
    \begin{equation}
        \label{ztm}
        \EE (z^t_i(r))^m = \sum_{\mu=1}^{C_E} \sum_{(T_1,\ldots, T_m) \in \cA_i(\mu)}
        \left( \prod_{k=1}^m \Gamma(T_k,\balpha, t) \right)
        \left( \prod_{k=1}^m x(T_k)  \right)
        \EE \left[ \prod_{k=1}^m W(T_k) \right] .
    \end{equation}
    Notice that the contributions of the $m$--uples of trees in the set
    \begin{equation*}
        \left\{ (T_1, \cdots, T_m) \in \cA_i(\mu) \,,\ \forall j < \ell, \
        \sum_{k\in[m]} \varphi(T_k)_{j\ell} \in \{ 0, 2 \} \right\}\ ,
    \end{equation*}
    are the same for $\EE (z^t_i(r))^m$ and $\EE (\tilde z^t_i(r))^m$ by the
    assumptions on the matrices $W$ and $\tW$. Three cases can be considered for
    a couple of indices $(j,\ell)$ where $j<\ell$ and $\sum_{k\in[m]} \vec{\varphi}(T_k)_{j\ell} = 2 $,
    \begin{itemize}
        \item $W_{j\ell}$ is represented two times in the trees $\Rightarrow$ contribution equal to $s_{j\ell}$,
        \item $W_{\ell j}$ is represented two times in the trees $\Rightarrow$ contribution equal to $s_{\ell j}$,
        \item $W_{j\ell}$ and $W_{\ell j}$ are both represented in the trees $\Rightarrow$ contribution equal to $\sqrt{s_{j\ell} s_{\ell j}}\tau_{j\ell}$.
    \end{itemize}
    Notice that in all three cases the contributions do not depend on the distributions of the entries of the matrix $W$ but only on the first and second moments.
    Thus, defining the set
    \begin{equation}
    \label{def:Acheck}
        \widecheck\cA_i(\mu) = \Bigl\{ (T_1,\ldots, T_m) \in \cA_i(\mu) \,,\ 
        \exists j < \ell, \ \sum_{k\in[m]} \vec{\varphi}(T_k)_{j\ell} \geq 3 \Bigr\},
    \end{equation}
    the proposition can be proven if we prove that for all $\mu\in [C_E]$, the
    real number
    \begin{equation*}
        \xi_{\mu} =
        \sum_{(T_1,\ldots, T_m) \in \widecheck\cA_i({\mu})}
        \left( \prod_{k=1}^m \Gamma(T_k,\balpha, t) \right)
        \left( \prod_{k=1}^m x(T_k)  \right)
        \EE \left[ \prod_{k=1}^m W(T_k) \right]
    \end{equation*}
    satisfies $$|\xi_{\mu} | =\mathcal{O}\left(\frac{1}{\sqrt{K_n}}\right). $$

    Using the bounds \eqref{alpha-vec} provided in the statement of Theorem~\ref{thm:amp-vec}, it is clear that
    $\prod_{k=1}^m \Gamma(T_k,\balpha, t)$ and $\prod_{k=1}^m x(T_k)$ are bounded as $n$ goes to infinity.

    Since there exists a constant $C$ such that $| \EE W_{j\ell}^s | \leq C K_n^{-s/2}$ for each integer $s > 0$ by
    A-\ref{ass:X} and A-\ref{ass:S}, for each $(T_1,\ldots, T_m) \in
        \widecheck\cA_i({\mu})$, we have
    \begin{eqnarray*}
            \left| \EE \prod_{k=1}^m W(T_k) \right| &=&
            \prod_{j<\ell} \Bigl| \EE W_{j\ell}^{\sum_{k=1}^m \vec{\varphi}_{j\ell}(T_k)} W_{\ell j}^{\sum_{k=1}^m \vec{\varphi}_{\ell j}(T_k)} \Bigr| \ ,\\
            &\leq& \prod_{j<\ell} \Bigl| \left(\EE W_{j\ell}^{2\sum_{k=1}^m \vec{\varphi}_{j\ell}(T_k)}\right)^{1/2} \left(\EE W_{\ell j}^{2\sum_{k=1}^m \vec{\varphi}_{\ell j}(T_k)}\right)^{1/2}\ ,\\
            &\leq& CK_n^{-\frac{1}{2}\sum_{j<\ell}\sum_{k}\vec{\varphi}_{j\ell}(T_k) + \vec{\varphi}_{\ell j}(T_k)}
            \quad \leq \quad C K_n^{-{\mu} / 2}\ .
    \end{eqnarray*}
    To complete the proof, we shall show that
    \begin{equation}
        \label{x(T)}
        \left| \widecheck\cA_i({\mu}) \right| = \mathcal{O}\left(K_n^{\frac{\mu-1}{2}}\right)\ .
    \end{equation}

    Given an $m$--uple $(T_1,\ldots, T_m) \in \widecheck\cA_i({\mu})$ of trees, we
    construct a graph $G = \bG(T_1,\ldots, T_m)$ by identifying the types of the vertices
    in all these trees. The marks as well as the orientation of the edges are ignored. 
    $G$ is then a rooted and labeled graph whose root is the vertex obtained by merging
    the roots of the trees $T_1,\ldots,T_m$ (remember that they all have the same type $i$).
    
    \noindent The number of edges of $G$ is
    \begin{equation*}
        | E(G) | = \sum_{j<\ell} \1_{\sum_k \varphi(T_k)_{j\ell} > 0}.
    \end{equation*}
    Remember that when $\sum_k \varphi(T_k)_{j\ell} > 0$, this sum is greater than $2$, so
    $$\forall j<\ell, \quad \sum_k \varphi(T_k)_{j\ell}  \geq 2 \1_{\sum_k \varphi(T_k)_{j\ell} > 0},$$
    we also know that for some $j<\ell$ we have $\sum_k \varphi(T_k)_{j\ell}\geq 3$.
    Consequently,
    $$
    2 (| E(G) |-1) + 3 \quad \leq\quad  \sum_{j<\ell} \sum_{k=1}^m \varphi(T_k)_{j\ell}\ ,
    $$ thus,
    $$
        | E(G) | \leq \frac{{\mu} - 1}{2}\ .
   $$
    Note that since $G$ is connected, as being obtained through the merger of
    planted trees with the same root's type,
    $$| V(G) | \leq | E(G) | +1,$$ which gives $$| \{ v \in V(G)\,,\ v \neq \degree \} |
        \leq ({\mu} - 1) / 2.$$ 
        Also, by construction, $G$ satisfies the following property:
    $$
        (u\to v) \in E(G) \ \Rightarrow \ \ell(u) \in \cK_{\ell(v)},
    $$
    where $\cK_i$ is defined in \eqref{def:calKcalC2}. And by A-\ref{ass:S}, this implies that $G$ satisfies the following property: 
    for any fixed labeled vertex $v\in V(G)$ if $(u\to v)\in E(G)$ then $u$ can be labeled 
    by at most $C K_n$ different values.

    We shall denote as $\cG^{\mu}_i$ the set of rooted, undirected and labeled graphs $G$ such that
    \begin{itemize}
        \item  $G$ is connected,
        \item $\ell(\degree) = i$, $| E(G) | \leq ({\mu} -
                  1)/2$,
        \item for any fixed labeled vertex $v\in V(G)$ if $(u,v)\in E(G)$ then $u$ can be labeled by at most $C K_n$ different values.
    \end{itemize}

    We denote as $\cR^{\mu}$ the set of
    all the elements of $\cG^{\mu}_i$ but without the labels. Given a graph
    $G \in \cG^{\mu}_i$, let us denote as $\bar G = \bU(G) \in \cR^{\mu}$ the
    unlabeled version of $G$.  With these notations, we have
    \begin{equation}
        \label{xcheck}
        \left| \widecheck\cA_i({\mu}) \right| =
        \sum_{\bar G \in \cR^{\mu}} \
        \sum_{\substack{G \in \cG^{\mu}_i \, : \\\bU(G) = \bar G}} \
        \left| \left\{
        (T_1,\ldots, T_m) \in \widecheck\cA_i({\mu}) \,,\ 
        \bG(T_1,\ldots, T_m)= G \right\} \right| .
    \end{equation}
    For each graph $G$, it is clear that
    \begin{equation}
        \label{TG}
        \left| \left\{ (T_1,\ldots, T_m) \in \widecheck\cA_i({\mu}) \,,\ 
        \bG(T_1,\ldots, T_m)= G \right\} \right| \quad \leq\quad  C\,,
    \end{equation}
    where $C = C(d,t,m)$ is independent of $G$. Our goal now is to show that
    \begin{equation}
        \label{GKmu}
        \left| \left\{ G \in \cG^{\mu}_i \,,\ \bU(G) = \bar G \right\} \right|
        \leq C K_n^{({\mu}-1) / 2}\,,
    \end{equation}
    which is simply the number of all possible labelings of a graph $\bar G$ under the constraints described above. To see this, consider a breadth first search ordering of the vertices of the graph $v_0 = \degree< v_1< \cdots < v_{|V(\bar G)|-1}$ that begins at the root $\degree$, this ordering has the property of visiting each vertex once and that each new vertex is connected to an already visited vertex, i,e.
    \begin{itemize}
        \item $\{v_0 = \degree,v_1, \cdots, v_{|V(\bar G)|-1}\} = V(\bar G)$,
        \item $\forall j = 1, \cdots \ \exists k < j $ such that $(v_j\to v_k)\in E(\bar G)$.
    \end{itemize}
    Now, starting with $v_1$ and by induction, after fixing the label of $v_{j-1}$, one can 
    see that $v_j$ can only be labeled in at most $C K_n$ possible ways. So the number of all 
    possible labelings of $\bar G$ is bounded by $C K_n^{|V(\bar G)|-1} \leq C K_n^{(\mu-1)/2}$.

    Furthermore, it is easy to check that
    $$
        \left| \cR^{\mu} \right| \leq C .
    $$
    
    Getting back to equality~\eqref{xcheck}, and using this last inequality along
    with inequalities~\eqref{GKmu} and~\eqref{TG}, we obtain
    inequality~\eqref{x(T)}, and the proposition is proved.
\end{proof}

Notice that for a tuple of trees $(T_1,\cdots, T_m)$ satisfying the following condition
$$
    \forall j< \ell, \ \sum_{k\in [m]} \varphi_{j\ell}(T_k) \in \{0,2\},
$$
if there exists a pair $(j,\ell)$ such that $\sum_{k\in [m]} \vec{\varphi}_{j\ell}(T_k) = 1$ and $(j,\ell)\in \mathcal{C}$, i.e. $\tau_{j\ell}\neq 0$, then $\mathbb{E}\left[\prod_{k=1}^m W(T_k)\right] = 0$. Consider the following subset $\tilde{\mathcal{A}}_i(\mu)$ of $\mathcal{A}_i(\mu)$ defined
\begin{align}
\label{def:Atilde}
    \tilde{\cA_i}(\mu) = \Bigl\{ & (T_1,\ldots, T_m)\in \cA_i(\mu),                                                    \\
                                 & \forall j< \ell, \ \sum_{k\in [m]} \varphi_{j\ell}(T_k) \in \{0,2\},  \notag                    \\
                                 & \forall j , \ell, \ \sum_{k\in[m]} \vec{\varphi}_{j\ell}(T_k) = 1 \ \Rightarrow (j,\ell) \in
    \mathcal{C} \Bigr\}.\notag
\end{align}
If $(T_1, \cdots, T_m) \in \tilde{\cA_i}(\mu)$ then the graph $G = \bG(T_1, \cdots, T_m)$ constructed by merging the trees has exactly $\mu/2$ edges, and that can be seen by writing
\begin{equation*}
    \begin{split}
        \sum_{k=1}^m \varphi_{jl}(T_k) &= 2 \ 1_{\sum_{k=1}^m \varphi_{jl}(T_k) > 0} \, , \\
        |E(G)| &= \sum_{j<l} 1_{\sum_{k=1}^m \varphi_{jl}(T_k) > 0} =\sum_{j<l} \frac{1}{2}\sum_{k=1}^m \varphi_{jl}(T_k) = \mu /2.\\
    \end{split}
\end{equation*}
Define the set of graphs $\tilde{\cG}_i^\mu$ analogously  to $\cG_i^\mu$ with the difference that we replace the requirement $|E(G)|\leq (\mu-1)/2$ with $|E(G)|= \mu/2$. We can then write
\begin{equation}
    \EE z_i^t(r)^m = \sum_{\mu=1}^{C_E} \chi_\mu + \sum_{\mu=1}^{C_E} \xi_\mu,
\end{equation}
where
\begin{equation}\label{def:ximu}
    \chi_\mu  = \sum_{\bar G \in \cR^{\mu}}
    \sum_{\substack{G \in \tilde{\cG}_i^\mu \, : \\\bU(G) = \bar G}}
    \sum_{\substack{(T_1, \cdots, T_m)\in \tilde{\cA_i}(\mu) \, : \\ \bs{G}(T_1,\cdots, T_m) = G}}
    \left( \prod_{k=1}^m \Gamma(T_k,\balpha, t) \right)
    \left( \prod_{k=1}^m x(T_k)  \right)
    \EE \left[ \prod_{k=1}^m W(T_k) \right] .
\end{equation}
Recalling that $|\xi_\mu| = \mathcal O(K_n^{-1/2})$, we focus on the
$\chi_\mu$. To that end, we further decompose the first sum on the unlabeled
graphs $\bar{G}\in \cR^\mu$ above into a sum on the graphs which are trees and
a sum on the graphs which are not trees, i.e., those that contain a cycle. Let 
us denote respectively the corresponding sums by $\chi_\mu^{\text{T}}$ and
$\chi_\mu^{\text{NT}}$, and write 
$$
    \chi_\mu = \chi_\mu^{\text{T}} + \chi_\mu^{\text{NT}}.
$$
We  show in the following lemma that the contribution of the term $\chi_\mu^{\text{NT}}$ is negligible.
\begin{lemma} Consider the same framework as in Proposition \ref{prop:ztzttilde}.
   We have $$\chi_\mu^{\text{T}} =\mathcal{O}(1) \quad \text{and} \quad \chi_\mu^{\text{NT}} =\mathcal{O}\left(\frac{1}{K_n}\right).$$
\end{lemma}
\begin{proof}
    In the proof of Proposition~\ref{prop:ztzttilde}, we have already 
    got that $|\EE\left[\prod_{k=1}^m W(T_k)\right]|$ is bounded by $CK_n^{-\mu/2}$, 
    so we only need to study the quantity
    \[
        \left| \left\{ G \in \tilde{\cG}^{\mu}_i \, , \ \bU(G) = \bar G \right\} \right|,
    \]
    in the case where $\bar{G}$ is a tree and where $\bar{G}$ in not a tree. 
    Recall that for a given $G\in \tilde{\cG}^{\mu}_i$ the graph $G$ is connected 
    and we have $|E(G)| = \mu/2$ so $|V(G)\setminus \{\degree\}| \leq \mu/2$ with 
    the equality if and only if $G$ is a tree. So repeating the same argument as 
    in Proposition~\ref{prop:ztzttilde} we find that
   $$    
            \left| \left\{ G \in \tilde{\cG}^{\mu}_i \, , \ \bU(G) = \bar G \right\} \right|
            \leq C K_n^{{\mu} / 2}\qquad \textrm{and}\qquad 
            \left| \left\{ G \in \tilde{\cG}^{\mu}_i \, , \ \bU(G) = \bar G \right\} \right|
            \leq C K_n^{{\mu} / 2 - 1}\ ,
   $$ 
    in the case of $\bar{G}$ being a tree and not a tree respectively. 
    Multiplying by $C K_n^{-\mu/2}$ yields to the desired result.
\end{proof}

\subsection{Approximation of the non-backtracking iterations}
\label{subsec:ztyt}

For each $n$, let us now consider an i.i.d.~sequence
$(W^{t})_{t=0,1,\ldots}$ of $n\times n$ matrices such that
$W^{t} \eqlaw W$.  We define the vectors $\by^t_{i\to j}$ and
$\by^t_i$ recursively in $t$ similarly to what we did for the vectors
$\bz^t_{i\to j}$ and $\bz^t_i$, with the difference that we now replace the
matrix $W$ with the matrix $W^t$ at step $t$. More precisely, we set
$\by^0_{i\to j} = \bcx^0_i$ for each $i,j \in [n]$ with $i\neq j$. Given $\{
    \by^t_{i\to j}\mid \ i,j \in [n], \ i\neq j \}$, we set
\begin{equation}
\label{def:yij}
    y^{t+1}_{i\to j}(r) = \sum_{\ell\in[n]\setminus\{j\}}
    W_{i\ell}^t f_r(\by^t_{\ell\to i}, \ell, t) ,  \quad i \neq j.
\end{equation}
Also,
\begin{equation}
\label{def:yi}
    y^{t+1}_{i}(r) = \sum_{\ell\in[n]} W_{i\ell}^t f_r(\by^t_{\ell\to i}, \ell, t) .
\end{equation}
We introduce here a similar quantity to $W(T)$ for a given labeled tree which 
is adapted to the computations related to the iterations $y^{t}_i$. 
We define $\overline{W}(T,t)$ by
$$
    \overline{W}(T,t) = \prod_{(u\rightarrow v) \in E(T)} W^{t- |u|}_{\ell(v)\ell(u)},
$$
where we recall that $|u|$ denotes the distance of the vertex $u$ to the 
root $\degree$ in the tree $T$.

\noindent We can prove similar structural identities for $y^{t}_{i}$ and $y^{t}_{i\rightarrow j}$ 
as what we did with the iterates $z^{t}_{i}$ and $z^{t}_{i\rightarrow j}$. In fact, we have
\begin{align}
    y^t_{i\to j}(r) & = \sum_{T\in \cT^t_{i\to j}(r)} \overline{W}(T,t) \Gamma(T,\balpha, t)
    x(T), \nonumber                                                                            \\
    y^t_{i}(r)      & = \sum_{T\in \cT^t_{i}(r)} \overline{W}(T,t) \Gamma(T,\balpha, t) x(T) .
    \nonumber
\end{align}

\begin{proposition}
\label{prop:ztyt}
    Let $(\bz^t)$ and $(\bs{y}^t)$ two sequences defined in \eqref{zi(r)} and \eqref{def:yi} respectively, 
    then for each $t \geq 1$ and each $\bm \in \NN^q$, we have that for each $i\in [n]$,
    $$
        \left| \EE (\bz^t_i)^\bm - \EE (\bs{y}^t_i)^\bm \right| = \mathcal{O}\left(\frac{1}{\sqrt{K_n}}\right)  .
    $$
\end{proposition}

\begin{proof}
    We follow the same strategy of proof as in Proposition~\ref{prop:ztzttilde}. For simplicity let us fix $\bm(r)=m$ for a certain $r\in [q]$. We have
    \[
        \EE\left[y_i^t(r)^m\right] = \sum_{\mu = 1}^{C_E} \sum_{(T_1, \cdots, T_m)\in \cA_i(\mu)} \left(\prod_{k=1}^{m} \Gamma(T_k,\balpha, t)\right)\left(\prod_{k=1}^{m} x(T_k)\right) \EE\left[\prod_{k=1}^m \overline{W}(T_k, t)\right].
    \]
    As in the case of $(z_i^t)$, we can also decompose this sum into a sum over trees $(T_1, \cdots, T_m)$ in the set $\widecheck\cA_i(\mu)$ (defined in \eqref{def:Acheck}) and trees that are in the set $\tilde{\mathcal{A}}_i(\mu)$ (defined in \eqref{def:Atilde}). The contribution of $m$-tuples of trees in $\widecheck\cA_i(\mu)$ is of order $K_n^{-1/2}$, so we may focus on $m$-tuples of trees in $\tilde{\mathcal{A}}_i(\mu)$. Recall the definition of a graph $G\in \cG_i^\mu$ as the merger of trees $(T_1, \cdots, T_m)$ where we identify vertices $u$ that have the same label $\ell(u)$. As in the previous proof, we further partition these graphs into trees and graphs that contain at least a cycle. The latter have a contribution of order $K_n^{-1}$ so we may focus on the contribution of graphs $G$ that are trees. Write
    $$
        \bar{\chi}_\mu^T  = \sum_{\substack{\bar G \in \cR^{\mu} \\ \bar G \text{ is a tree}}}
        \sum_{\substack{G \in \tilde{\cG}_i^\mu \, : \\\bU(G) = \bar G}}
        \sum_{\substack{(T_1, \cdots, T_m)\in \tilde{\cA_i}(\mu) \, : \\ \bs{G}(T_1,\cdots, T_m) = G}}
        \left( \prod_{k=1}^m \Gamma(T_k,\balpha, t) \right)
        \left( \prod_{k=1}^m x(T_k)  \right)
        \EE \left[ \prod_{k=1}^m \overline{W}(T_k) \right]\ .
    $$
    The proof of this proposition will be completed if we can show that $\chi_\mu^T=\bar{\chi}_\mu^T$.

    First, notice that the terms $\prod_{k=1}^m \Gamma(T_k,\balpha, t) $ and $\prod_{k=1}^m x(T_k) $ are the same in the expressions of $\chi_\mu^T$ (defined in \eqref{def:ximu}) and $\bar{\chi}_\mu^T$. So it suffices study the term $\EE \left[ \prod_{k=1}^m \overline{W}(T_k) \right]$. Two cases can be studied, whether this term is zero or non-zero.

    Consider any $m$-tuple of trees $(T_1, \cdots, T_m)\in \tilde{\cA_i}(\mu)$, if 
    $$
    \EE \left[ \prod_{k=1}^m \overline{W}(T_k) \right] \neq 0\ ,
    $$ 
    then for every matrix entry $(i,j)$ which is represented in the trees $T_1, \cdots, T_m$ there exist exactly two edges $(a\rightarrow b)$ and $(c\rightarrow d)$ such that $\{\ell(a),\ell(b)\} = \{\ell(c),\ell(d)\} = \{i,j\}$, in addition $|a| = |c|$ otherwise $\EE \left[ \prod_{k=1}^m \overline{W}(T_k) \right] = 0$, we then obtain a second moment of $W$ which means that 
    $$
    \EE \left[ \prod_{k=1}^m \overline{W}(T_k) \right]= \EE \left[ \prod_{k=1}^m W(T_k) \right]\ .
    $$
    Now suppose for the sake of contradiction that $$\EE \left[ \prod_{k=1}^m \overline{W}(T_k) \right] = 0\quad \text{and}\quad \EE \left[ \prod_{k=1}^m W(T_k) \right] \neq 0\ ,$$
    we  show that in this case the graph $G = \bG (T_1, \cdots, T_m)$ is not a tree which is a contradiction. There exists a matrix entry $(i,j)$ with $i<j$ which is represented in the trees $(T_1, \cdots, T_m)$ by two edges $(a\rightarrow b)$ and $(c\rightarrow d)$ such that vertices $a$ and $c$ do not have the same distance to the root $\degree$, i.e. $|a| > |c|$ for example. This is because $\EE \left[ \prod_{k=1}^m \overline{W}(T_k) \right] = 0$ and because $\tau_{ij}\neq 0$, $s_{ij}\neq 0$ and $s_{ji}\neq 0$. Three possible cases can be considered:
    \begin{itemize}
        \item $(a\rightarrow b)$ and $(c\rightarrow d)$ exist on the same path of a certain tree: by the non-backtracking condition, these edges should be separated by at least one vertex say $e$ of label $k\notin \{i,j\}$, i.e.: $$\cdots \rightarrow a \rightarrow b \rightarrow e \rightarrow \cdots \rightarrow c \rightarrow d \cdots \rightarrow \degree. $$
              As for the graph $G$, this means that starting from a vertex of label $i$ we should pass through a vertex of label $k\notin \{i,j\}$ and then return to the vertex of label $i$ which creates a cycle.

        \item  $(a\rightarrow b)$ and $(c\rightarrow d)$ exist in two different trees say $T_1$ and $T_2$ respectively:
              \begin{align*}
                  \cdots \rightarrow a \rightarrow b \rightarrow \cdots \rightarrow \cdots \rightarrow \degree \quad (T_1) \\
                  \cdots \rightarrow * \rightarrow c \rightarrow \ d \rightarrow \cdots \rightarrow \degree \quad (T_2)
              \end{align*}
              First notice that the labels of the vertices in each of these two paths are different: if two vertices on the same path have the same label say $k$ then due to the non-backtracking condition they should be separated by at least two other vertices which  result in a cycle in the graph $G$. Recall that the roots $\degree_{T_1}$ and $\degree_{T_2}$ are identified in the graph $G$ which means that in $G$ there exist a path from the vertex $\ell(b)$ to $\degree$ and another path from $\ell(d)$ to $\degree$, these two paths are distinct as they have different lengths which is a consequence of the condition $|a|<|c|$. In addition $\ell(b)$ and $\ell(d)$ are either equal or linked in $G$, this creates a cycle in the graph.

        \item $(a\rightarrow b)$ and $(c\rightarrow d)$ exist in two different paths of the same tree: similar to the previous case.

    \end{itemize}
\end{proof}

\subsection{Approximation of the AMP iterations}
\label{subsec:xtzt}

Let us now establish the relationship between AMP iterates $(\bx^t)_t$ and the non-backtracking iterations $(\bz^t)_t$. We  see in the following proposition that the moments of $\bx^t$ can be approximated by the moments of $\bz^t$.

\begin{proposition}
\label{prop:xtzt}
    For each $t \geq 1$ and each $\bm \in \NN^q$, we have that for each $i\in [n]$,
    $$
        \left| \EE (\bx^t_i)^\bm - \EE (\bz^t_i)^\bm \right| = \mathcal{O}\left(\frac{1}{\sqrt{K_n}}\right).
    $$
\end{proposition}
In order to prove this proposition we  need the following structural lemma that
connects $x_i^t(r)$ to $z_i^t(r)$ for $i \in [n]$, $r\in [q]$ and $t\in
\mathbb{N}$. Consider $\bar\cU_i^t$ (resp. $\cU_i^t$) the set of unmarked trees
of the set $\bar\cT_i^t$ (resp. $\cT_i^t$). We can consider that these sets are
constructed by identifying the trees with the same structure and labels.
Denote also by $\cU$ the map that assigns to a tree $T$ its unmarked version
$\hat{T} := \cU(T)$. The two equations in Lemma~\ref{ztree} can be reformulated
as: 
\begin{align*}
    z_{i\to j}^t(r) &= \sum_{\hat{T}\in \cU^t_{i\to j}} W(\hat{T}) \Gamma(\hat{T}, r,t)x(\hat{T}), \\
    z_i^t(r) &= \sum_{\hat{T}\in \cU^t_i} W(\hat{T}) \Gamma(\hat{T},r,t)x(\hat{T}),
\end{align*}
where $W(T)$ and $x(T)$ are invariant with respect to the marking of the tree, and
$$\Gamma(\hat{T}, r, t) := \sum_{T\in \cT^t(r)\ : \ \cU(T)=\hat{T}}\Gamma(T, \balpha, t), \quad \forall \hat{T}\in \cU_i^t.$$
Consider $\cB^t\subset \bar\cU^t$ to be the set of trees $T$ such that for each $(u\to v)\in E(T)$ we have $\ell(u)\neq \ell(v)$, in addition at least one of the following conditions holds,
\begin{itemize}
    \item there exists a backtracking path of length $3$: a path $a\to b\to c\to d$ such that $\ell(a)=\ell(c)$ and $\ell(b) = \ell(d)$,
    \item there exists a backtracking star: $a\to b \to c$ and $a^\prime \to b\to c$ such that $\ell(a)=\ell(a^\prime)=\ell(c)$.
\end{itemize}

\begin{lemma}
\label{lem:xtzt}
For each $t,r,i$ there exists a $\tilde \Gamma (., t, r)$ such that $\tilde \Gamma (T, r, t)= \mathcal{O}(1)$ uniformly in $T$ and 
$$ x_i^t(r) = z_i^t(r) + \sum_{T\in \cB_i^t} W(T) \tilde{\Gamma}(T, r, t)x(T). $$
\end{lemma}

\begin{proof}
    We prove this lemma by induction on $t$. The cases $t=0,1$ are simple, suppose that $t\geq 2$, and that the equation is valid for $t$. Recall the AMP recursion given by,
    $$ x_i^{t+1}(r) = \sum_{\ell = 1}^n W_{i\ell} f_r(x_\ell^t) - \sum_{s=1}^q\sum_{\ell = 1}^n W_{i\ell}W_{\ell i}f_s(x_i^{t-1}) \partial_{x(s)} f_r(x_\ell^t). $$
    Here we omit the dependence of $f$ on $\ell$ and $t$, i.e. $f_r(x_\ell^t, \ell, t) = f_r(x_\ell^t)$. Recall that
    $f_r$ is a multivariate polynomial, so by Taylor's expansion at $z_{\ell \to i}^t$, we can write 
    \begin{equation}
    \begin{split}
        \label{eq:taylor}
        f_r(x_\ell^t) &= f_r(z_{\ell\to i}^t) + \sum_{s\in [q]} \left(x_\ell^t(s) - z_{\ell \to i}^t(s)\right) \partial_{x(s)} f_r(z_{\ell\to i}^t) \\ 
        &+ \sum_{\bs{k} \ : \ k_1+\cdots + k_q\geq 2}\left[\prod_{s=1}^q \frac{\left(x_\ell^t(s) - z_{\ell \to i}^t(s)\right)^{k_s}}{k_s!}\right] D^{\bs{k}} f_r(z_{\ell \to i}),\\
        \end{split}
    \end{equation}
    where for $\bs{k}\in \NN^q$ and $x\in \RR^q$ we denote by $D^{\bs{k}}_x$ the following differential operator $$D^{\bs{k}} g(x) = \frac{\partial^{k_1+\cdots+k_q} g(x)}{\partial x(1)^{k_1}\cdots x(q)^{k_q}}.$$
    Let $e_\ell^t(r) := \sum_{T\in \cB_\ell^t} W(T) \tilde{\Gamma}(T, r, t) x(T)$, by the induction hypothesis we have 
    \begin{align*}
        x_\ell^t(r) &= z_\ell^t(r) + e_\ell^t(r) \\
        &= z_{\ell \to i}^t(r) + z_{\ell, i}^t(r) + e_\ell^t(r),
    \end{align*}
    where we use the notation $z_{\ell, i}^t(r) := W_{\ell i}f_r(z_{i\to \ell}^{t-1})$. Plugging this equation into \eqref{eq:taylor} gives 
    \begin{equation}
        \label{eq:taylor2}
        \begin{split}
        f_r(x_\ell^t) &= f_r(z_{\ell\to i}^t) + \sum_{s\in [q]} \left(z_{\ell, i}^t(s) + e_\ell^t(s)\right) \partial_{x(s)} f_r(z_{\ell\to i}^t) \\ 
        &+ \sum_{k \ : \ k_1+\cdots + k_q\geq 2}\left[\prod_{s=1}^q \frac{\left(z_{\ell, i}^t(s) + e_\ell^t(s)\right)^{k_s}}{k_s!}\right] D^k f_r(z_{\ell \to i}),\\
        \end{split}
    \end{equation}
    Now, multiplying by $W_{i\ell}$ on both sides and summing over $\ell$ gives the following
    \begin{equation}
    \label{sumwf}
    \begin{split}
        \sum_{\ell\in [n]} W_{i\ell} f_r(x_\ell^t) &= z_i^{t+1}(r) + \sum_{\ell\in [n], s\in[q]} W_{i\ell} \left(z_{\ell, i}^t(s) + e_\ell^t(s)\right) \partial_{x(s)} f_r(z_{\ell\to i}^t) \\
        &+ \sum_{\ell\in [n],\ k_1+\cdots + k_q\geq 2}W_{i\ell} \left[\prod_{s=1}^q \frac{\left(z_{\ell, i}^t(s) + e_\ell^t(s)\right)^{k_s}}{k_s!}\right] D^k f_r(z_{\ell \to i}).
    \end{split}
    \end{equation}
    The first term is obtained by the definition of $z_i^{t+1}(r)$, see Eq~\eqref{zi(r)}. The second term can be decomposed
    into the two following sums,
    \begin{equation*}
        \begin{split}
            &\sum_{\ell\in [n], s\in[q]} W_{i\ell} W_{\ell i} f_r(z_{i\to \ell}^{t-1})\partial_{x(s)} f_r(z_{\ell\to i}^t) + \sum_{\ell\in [n], s\in[q]} W_{i\ell} e_\ell^t(s) \partial_{x(s)} f_r(z_{\ell\to i}^t).\\
        \end{split}
    \end{equation*}
    Now subtracting the Onsager term from both sides of Eq~\eqref{sumwf} gives the following 
    \begin{eqnarray}
            x_i^{t+1}(r) &=& z_i^{t+1}(r) - \sum_{\ell\in [n], s\in[q]} W_{i\ell} W_{\ell i} \left(f_r(x_i^{t-1})\partial_{x(s)} f_r(x_\ell^t)-f_r(z_{i\to \ell}^{t-1})\partial_{x(s)} f_r(z_{\ell\to i}^t)\right) \nonumber \\
            &&+ \sum_{\ell\in [n], s\in[q]} W_{i\ell} e_i^t(s) \partial_{x(s)} f_r(z_{\ell\to i}^t) \\
        &&\qquad + \sum_{\ell\in [n],\ k_1+\cdots + k_q\geq 2}W_{i\ell} \left[\prod_{s=1}^q \frac{\left(z_{\ell, i}^t(s) + e_i^t(s)\right)^{k_s}}{k_s!}\right] D^k f_r(z_{\ell \to i})\ .\nonumber
    \end{eqnarray}
    Denote by $S_1$, $S_2$ and $S_3$ respectively, the three terms in the right hand side of the previous equation except $z_i^{t+1}(r)$. One wants to prove that these three terms can be written as sums over trees in $T\in \cB_i^{t+1}$ of terms having the form,
    $$ W(T)\tilde{\Gamma}(T, r, t)x(T), $$
    where $\tilde{\Gamma}(T, r, t)$ is obtained by construction, the exact form of this term is not important, we only need it to be bounded as $n$ goes to infinity.
    \paragraph{The term $S_2$} The second term is given by the following formula,
    $$ S_2 = \sum_{\ell \in [n], \ s\in [q]} W_{i\ell} e_\ell^t(s) \partial_{x(s)} f_r(z_{\ell \to i}^t). $$
    The terms in this sum are given by 
    \begin{align*}
        e_\ell^t(s) &= \sum_{T\in \cB_\ell^t} W(T) \Gamma(T, s, t) x(T), \\
        \partial_{x(s)}f_r(z_{\ell \to i})& = \sum_{k_1+ \cdots+ k_q\leq d} \alpha_{k_1, \cdots, k_q}(r, \ell, t) k_s \left(z_{\ell \to i}^t(s)\right)^{k_s - 1} \prod_{u\in [q]\setminus \{s\}} \left(z_{\ell \to i}^t(u)\right)^{k_u},\\
    \end{align*}
    with 
    $$ z_{\ell \to i}^t(u) = \sum_{T \in \cU_{\ell \to i}^t} W(T) \Gamma(T,u,t) x(T). $$
    $S_2$ can thus be interpreted as a sum over trees $T\in \cB_i^{t+1}$ constructed as follows: 
    \begin{itemize}
        \item The root $\degree$ has a type equal to $i$, and $\degree$ has a child, say $\square$, of type $\ell$. This is due to $W_{i\ell}$.
        \item The vertex $\square$ is the root
        of a tree in $\cB_\ell^t$. This is due to the term $e_\ell^t(s)$.
        \item The root's child $\square$ is also the root of $k_1+\cdots+(k_s-1)+\cdots+k_q$ additional trees in $\cU_{\ell\to i}^t$. This is due to the term $\partial_{x(s)}f_r(z_{\ell \to i})$. Note that in total, $\square$ has at most $d\geq k_1+\cdots+(k_s-1)+\cdots+k_q + 1$ children.
    \end{itemize}
    By construction, we can easily see that $T$ is in $\cB_i^{t+1}$.
    \paragraph{The term $S_1$}
    The first term is given by the following formula,
    $$ S_1 = \sum_{\ell\in [n], s\in[q]} W_{i\ell} W_{\ell i} \left(f_r(x_i^{t-1})\partial_{x(s)} f_r(x_\ell^t)-f_r(z_{i\to \ell}^{t-1})\partial_{x(s)} f_r(z_{\ell\to i}^t)\right) \,.$$
    Doing a Taylor expansion of the polynomial $g:(x, x^\prime) \mapsto f_r(x)\partial_{x(s)} f_r(\bs{x}^\prime)$ around $(z_{i\to \ell}^{t-1},z_{\ell\to i}^t)$ gives 
    $$ S_1 = \sum_{\ell\in [n] s\in [q]}W_{i\ell}W_{\ell i} \sum_{|j|+|k|\geq 1} \left[\prod_{u=1}^q \frac{\left(z_{i, \ell}^{t-1}(u) + e_i^{t-1}(u)\right)^{j_u}\left(z_{\ell, i}^t(u) + e_\ell^t(u)\right)^{k_u}}{(j_u+k_u)!}\right] D^{(j,k)}g(z_{i\to \ell}^{t-1},z_{\ell\to i}^t). $$
    $S_1$ can be seen as a sum, up to multiplication factors, of the following terms
    $$ W_{i\ell} W_{\ell i} \prod_{u=1}\left(z_{i, \ell}^{t-1}(u) + e_i^{t-1}(u)\right)^{j_u}\left(z_{\ell, i}^t(u) + e_\ell^t(u)\right)^{k_u} \left(z_{i\to \ell}^{t-1}(u)\right)^{a_u}\left(z_{\ell\to i}^t(u)\right)^{b_u}, $$
    with the constraint that $\sum_{u =1}^q (j_u + k_u) \geq 1$. To show that $S_1$ can be seen a sum of trees $T$ belonging to $\cB_i^{t+1}$, two cases should be considered, either it exists a $u$ such that $j_u\geq 1$ or $k_u\geq 1$.
    \begin{itemize}
        \item If there exists a $u$ such that $j_u\geq 1$, we construct a tree in $\cB^{t+1}$ as follows: 
        \begin{itemize}
            \item The root $\degree$ has a type equal to $i$, and $\degree$ has a child, say $\square$, of type $\ell$. This is due to $W_{i\ell}$.
            \item The vertex $\square$ is the root of trees in $\cU_{\ell\to i}^t$, which is due to the multiplication by $z_{\ell\to i}^t(u)$.
            \item The vertex $\square$ has a child, say $\lozenge$, of type $i$, which is due to $W_{\ell i}$.
            \item The vertex $\lozenge$ is the root of trees in $\cU_{i\to \ell}^{t-1}$, which is due to $z_{i\to \ell}^{t-1}(u)$.
        \end{itemize}
        Now because $j_u\geq 1$, at least one of the following holds:
        \begin{itemize}
            \item The vertex $\lozenge$ is the root of trees in $\cB_i^{t-1}$, which obviously results in a tree $T\in \cB_i^{t+1}$.
            \item The vertex $\lozenge$ is has a child of type $\ell$, which creates a backtracking path of length $3$ of types $\ell \to i\to \ell \to i$ which also results in a tree $T\in \cB_i^{t+1}$. This child is the root of a tree in $\cU_{i\to \ell}^{t-1}$. And this is due to the term $z_{i,\ell}^{t-1}$.
        \end{itemize}
        \item If there exists a $u$ such that $k_u\geq 1$, we repeat the same argument. This time, the multiplication by $z_{\ell, i}^t(u)$ gives a backtracking star $[i ,\ i \to \ell \to i]$, which results in a tree $T\in \cB_i^{t+1}$. Otherwise, the multiplication by $e_\ell^t(u)$ adds a tree in $\cB_\ell^t$ which obviously results in a final tree $T$ belonging to $\cB_i^{t+1}$.
    \end{itemize}
    \paragraph{The term $S_3$}
    The third term is given by the following formula,
    \begin{equation*}
        S_3 = \sum_{\ell\in [n],\ k_1+\cdots + k_q\geq 2}W_{i\ell} \left[\prod_{s=1}^q \frac{\left(z_{\ell, i}^t(s) + e_i^t(s)\right)^{k_s}}{k_s!}\right] D^k f_r(z_{\ell \to i}).
    \end{equation*}
    Similarly to the interpretation of $S_2$ as a sum of trees in $\cB_i^{t+1}$, we can repeat the same arguments for $S_3$. The terms that have $e_i^t(s)$ as a multiplication factor naturally results in trees belonging to $\cB_i^{t+1}$. In the other case, notice that the constraints $k_1+\cdots+k_q\geq 2$ implies that a term of the form $W_{i\ell} z_{\ell, i}^t(s)z_{\ell, i}^t(s^\prime)$ always exists, this term produces a backtracking star and thus the final tree $T$ belongs to $\cB_i^{t+1}$.

 By studying the tree terms, we proved the existence of a $\tilde{\Gamma}(T,t,t+1)$ such that 
 $$ x_i^{t+1}(r) = z_i^{t+1} + \sum_{T\in \cB_i^{t+1}} W(T) \tilde{\Gamma}(T,r,t+1) x(T). $$
Where $\tilde{\Gamma}(T,t,t+1)$ is a function of $\tilde{\Gamma}(T,t,t)$ and the activation functions' coefficients. It remains to check that $\tilde{\Gamma}(T,t,t+1) = \mathcal{O}(1)$. This can be easily verified, and its proof will be omitted.
\end{proof}

\begin{remark}
    The previous proof is a non-Symmetric adaptation of the techniques developed in \cite{Bayati_2015} and \cite{HACHEM2024104276} in the symmetric case. Instead of terms $W_{i\ell}^2$ in the symmetric case, we handle their counterparts $W_{i\ell}W_{\ell i}$ in the non-Symmetric case and properly interpret them as edges of a tree. Accordingly, we rely on an Onsager term based on matrix $W\odot W^\top$ instead of $W^{\odot 2}$. 
\end{remark}

Finally, we can prove Proposition~\ref{prop:xtzt} by repeating the same arguments used in the proof of Proposition~\ref{prop:ztzttilde}.

\begin{proof}[Proof of Proposition~\ref{prop:xtzt}] 
We can restrict ourselves to the case of $\bm(r) = m$ and $\bm(s)=0$ for $s\neq r$. The $m$-th power of $x_i^t(r)$ is given by 
\begin{align*}
    \EE\left(x_i^t(r)\right)^m - \EE\left(z_i^t(r)\right)^m &= \EE\left(z_i^t(r) + \sum_{T\in \cB_i^t} W(T)\tilde{\Gamma}(T, r, t)x(T) \right)^m - \EE\left(z_i^t(r)\right)^m\\
    &\leq C \sum_{T_1\in \cB_i^t}\sum_{T_2, \cdots, T_m\in \cB_i^t \cup \cT_i^t(r)}\left|\EE\prod_{i=1}^m W(T_i)\right|.
\end{align*}
The key observation here is to notice that the graph obtained by merging the trees $(T_1, \cdots, T_m)$ has an edge which is the result of the fusion of at least three edges, and this is because $T_1$ has a backtracking path or a backtracking star. This implies a bound on the number of edges of the resulting graph.
    
\end{proof}

\subsection{End of proof of Theorem~\ref{thm:amp-vec}}
\label{subsec:xtut}

We  now show that the sequence of Gaussian vectors $(U^t)$ defined in \eqref{eq:MatrixDE} by the Density Evolution equations
approximate the iterations $(\bs{y}^t)$ defined in \eqref{def:yij} and \eqref{def:yi} where the matrices $(W^t)_{t\in
\mathbb{N}}$ are independent and Gaussian.
\begin{proposition}
\label{prop:ytut}
    Let $W$ be a random matrix defined in \eqref{def:W} and satisfying assumptions A-\ref{ass:X} and A-\ref{ass:S}, suppose in addition that $W$ is gaussian. Let $(W^t)_{t\in \mathbb{N}}$ be a sequence of independent copies of $W$. Then for each multi-index $\bm \in \NN^{q}$ and each integer $t>0$ we have
    \begin{equation*}
        \underset{i\in [n]}{\max} \left| \EE\left[(\bs{y}_i^t)^\bm\right] 
  - \EE\left[(U_i^t)^\bm\right]  \right| \longrightarrow 0.
    \end{equation*}
\end{proposition}

\begin{remark}
    Recall that the random matrix $(U^t_1,\cdots, U^t_n)^\top \in \RR^{n\times q}$ is defined such that $(U^t_i)_{i\in [n]}$ are independent and such that $U_i^t \sim \cN(0,Q_i^t)$ where $(Q_i^t)_t$ is a sequence of $k\times k$ covariance matrices defined recursively by
    $$ 
    Q_i^{t+1} = \sum_{\ell \in [n]} s_{i\ell} \EE\left[f(U_\ell^t,\ell, t)f(U_\ell^t,\ell, t)^\top\right]\ . 
    $$
    In particular, the law of $U$ does not depend on our correlation profile.
    
    We also recall that the iterations $y^t$ are defined by $y^0_{i\rightarrow j} = x_i^0$ and
    $$ y_{i\rightarrow j}^{t+1} = \sum_{\ell \in [n] \setminus \{j\}} W_{i\ell}^t f(y_{\ell\rightarrow i}^{t}, \ell, t)$$
    which implies that the conditional distribution of $y_{i\rightarrow j}^{t+1} $ given $\cF_t := \sigma\{W^0,\cdots, W^{t-1}\}$ is $\cN_k\left(0, H_{ij}^{t+1}\right)$ where $(H_{ij}^t)_t$ is a sequence of $q\times q$ covariance matrices defined for each $t\in \NN^\star$ by the following recursion
    $$ H_{ij}^{t+1} = \sum_{\ell\in [n]\setminus \{j\}} s_{i\ell} \EE\left[f(y_{\ell\rightarrow i}^{t}, \ell, t) f(y_{\ell\rightarrow i}^{t}, \ell, t)^\top\right]  \, .
    $$
We therefore notice that the conditional distribution of 
$y^{t+1}_{i\rightarrow j}$ given $\cF_{t}$ is unchanged if we replace the 
matrix $W$ with a random symmetric matrix $\Tilde{W}$ having the same variance 
profile as $W$. By doing so, we can directly apply the result in \cite[Proposition 15]{HACHEM2024104276}.
\end{remark}

Combining the previous results we get the following convergence for each multi-index $\bm$
\begin{equation*}
    \underset{i\in [n]}{\max} \left| \EE\left[(x_i^t)^\bm\right] - \EE\left[(U_i^t)^\bm\right]  \right| \longrightarrow 0.
\end{equation*}
We can then use the triangular inequality to get this same result for any multivariate polynomial with bounded coefficients instead considering only the monomial $X^\bm$.
\begin{proposition}
\label{prop:xtut}
    Let $\psi: \RR^q \times [n] \rightarrow \RR$ such that $\psi(., \ell)$ is a multivariate polynomial with bounded degree and bounded coefficients. Then for each subset $\cS^{(n)}$ of $[n]$ with $|\cS^{(n)}|\rightarrow \infty$, it holds that
    \begin{equation*}
        \frac{1}{|\cS^{(n)}|} \sum_{i\in \cS^{(n)}} \EE\left[\psi(x_i^t, i)\right] - \EE\left[\psi(U_i^t, i)\right] \longrightarrow 0.
    \end{equation*}
\end{proposition}

Finally, in order to get the convergence in probability stated in Theorem~\ref{thm:amp-vec}, we  only need to show that the following variance
\begin{equation}
\label{varconv}
    \mathbb{V}ar\left[\frac{1}{K_n} \sum_{i\in S^{(n)}} \psi(x_i^t, i)\right]\longrightarrow 0
\end{equation}
converges to zero. The proof of this convergence is similar to the proof of \cite[Proposition 17]{HACHEM2024104276} and thus is be omitted.

The proof of Theorem~\ref{thm:amp-vec} follows then from Proposition~\ref{prop:xtut} and the convergence in \eqref{varconv}.

\section{AMP with general activation functions and non-zero diagonal matrix}
\label{sec:generalActiv}

\subsection{AMP for general activation functions}
\label{subsec:generalActiv}

Now that we have proved the AMP convergence result for polynomial activation functions in Theorem~\ref{thm:amp-pscal}, we can generalize this result for non polynomial activation functions by approximation arguments. In other words we  complete the proof of our main Theorem~\ref{thm:main} still assuming that the matrix model has a zero diagonal ($X_{ii}=0$).

We start this section with an approximation of the activation function $h$ by polynomials in order to use the convergence results of polynomial AMP.
\begin{lemma}
\label{lem:polyApprox_e}
    Let $h$ be an activation function satisfying A-\ref{ass:activation_function} and let $\left(Z^1, \cdots, Z^t\right)\sim \DE\left(S,h, \bs{x}^0,  t\right)$. Let $e>0$ be a (small) real number, then there exists a set of functions $(p_e(\cdot,\cdot,t))_{t=1}^{t_{\max}}$ such that for each $\eta\in \cQ_{\eta}$, $p_e(.,\eta, t)$ is a polynomial and 
    \begin{equation*}
         \EE\left(h(Z_i^t, \eta_i, t) - p_e(Z_i^t, \eta_i, t)\right)^2 \leq e
        \quad \text{and} \quad \left|\EE\left( \partial h( Z_i^t, \eta_i, t) - \partial p_e(Z_i^t, \eta_i, t)\right)\right| \leq e,
    \end{equation*}
    for $t=0, \cdots, t_{\max}$ with the convention that $Z^0=\bs{x}^0$ deterministic. In addition, let $\chR_i^{t_{\max}}$ be the covariance matrix of the $i$-th row of $\left(\bs{\cZ}^1, \cdots, \bs{\cZ}^t\right)\sim \DE\left(p_e, \bs{x}^0, S, t\right)$, then there exists $\delta(e)$ such that $\delta(e)\to 0$ when $e\to 0$ and 
    $$ \lVert R_i^{t_{\max}} - \chR_i^{t_{\max}} \rVert \leq \delta(e), \quad \forall i \in [n]. $$
\end{lemma}
In order to prove this lemma, we need to show that the variances of $Z_i^t$ are bounded away from zero. To that end, we use Assumptions A-\ref{ass:initial_point}, A-\ref{ass:activation_function} and A-\ref{ass:nonDegen}.

\begin{lemma}
    \label{lem:boundVariance}
    Let $S$ be a matrix satisfying A-\ref{ass:S}, $\bs{x}^0$ an $n$-dimensional vector satisfying A-\ref{ass:initial_point}, $h$ a function satisfying A-\ref{ass:activation_function} and A-\ref{ass:nonDegen}. Following the notations of Definition~\ref{def:de} let $\left(\bs{Z}^1, \cdots, \bs{Z}^t\right) \sim \DE\left(h, \bs{x}^0, S, t\right)$ and recall the definition of the covariance matrix $R_i^t\in\RR^{t\times t}$. Then for every $t\in \NN$ there exist two constant $C = C(t)>0$ and $c = c(t) >0$ such that 
    \begin{enumerate}
        \item The spectral norms of the covariance matrices are bounded
        $$\forall n\in \NN, \ \forall i\in [n], \quad  \lVert R_i^t \rVert \leq C. $$
        \item The variances of $Z_i^t$ are bounded away from zero
        $$ \forall n\in \NN, \ \forall i\in [n], \quad R_i^t(t,t) \geq c. $$
    \end{enumerate}
\end{lemma}
The proof of this technical lemma is given in Appendix~\ref{app:boundvariance}.
The proof of the first part of Lemma~\ref{lem:polyApprox_e} relies on the
polynomial density Lemma~\ref{lem:polyApprox} and the fact that the variances
of $Z_i^t$ are bounded from above and also bounded away from zero which is
detailed in Lemma~\ref{lem:boundVariance}. The second part uses the same proof
technique described in the proof of Lemma~\ref{lem:RRtilde}.  An immediate
consequence of this approximation is that the covariance matrices
$\chR_i^{t_{\max}}$ are also bounded. 

Let $(\bs{\cx}^t)$ the AMP sequence considered in Theorem~\ref{thm:amp-pscal}. The following lemma allows us to replace the ``random" formulation of the Onsager term by a deterministic equivalent, i.e. 
$$\diag \Bigl(W\odot W^\top \partial p_e(\bs{\cx}^{t},\cdot,t) \Bigr) \qquad \text{with} \qquad \diag \Bigl(V\partial p_e(\bs{\cx}^{t},\cdot,t) \Bigr).$$

\begin{lemma}
\label{lem:Onsagerswitch}
    For each $t\in \mathbb{N}$ there exists a constant $C$ that does not depend on $n$ such that:
    \begin{equation*}
        \mathbb{E}\left[ \left(\sum_{j\in [n]} \left(W_{ij}W_{ji} - V_{ij}\right)\partial p_e(\cx^t_j, \eta_j, t) \right)^4  \right] \leq C/K_n^2 \quad \text{for all } i\in [n].
    \end{equation*}

    where $V_{ij} = \tau_{ij} \sqrt{s_{ij} s_{ji}} = \mathbb{E}\left[W_{ij}W_{ji}\right]$.
\end{lemma}
The proof of this lemma is provided in Appendix~\ref{app:Onsagerswitch}.

The following lemma gives the desired comparison of two sequences $(\bs{x}^t)$ and $(\bs{\cx}^t)$ defined by 
\begin{equation}
\label{eq:2sequences}
    (\bs{x}^t) = \AMPZ\left(X, S, h, \bs{x}^0, \bs{\eta}\right) \quad \text{and} \quad (\bs{\cx}^t) = \AMPW\left(X, S, p_e, \bs{x}^0, \bs{\eta}\right),
\end{equation}
where $p_e$ is the polynomial approximation of the function $h$ by an error margin $e$ in the sense of Lemma~\ref{lem:polyApprox_e}.

\begin{lemma}
\label{lem:ampComparison}
    Fix $t_{\max}>0$. Let $(\bs{x}^t)$ and $(\bs{\cx}^t)$ be two AMP sequences defined as in Eq.~\eqref{eq:2sequences}, then there exists $\delta(e)\to 0$ as $e\to 0$ such that the following holds for each $t=1, \cdots, t_{\max}$,
    \begin{equation*}
        \lVert \bs{x}^t - \bs{\cx}^t \rVert_n \leq \delta(e) + o_{\mathbb{P}}(1)\quad \text{and} \quad \lVert h(\bs{x}^t) - p_e(\bs{\cx}^t) \rVert_n \leq \delta(e) + o_{\mathbb{P}}(1),
    \end{equation*}
    where $o_{\mathbb{P}}(1) \toprobalong 0$.
\end{lemma}

Using this Lemma, we are now able to prove the AMP convergence result for general activation functions.
\begin{proof}[Proof of Theorem~\ref{thm:main} in the zero-diagonal case
]
Let $\varphi : \RR^{t_{\max}} \to \RR$ be a pseudo-Lipschitz function and denote $\bx_i = \left(x_i^1, \cdots, x_i^{t_{\max}}\right)^\top$ and $\bs{\cx}_i = \left(\cx_i^1, \cdots, \cx_i^{t_{\max}}\right)^\top$, without loss of generality we  omit the scalars $\beta_i$ and the parameters $\eta_i$ by considering that $\varphi$ depends also on the index $i$. We have
\begin{equation*}
    \frac{1}{n} \sum_{i\in [n]} \varphi(\bx_i) = \frac{1}{n} \sum_{i\in [n]} \left(\varphi(\bx_i) - \varphi(\bs{\cx}_i)\right) + \frac{1}{n} \sum_{i\in [n]} (\varphi(\bs{\cx}_i) - \varphi(\bs{\cZ}_i))+ \frac{1}{n} \sum_{i\in [n]} (\varphi(\bs{\cZ}_i) - \varphi(\bs{Z}_i)).
\end{equation*}
The pseudo-Lipschitz property of $\varphi$ implies that 
\begin{equation*}
\begin{split}
    \frac{1}{n} \left|\sum_{i\in [n]} \varphi(\bx_i) - \varphi(\bs{\cx}_i)\right| &\leq \frac{C}{n} \sum_{i\in [n]} \lVert \bx_i - \bs{\cx}_i\rVert \left(1 + \lVert \bx_i\rVert  + \lVert \bs{\cx}_i\rVert \right)\\
    &\leq C \left(\sum_{t=1}^{t_{\max}} \lVert \bs{x}^t - \bs{\cx}^t\rVert_n\right)\left(1+\sum_{t=1}^{t_{\max}}\lVert \bs{x}^t\rVert_n + \sum_{t=1}^{t_{\max}}\lVert \bs{\cx}^t\rVert_n\right)\,.\\
\end{split}
\end{equation*}
By Lemma~\ref{lem:ampComparison} we have $\sum_{t=1}^{t_{\max}} \lVert \bs{x}^t - \bs{\cx}^t\rVert_n \leq \delta(e) + o_{\mathbb{P}}(1)$, and by Theorem~\ref{thm:amp-pscal} applied to the test function $x\mapsto x^2$ we get $\lVert\bs{\cx}^t\rVert_n \leq C + o_{\mathbb{P}}(1)$ which also implies that $\lVert \bs{x}^t\rVert_n \leq C + o_{\mathbb{P}}(1)$, finally we have
$$\frac{1}{n} \left|\sum_{i\in [n]}\varphi(\bx_i) - \varphi(\bs{\cx}_i)\right| \leq \delta(e) + o_{\mathbb{P}}(1) .$$
By Theorem~\ref{thm:amp-pscal}, we have that 
$$\frac{1}{n} \sum_{i\in [n]} (\varphi(\bs{\cx}_i) - \varphi(\bs{\cZ}_i)) =o_{\mathbb{P}}(1) \,.$$
And finally by using Lemma~\ref{lem:polyApprox_e} we get 
$$\frac{1}{n} \left|\sum_{i\in [n]} \varphi(\bs{Z}_i) - \varphi(\bs{\cZ}_i) \right|\leq \delta(e), $$
which concludes the proof of our main theorem.
\end{proof}

In order to provide a comparison between the two AMP sequences in \eqref{eq:2sequences}, we need the boundedness of the spectral norm of $W$, a technical yet very important condition. This condition is enforced by A-\ref{ass:nu} that controls the sparsity level of the random matrix.
\begin{proposition}
\label{prop:spectralNormBound}
    Let A-\ref{ass:X}, A-\ref{ass:S} and A-\ref{ass:nu} hold true. Then the following bound holds true with probability one,
    \begin{equation*}
        \sup_{n\ge 1} \lVert W \rVert < \infty.
    \end{equation*}
\end{proposition}
The proof of this proposition is due to a result of \cite{Bandeira_2016} and is provided in Appendix~\ref{app:spectralnormbound}.
In the following paragraph we give the sketch of proof of Lemma~\ref{lem:ampComparison}.

\begin{proof}[Proof of Lemma~\ref{lem:ampComparison}]
    The proof is basically an induction argument in which we use Lemma~\ref{lem:polyApprox_e}, Lemma~\ref{lem:Onsagerswitch} and the AMP convergence result for polynomial activation functions. The base case ($t=1$) is easy. Suppose now that the result is valid for all $s=1, \cdots, t$ and let us prove that it also holds for $s=t+1$. By the triangular inequality, we can write 
    \begin{equation*}
    \begin{split}
        \lVert \bs{x}^{t+1} - \bs{\cx}^{t+1}\rVert_n \leq & \lVert W\rVert \lVert h(\bs{x}^t) - p_e(\bs{\cx}^t)\rVert_n\\ 
        &+ \lVert \diag\left(V \EE\partial h(Z^t)\right)h(\bs{x}^{t-1}) - \diag\left(W\odot W^\top\partial p_e(\bs{\cx}^t)\right)p_e(\bs{\cx}^{t-1})\rVert_n.\\
    \end{split}
    \end{equation*}
The first term is directly handled by the induction hypothesis as well as the bound on the spectral norm of $W$ (see Proposition~\ref{prop:spectralNormBound} ). Let us now show that the second term, which corresponds to the normalized distance between the two Onsager terms, can also be bounded by $\delta(e) + o_{\mathbb{P}}(1)$. Using the triangular inequality, this term is less than $\lVert \Delta^{(1)}\lVert_n+\lVert \Delta^{(2)}\lVert_n+\lVert \Delta^{(3)}\lVert_n+\lVert \Delta^{(4)}\lVert_n$, where 
\begin{equation*}
    \begin{split}
        \Delta^{(1)} &= \diag\left(V\left(\EE\partial h(Z^t) - \EE \partial p_e (\bs{\cZ}^t)\right)\right)h(\bs{x}^{t-1}),\\
        \Delta^{(2)} &= \diag\left(V\EE\partial p_e(\bs{\cZ}^t)\right)\left(h(\bs{x}^{t-1}) - p_e(\bs{\cx}^{t-1})\right),\\
        \Delta^{(3)} &= \diag\left(V\left(\EE\partial p_e(\bs{\cZ}^t)-\partial p_e(\bs{\cx}^t)\right)\right)p_e(\bs{\cx}^{t-1}),\\
        \Delta^{(4)} &= \diag\left((V-W\odot W^\top)\partial p_e(\bs{\cx}^t)\right)p_e(\bs{\cx}^{t-1}).\\
    \end{split}
\end{equation*}

For $\lVert \Delta^{(1)} \rVert_n$. We bound $|[V(\EE\partial h(Z^t) - \EE \partial p_e (\bs{\cZ}^t))]_i|$ by
\begin{equation}
\label{eq:localBoundcZ}
    \left|[V\EE\partial h(Z^t) - V\EE\partial h(\bs{\cZ}^t)]_i\right| + \left|[V\EE\partial h(\bs{\cZ}^t) - V\EE\partial p_e(\bs{\cZ}^t)]_i\right| \quad \leq\quad  Ce + \delta(e),
\end{equation}
where the last inequality is due to Lemma~\ref{lem:polyApprox_e}. The normalized norm of $h(\bs{x}^{t-1})$ can be controlled using the Lipschitz property of $h$ and the result of Lemma~\ref{lem:polyActivation}.

For $\lVert \Delta^{(2)} \rVert_n$. We bound the real numbers $[V \EE\partial p_e(\bs{\cZ}^t)]_i$ using inequality~\eqref{eq:localBoundcZ} and we conclude using the induction hypothesis.

For $\lVert \Delta^{(3)} \rVert_n$. We use Theorem~\ref{thm:amp-pscal}-\eqref{cvg-small} to show that $[V(\EE\partial p_e(\bs{\cZ}^t)-\partial p_e(\bs{\cx}^t))]_i \toprobalong 0$ for any sequence $(i)$ less than $(n)$. We then use the bounds \eqref{mom-cx} to show that $\EE \lVert \Delta^{(3)} \rVert_n \tolong 0.$

For $\lVert \Delta^{(4)} \rVert_n$. Finally, we use Lemma~\ref{lem:Onsagerswitch} to show that $\lVert \Delta^{(4)} \rVert_n \toprobalong 0. $

Using all these bounds we finally get 
\begin{equation}
\label{eq:ind_tplus1}
    \lVert \bs{x}^{t+1} - \bs{\cx}^{t+1}\rVert_n \quad \leq\quad  \delta(e) + o_{\mathbb{P}}(1).
\end{equation}

Now, it remains to show that 
$$ \lVert h(\bs{x}^{t+1}) - p_e(\bs{\cx}^{t+1})\rVert_n \quad \leq\quad  \delta(e) + o_{\mathbb{P}}(1). $$
Using Lipschitz property of $h$ as well as the bound \eqref{eq:ind_tplus1}, we get
$$ 
\lVert h(\bs{x}^{t+1}) - p_e(\bs{\cx}^{t+1})\rVert_n 
\quad \leq\quad 
\delta(e) + o_{\mathbb{P}}(1) + \lVert h(\bs{\cx}^{t+1}) - p_e(\bs{\cx}^{t+1})\rVert_n\ . 
$$
Let $\varphi(x) = (h(x) - p_e(x))^2$ a continuous function with at most polynomial growth at infinity, we write
\begin{equation*}
    \lVert h(\bs{\cx}^{t+1}) - p_e(\bs{\cx}^{t+1})\rVert_n^2 = \frac{1}{n}\sum_{i\in [n]}\left(\varphi(\cx^{t+1}_i) -\EE\varphi(\cZ^{t+1}_i)\right) + \EE\lVert h(\bs{\cZ}^{t+1}) - p_e(\bs{\cZ}^{t+1})\rVert_n^2,
\end{equation*}
by Lemma~\ref{lem:polyActivation} the first term converges to $0$ in probability, and by Lemma~\ref{lem:polyApprox_e} the second term is bounded by $e$.
\end{proof}

\subsection{The non-zero diagonal matrix model}
\label{subsec:nonzeroDiag}

We have been working so far with a matrix $S$ with vanishing diagonal
($S_{ii}=0$), under A-\ref{ass:diag-nulle}. In \cite{HACHEM2024104276} and
\cite{Bayati_2015}, this assumption simplifies the combinatorial derivations
since it prevents the appearance of loops in the combinatorial structures. 

In this section, we lift Assumption A-\ref{ass:diag-nulle} and prove that Theorem \ref{thm:main} holds for random matrices with non zero diagonal elements. We proceed with a perturbation argument.

Consider a matrix $X$ that satisfies A-\ref{ass:X}. Let $S = (s_{ij})_{1\leq i,j \leq n}$ be the variance profile matrix satisfying A-\ref{ass:S} where the diagonal entries $s_{ii}$ are non necessarily zero. Finally, define the matrix $W$ as in Eq.~\ref{def:W}, i.e.
$$ 
W_{ij} = \sqrt{s_{ij}} X_{ij}\, .
$$
Let $\bs{x}^{0}$ and $\bs{\eta}$ two $n$ dimensional vectors satisfying A-\ref{ass:initial_point}, and $h$ a function satisfying A-\ref{ass:activation_function} and A-\ref{ass:nonDegen}. Consider the sequence defined by 
$$
\left(\bs{x}^{t}\right)_{t\in\mathbb{N}} := \AMPZ\left(X, S, h, \bs{x}^0, \bs{\eta}\right)\,.
$$ 
We remind below the iteration expression:
\begin{equation*}
    \bs{x}^{t+1} = W h\left(\bs{x}^t, \bs{\eta}, t\right) - \diag\left(V\mathbb{E}\left[\partial h(\bZ^t,\bs{\eta},t)\right]\right) h(\bs{x}^{t-1},\bs{\eta}, t-1)\,,
\end{equation*}
where $V = (v_{ij}) = (\tau_{ij}\sqrt{s_{ij}s_{ji}})$ and $(Z^1, \cdots, Z^t)\sim \DE\left(h, \bs{x}^0, S, t\right)$.

In order to proceed, define $\Tilde{S}$ to be equal to $S$ except the diagonal elements that we set to zero;
$$ \Tilde{s}_{ij} = (1-\delta_{ij}) s_{ij}\,. $$
Define matrix $\widetilde{W}$ by
$\widetilde{W}_{ij} = \sqrt{\Tilde{s}_{ij}} X_{ij}$, 
and the $\mathbb{R}^n$-valued sequences $\left( \tilde{\bs{x}}^{t}\right)_{t\in\mathbb{N}}$ by 
$$ \left( \tilde{\bs{x}}^{t}\right)_{t\in\mathbb{N}} := \AMPZ\left(X, \widetilde{S}, h, \bs{x}^0, \bs{\eta}\right), $$
where the iterations are given by
\begin{equation*}
     \tilde{\bs{x}}^{t+1} = \widetilde{W} h\left( \tilde{\bs{x}}^t, \bs{\eta}, t\right) - \diag\left(\widetilde{V}\mathbb{E}\left[\partial h(\tilde{Z}^t,\bs{\eta},t)\right]\right) h( \tilde{\bs{x}}^{t-1},\bs{\eta}, t-1)\, .
\end{equation*}
Here $\widetilde{V} = \left(\widetilde{S}\odot \widetilde{S}^\top\right)^{\odot 1/2}\odot T = \left((1-\delta_{ij})v_{ij}\right)$ and $(\tilde{Z}^1, \cdots, \tilde{Z}^t)\sim \DE\left(h, \bs{x}^0, \tilde{S}, t\right)$.

Since this sequence is generated using a matrix model with vanishing diagonal, we can apply the AMP result proven so far, i.e. for every uniformly bounded sequence $(\beta_i)_{i\in [n]}$ and every PL test function $\varphi : \mathbb{R}^{t_{\max}+1}\rightarrow \mathbb{R}$, we have 
\begin{equation*}
    \frac{1}{n} \sum_{i\in [n]} \beta_i \varphi(\eta_i, \Tilde{x}^1_i,\cdots, \Tilde{x}^{t_{\max}}_i) -  \beta_i \varphi(\eta_i, \Tilde{Z}^1_i,\cdots, \Tilde{Z}^{t_{\max}}_i) \quad \toprobalong\quad  0 \,.
\end{equation*}

In order to prove the same convergence result for $(\bs{x}^t)_{t\in \NN}$, we prove that $\bs{x}^t$ is a small perturbation of $ \tilde{\bs{x}}^t$ as $n$ grows to infinity.

\begin{lemma}
\label{lem:RRtilde}
    For each $i\in [n]$ and $t \leq t_{\max}$ recall that $R_i^t$ (respectively $\tilde{R}_i^t$) is the covariance matrix of $\vec Z^{t}_i := [Z_i^1, \cdots, Z_i^t]^\top$ (respectively $\vec {\Tilde{Z^{t}_i}}$). Then $\lVert R_i^t  - \tilde{R}_i^t\rVert$ converges to $0$ as $n$ grows to infinity.
\end{lemma}

\begin{proof}
    We  prove this result by induction on $t$. For $t=1$ we write:
    \[
     R_i^1  - \tilde{R}_i^1 = \sum_{\ell \in [n]} s_{i\ell} \left(h(x_\ell^0,\eta_\ell, 0)\right)^2 - \sum_{\ell \in [n] \ : \ \ell \neq i} s_{i\ell} \left(h(x_\ell^0,\eta_\ell, 0)\right)^2 = s_{ii}\left(h(x_i^0,\eta_i, 0)\right)^2\, .
    \]
    Hence 
    $$ 
    \left|R_i^1  - \tilde{R}_i^1 \right| \ \leq\ \frac{C}{K_n} \quad \tolong\quad  0\,.
    $$
    Suppose now that for all $s\leq t$ the quantity $\lVert R_i^s  - \tilde{R}_i^s\rVert$ converges to zero and let us now prove that
    this convergence also holds at iteration step $t+1$.
    To this end, we must study the $(t+1, s+1)$-th entry of the $(t+1)\times (t+1)$ of the covariance matrices $R_i^{t+1}$ and $\tilde{R}_i^{t+1}$. We have
    \begin{multline}\label{eq:Rapprox}
        R_i^{t+1}(t+1, s+1) - \tilde{R}_i^{t+1}(t+1, s+1) \\
        =\sum_{\ell \in [n] \ : \ \ell \neq i} s_{i\ell} \left(\mathbb{E}\left[h(Z_\ell^t, \eta_\ell, t) h(Z_\ell^s, \eta_\ell, s)\right]-\mathbb{E}\left[h(\tilde{Z}_\ell^t, \eta_\ell, t)h(\tilde{Z}_\ell^s, \eta_\ell, s)\right]\right)\\ + s_{ii}\mathbb{E}\left[h(Z_i^t, \eta_i, t) h(Z_i^s, \eta_i, s)\right]\, .
    \end{multline}
Using the fact that $\mathbb{E}\left[h(Z_i^s, \eta_i, s)^2\right]$ is bounded by a constant that depends only on $t$ and using Cauchy-Schwartz inequality, we have
$$  \left|s_{ii}\mathbb{E}\left[h(Z_i^t, \eta_i, t) h(Z_i^s, \eta_i, s)\right]\right| \leq \frac{C}{K_n}\, . $$

In order to bound the first term of the right hand side of Eq. \eqref{eq:Rapprox}, first notice that since $h$ is Lipschitz then $H:(x_1, x_2) \mapsto h(x_1)h(x_2)$ is PL, i.e. there exists $C>0$ such that 
$$ 
\forall x, y\in \mathbb{R}^2\, \quad \left|H(x)-H(y)\right| \leq C \lVert x-y\rVert_2 \left(1+\lVert x\rVert_2+\lVert y\rVert_2\right)\,.$$

Let $\Sigma^2\in \mathbb{R}^{2\times 2}$ and $\tilde\Sigma^2\in \mathbb{R}^{2\times 2}$ be the covariance matrices of the vectors $Z_\ell^{t,s}=(Z_\ell^t,Z_\ell^s)$ and $\tilde Z_\ell^{t,s}=(\tilde Z_\ell^t,\tilde Z_\ell^s)$ respectively. Then given $\xi\sim \mathcal{N}_2(0,I_2)$ we can write 
\begin{equation*}
    \begin{split}
        \left|\mathbb{E}\left[H(Z_\ell^{t,s}) - H(\tilde Z_\ell^{t,s})\right]\right| &= \left|\mathbb{E}\left[ H(\Sigma \xi) - H(\tilde\Sigma \xi) \right]\right|\\
        &\leq C \lVert\Sigma-\tilde\Sigma \rVert\mathbb{E}\left[\lVert \xi\rVert_2 \left(1+\lVert Z_\ell^{t,s}\rVert_2+\lVert \tilde Z_\ell^{t,s}\rVert_2\right)\right]\,. \\
    \end{split}
\end{equation*}
Using Lemma~\ref{lem:boundVariance} it is easy to see that the factor
$$\mathbb{E}\left[\lVert \xi\rVert_2 \left(1+\lVert Z_\ell^{t,s}\rVert_2+\lVert \tilde Z_\ell^{t,s}\rVert_2\right)\right]\,. $$
is bounded by a constant depending only on $t_{\max}$. Now using the induction hypothesis we obtain the following inequality: 
$$ \lVert\Sigma-\tilde\Sigma \rVert \leq \lVert\Sigma^2-\tilde\Sigma^2 \rVert^{1/2} \leq \lVert R_\ell^t- \tilde{R}^t_\ell \rVert^{1/2} \tolong 0 $$
Here we used the fact that the matrix squared root is $1/2$-Hölder continuous on the set of symmetric positive matrices, the proof in in Appendix~\ref{app:srholdercont}. Note that by A-\ref{ass:S} we have $s_{ij}\leq C_S K_n^{-1}$, plugging this into \eqref{eq:Rapprox} gives the desired result. 
\end{proof}
\begin{remark}
    Notice that we can also specify the convergence rate of $\lVert R_i^t  - \tilde{R}_i^t\rVert$ to $0$. In fact we can show that 
    $$ \lVert R_i^t  - \tilde{R}_i^t\rVert \leq \frac{C}{K_n^{1/2^t}}. $$
\end{remark}
\subsubsection*{Proof of Theorem~\ref{thm:main} in the general case}
We  begin by proving the following convergence by induction on $t$,
\begin{equation}
\label{eq:xxapprox}
    \lVert \bs{x}^t -  \tilde{\bs{x}}^t\rVert_n \toprobalong 0\,. 
\end{equation}
For $t=1$, knowing that the $x_i^0$'s live on a compact $\mathcal{Q}_x$ we get
\begin{equation}
\label{eq:xxtilde}
    \lVert \bs{x}^1 -  \tilde{\bs{x}}^1\rVert_n^2 = \lVert( W- \tilde{W}) h(\bs{x}^0)\rVert_n^2 = \frac{1}{n}\sum_{i=1}^n s_{ii} X_{ii}^2 h(x_i^0)^2\leq \frac{C}{K_n} \left(\frac{\sum_{i=1}^{n}X_{ii}^2}{n}\right)  \,,  
\end{equation}
thus $ \lVert \bs{x}^1 -  \tilde{\bs{x}}^1\rVert_n^2 \toprobalong 0$. Now assume that this holds for all $s\in \{1,\cdots, t\}$ and let us show that it is also satisfied for $t+1$, i.e.
$$ \lVert \bs{x}^{t+1} -  \tilde{\bs{x}}^{t+1}\rVert_n \toprobalong 0\,. $$
Let us write the difference between $\bs{x}^{t+1}$ and $ \tilde{\bs{x}}^{t+1}$,
\begin{multline*}
     \bs{x}^{t+1} -  \tilde{\bs{x}}^{t+1} = Wh(\bs{x}^t) - \tilde{W}h( \tilde{\bs{x}}^t)\\ + \diag\left(V\mathbb{E}\left[\partial h(\bZ^t)\right]\right) h(\bs{x}^{t-1}) - \diag\left(\tilde{V}\mathbb{E}\left[\partial h(\tilde{\bZ}^t)\right]\right) h( \tilde{\bs{x}}^{t-1})\ ,
\end{multline*}
We first show that
\begin{equation*}
    \rVert Wh(\bs{x}^t) - \tilde{W}h( \tilde{\bs{x}}^t)\rVert_n \toprobalong 0\,.
\end{equation*}
We have

\begin{equation}
\label{eq:WHaprox}
     \rVert Wh(\bs{x}^t) - \tilde{W}h( \tilde{\bs{x}}^t)\rVert_n \leq \rVert(W-\tilde{W})h( \tilde{\bs{x}}^t)\lVert_n + \rVert W(h(\bs{x}^t)-h( \tilde{\bs{x}}^t))\lVert_n
\end{equation}
Using the fact that the $\tilde{x}_i^t$ are bounded by a constant $C=C(t)$ independent of $n$ we can directly see that the first term of \eqref{eq:WHaprox} converges to zero. For the second term, we use the bound on $\lVert W\rVert$ (see Proposition~\ref{prop:spectralNormBound}) as well as the Lipschitz property of $h$ and the induction hypothesis.

Now let us study the term
\begin{equation}
\label{eq:onsagerdiff}
 \diag\left(V\mathbb{E}\left[\partial h(\bZ^t)\right]\right) h(\bs{x}^{t-1}) - \diag\left(\tilde{V}\mathbb{E}\left[\partial h(\tilde{\bZ}^t)\right]\right) h( \tilde{\bs{x}}^{t-1})\,.
\end{equation}
This term can be decomposed as follows
\begin{equation*}
\begin{split}
    &\diag\left((V-\tilde{V})\mathbb{E}\left[\partial h(\bZ^t)\right]\right) h(\bs{x}^{t-1})\\ 
    &+ \diag\left(\tilde{V}\mathbb{E}\left[\partial h(\bZ^t)-\partial h(\tilde{\bZ}^t)\right]\right) h(\bs{x}^{t-1})\\
    &+\diag\left(\tilde{V}\mathbb{E}\left[\partial h(\tilde{\bZ}^t)\right]\right) \left(h(\bs{x}^{t-1})-h( \tilde{\bs{x}}^{t-1})\right)\\ &:= \Delta_1 + \Delta_2 + \Delta_3\,.
\end{split}
\end{equation*}
Using the Lipschitz property of $h$ we can bound $\lVert \Delta_3\rVert_n^2$ as follows:
\begin{equation*}
\begin{split}
    \lVert \Delta_3\rVert_n &= \left\lVert \diag\left(\tilde{V}\mathbb{E}\left[\partial h(\tilde{\bZ}^t)\right]\right)\left(h(\bs{x}^{t-1})-h( \tilde{\bs{x}}^{t-1})\right)\right\rVert_n\\ 
    &\leq \left\lVert \diag\left(\tilde{V}\mathbb{E}\left[\partial h(\tilde{\bZ}^t)\right]\right)\right\rVert\lVert h(\bs{x}^{t-1})-h( \tilde{\bs{x}}^{t-1})\rVert_n\\
    &\leq C \underset{j\in [n]}{\max}\left\{\mathbb{E}\left|\partial h(\tilde{\bZ}_j^t)\right|\right\}\lVert \bs{x}^{t-1}- \tilde{\bs{x}}^{t-1}\rVert_n\,.
\end{split}
\end{equation*}
Recall that $\underset{j\in [n]}{\max}\left\{\mathbb{E}\left|\partial h(\tilde{\bZ}_j^t)\right|\right\}$ is bounded by $C=C(t)$, using the induction hypothesis we prove that $\lVert \Delta_3\rVert_n \toprobalong  0$.

In order to bound the first term $\lVert \Delta_1\rVert_n$, notice that $V - \tilde{V}$ is a diagonal matrix whose entries are bounded by $C/K_n$, thus 
\begin{equation*}
    \left\lVert \diag\left((V-\tilde{V})\mathbb{E}\left[\partial h(\bZ^t)\right]\right)\right\rVert \leq \frac{C}{K_n}  \underset{i\in [n]}{\max}\left\{\mathbb{E}\left|\partial h(\bZ_i^t)\right|\right\} = \mathcal{O}\left(\frac{1}{K_n}\right),
\end{equation*}
where the last equality is by the boundness of $\underset{i\in [n]}{\max}\left\{\mathbb{E}\left|\partial h(\bZ_i^t)\right|\right\}$. Now write $$h(\bs{x}^{t-1}) = \left(h(\bs{x}^{t-1})-h( \tilde{\bs{x}}^{t-1})\right) + h( \tilde{\bs{x}}^{t-1}),$$ by the induction hypothesis we clearly see that $\lVert h(\bs{x}^{t-1})-h( \tilde{\bs{x}}^{t-1})\rVert_n \toprobalong 0$, in addition we know that $\lVert h( \tilde{\bs{x}}^{t-1})\rVert_n^2 - \mathbb{E}\left\lVert h(\tilde{Z}^{t-1})\right\rVert_n^2 \toprobalong 0$ so by bounding $\mathbb{E}\left\lVert h(\tilde{Z}^{t-1})\right\rVert_n^2$ we get that the probability of $\lVert h(\bs{x}^{t-1})\rVert_n$ not being bounded converges to $0$. Finally $\lVert \Delta_1\rVert_n \toprobalong 0$.

For $\lVert \Delta_2\rVert_n$, we use Lemma~\ref{lem:RRtilde} to bound $\left\lVert \diag\left(\tilde{V}\mathbb{E}\left[\partial h(\bZ^t)-\partial h(\tilde{\bZ}^t)\right]\right)\right\rVert$ by $C/K_n$ and finally get $\lVert \Delta_2\rVert_n \toprobalong 0$. To sum up, we have proved that the difference between the two Onsager terms \eqref{eq:onsagerdiff} has a normalized norm converging to $0$. Finally, we have proved \eqref{eq:xxtilde} by induction, i.e. $ \tilde{\bs{x}}^t$ asymptotically approximates $\bs{x}^t$ in terms of normalized norm. Now we are able to use the convergence result of the sequence $( \tilde{\bs{x}}^t)_t$ to prove the convergence of $ \tilde{\bs{x}}^t$ as $n$ grows to $\infty$. Let $\varphi : \RR^{t_{\max}} \to \RR$ be a pseudo-Lipschitz function and denote $\bx_i = \left(x_i^1, \cdots, x_i^{t_{\max}}\right)^\top$ and $\tilde{\bx}_i = \left(\tilde{x}_i^1, \cdots, \tilde{x}_i^{t_{\max}}\right)^\top$, and without loss of generality we  omit the scalars $\beta_i$ and the parameters $\eta_i$ by considering that $\varphi$ depends also on the index $i$. We have 
\begin{equation*}
\begin{split}
     \frac{1}{n} \left|\sum_{i=1}^{n}\varphi(\bx_i) -\varphi(\vec {Z^{t}_i}) \right| &\leq \frac{1}{n} \left|\sum_{i=1}^{n}\varphi(\bx_i) -\varphi(\tilde{\bx}_i) \right| + \frac{1}{n} \left|\sum_{i=1}^{n}\varphi(\tilde{\bx}_i) -\varphi(\vec {\Tilde{Z^{t}_i}}) \right| + \frac{1}{n} \left|\sum_{i=1}^{n}\varphi(\vec {\Tilde{Z^{t}_i}}) -\varphi(\vec {Z^{t}_i}) \right|\\
     &=: \Theta_1 + \Theta_2 + \Theta_3\,. \\
\end{split}
\end{equation*}

Using the pseudo-Lipschitz property of $\varphi$ we get the following 
\begin{equation*}
    \begin{split}
        \Theta_1 &\leq \frac{C}{n}\sum_{i=1}^n \lVert \bx_i -\tilde{\bx}_i\rVert  \left(1+\lVert \bx_i\rVert+\lVert\tilde{\bx}_i\rVert\right)  \\
        &\leq \frac{C}{n}\left(\sum_{t=1}^{t_{\max}} \lVert \bs{x}^t - \tilde{\bs{x}}^t\rVert^2 \right)^{\frac{1}{2}} \left(\sum_{i=1}^n\left(1+\lVert \bx_i\rVert+\lVert\tilde{\bx}_i\rVert\right)^2\right)^{\frac{1}{2}}  \\
        &\leq \frac{C}{n}\left(\sum_{t=1}^{t_{\max}} \lVert \bs{x}^t - \tilde{\bs{x}}^t\rVert \right) \left(n+\sum_{t=1}^{t_{\max}}\lVert \bs{x}^t\rVert^2+\lVert \tilde{\bs{x}}^t\rVert^2\right)^{\frac{1}{2}}  \\
        &\leq C\left(\sum_{t=1}^{t_{\max}} \lVert \bs{x}^t - \tilde{\bs{x}}^t\rVert_n \right) \left(1+\sum_{t=1}^{t_{\max}}\lVert \bs{x}^t\rVert_n+\sum_{t=1}^{t_{\max}}\lVert \tilde{\bs{x}}^t\rVert_n\right)\,.  \\
    \end{split}
\end{equation*}
Then, by using \eqref{eq:xxtilde} we get $\Theta_1 \toprobalong 0$. The term $\Theta_2$ converges to $0$ in probability by Theorem~\ref{thm:main} applied with zero diagonal matrix model. As for $\Theta_3$ we use the pseudo-Lipschitz property of $\varphi$ as well as Lemma~\ref{lem:RRtilde}. This ends the proof for Theorem~\ref{thm:main}.

\appendix

\section{Proof of Theorem~\ref{th:noncentered}}
\label{app:noncenteredproof}
We prove here the AMP result for non-centered matrices described in Theorem~\ref{th:noncentered}.

We follow the general idea described in \cite{feng2022unifying}, which is to reduce the problem to an AMP with centered random matrix model and apply Theorem~\ref{thm:main}. To this end, write the following,
\begin{align*}
    \bs{x}^{t+1} &= \lambda\left\langle\bv,h_t(\bs{x}^t, \bs{\eta})  \right\rangle\bu+ W h_t(\bs{x}^t, \bs{\eta}) - \diag\left(V \EE\partial h_t(Z^t+\mu_t \bu, \bs{\eta})\right) h_{t-1}(\bs{x}^{t-1}, \bs{\eta})\\
    &= \mu_{t+1}\bu+ W h_t(\bs{x}^t, \bs{\eta}) - \diag\left(V \EE\partial h_t(Z^t+\mu_t \bu, \bs{\eta})\right) h_{t-1}(\bs{x}^{t-1}, \bs{\eta}) + \delta_{t+1}\bu,
\end{align*}
where $\delta_t:= \lambda\left\langle\bv,h_{t-1}(\bs{x}^{t-1}, \bs{\eta})  \right\rangle - \mu_t$. One should think of $\delta_{t+1}\bu$ as an error term, we will show later that this term has a negligible effect. Define now the following sequence $\left(\bs{\tilde{y}}^t\right)_{t\in\NN}$ as follows,
\begin{equation*}
    \bs{\tilde{y}}^0 = \bx^0 \quad \text{and} \quad \bs{\tilde{y}}^t := \bs{x}^t- \mu_t\bu \quad \text{for } t\geq 1,
\end{equation*}
this sequence satisfies the following recursion,
\begin{align}
\label{eq:ytildeproof}
    \bs{\tilde{y}}^{t+1}
    &= W g_t(\bs{\tilde{y}}^t, \bv, \bs{\eta}) - \diag\left(V \EE\partial g_t(\bZ^t, \bv, \bs{\eta})\right) g_{t-1}(\bs{\tilde{y}}^{t-1}, \bv,\bs{\eta}) + \delta_{t+1}\bv,
\end{align}
where the function $g_t(x,v, \eta)$ with parameters $v$ and $\eta$ is given by,
\begin{equation*}
    g_t(x,v, \eta) := h_t(x+\lambda v, \eta)\quad \forall x\in \RR.
\end{equation*}
One can clearly see that this function satisfies the same assumptions as $h_t$. Now define the following AMP algorithm $(\bs{y}^t)_{t\in\NN}$ by 
\begin{equation}\label{eq:yampproof}
    \begin{cases}
        \bs{y}^0 &= \bx^0,\\
    \bs{y}^{t+1}
    &= W g_t(\bs{y}^t, \bv, \bs{\eta}) - \diag\left(V \EE\partial g_t(\bZ^t, \bv, \bs{\eta})\right) g_{t-1}(\bs{y}^{t-1}, \bv,\bs{\eta})\ ,
    \end{cases}
\end{equation}
where 
\begin{equation*}
    \left(\bZ^1, \cdots, \bZ^t\right) \sim \widetilde{\DE}\left(h, \bx^0, S, t,\bu, \bv\right)\,,
\end{equation*}
in the sense of Definition \ref{def:noncenteredde}. A key observation is that 
\begin{equation*}
    \left(\bZ^1, \cdots, \bZ^t\right) \sim \DE\left(g, \bx^0, S, t\right)\, .
\end{equation*}
Hence Theorem~\ref{thm:main} applies for the recursion \eqref{eq:yampproof} and yields that for any pseudo-Lipschitz test function $\varphi :\RR^{t+1} \to \RR$ it holds that
    \begin{equation}
    \label{eq:ampresulty}
        \frac{1}{n} \sum_{i=1}^n \beta_i \varphi\left(\eta_i, y_i^1, \cdots, y_i^t\right) - \beta_i\EE\left[\varphi\left(\eta_i, Z_i^1, \cdots, Z_i^t\right)\right] \toprobalong 0\, .
    \end{equation}

In order to prove our result, it suffices to show that the error term $\delta_{t+1}\bu$ in Eq.~\eqref{eq:ytildeproof} is negligible and that for all $t$ one has $\bs{y}^t\approx \bs{\tilde{y}}^t$. To this end, we want to prove by induction on $t$ that,
\begin{equation}
\label{eq:induction_centered}
    \delta_t \toprobalong 0 \quad \text{and} \quad \lVert\bs{\tilde{y}}^t-\bs{y}^t \rVert_n \toprobalong 0, \quad \text{for all }\, t\geq 1.
\end{equation}
For $t=1$, we have $\delta_1 = 0$ and $\bs{\tilde{y}}^1 = \bs{y}^1$. Suppose that \eqref{eq:induction_centered} is true for $t$, and let us prove that this remains true for $t+1$ as well. Let us begin with $\delta_{t+1}$. We have the following
\begin{align*}
    \delta_{t+1} &= \lambda \sum_{i\in [n]}v_i \left(g_t(\tilde{y}^t_i) - \EE g_t(Z_i^t)\right)\\
    &= \lambda \sum_{i\in [n]}v_i \left(g_t(\tilde{y}^t_i) - g_t(y_i^t)\right) +   \lambda \sum_{i\in [n]}v_i \left(g_t(y^t_i) - \EE g_t(Z_i^t)\right)\\
    &:= T_1 + T_2\,.
\end{align*}
Using the Lipschitz property of the function $g_t$ as well as the induction hypothesis, namely, $\lVert\bs{\tilde{y}^t}-\bs{y}^t\rVert_n\toprobalong 0$ we directly get that $T_1 \toprobalong 0$. As for the second term, $T_2\toprobalong 0$ is a direct application of Theorem~\ref{thm:main}, i.e. Eq.~\eqref{eq:ampresulty}.

It remains to show that $\lVert\bs{\tilde{y}^{t+1}}-\bs{y}^{t+1}\rVert_n\toprobalong 0$. Using the recursive definition of $\left(\tilde{\bs{y}}^t\right)_t$ and $\left(\bs{y}^t\right)_t$ in \eqref{eq:ytildeproof} and \eqref{eq:yampproof} we can write the following;
\begin{align*}
    \bs{\tilde{y}^{t+1}}-\bs{y}^{t+1} = W\left(g_t(\bs{\tilde{y}}) - g_t(\bs{y}^t)\right) - \diag\left(V\EE\partial g_t(\bZ^t)\right)\left(g_{t-1}(\bs{\tilde{y}^{t-1}})-g_{t-1}(\bs{y^{t-1}})\right) + \delta_{t+1} \bu\,.
\end{align*}
The normalized norm of the first term can be easily handled using the Lipschitz property of the function $g_t$ as well as the induction hypothesis, we also use Proposition~\ref{prop:spectralNormBound} which ensures the boundness of the spectral norm $\lVert W\rVert$. As for the second term, we similarly show that the quantity $\lVert g_{t-1}(\tilde{\bs{y}}^{t-1})-g_{t-1}(\bs{y}^{t-1})  \rVert_n$ vanishes, in probability. It remains to show that $\lVert \diag\left(V\EE\partial g_t(\bs{Z}^t)\right)\rVert$ is bounded as $n$ goes to infinity, this clearly holds as $\partial g_t$ is the derivative of a Lipschitz function and thus is bounded.

Finally, we have proved that $\lVert\bs{\tilde{y}^{t+1}}-\bs{y}^{t+1}\rVert_n \toprobalong 0$ which ends the induction argument. Using \eqref{eq:induction_centered} and the AMP result of the sequence $\left(\bs{y}^t\right)_t$ we directly deduce an AMP result of the sequence $\left(\tilde{\bs{y}}^t\right)_t$.

\section{Elements of proof of Lemma~\ref{lem:polyActivation}}
\begin{lemma}
\label{lem:momentsTogeneral}
    Let $(m_n)$ and $ (\sigma_n^2)$ be two bounded sequences and let $(\nu_n)$ be the sequence of Gaussian measures with means $m_n$ and variances $\sigma_n^2$. Let $(\mu_n)$ be any sequence of probability measures such that the following holds for each $k\in \NN$, 
    \begin{equation}
    \label{eq:convPhiAss}
        \int x^k d \mu_n - \int x^k d \nu_n \tolong 0\,.
    \end{equation}
    Then for any continuous function $\psi : \RR \to \RR$ such that $|\varphi(x)| \leq C(1+ |x|^m) $ for some constant $C>0$ and some integer $m$ we have 
    \begin{equation}
    \label{eq:convPhi}
        \int \psi(x) d \mu_n - \int \psi(x) d \nu_n \tolong 0\,.
    \end{equation}
\end{lemma}

\begin{proof}
    First, it is sufficient to show that from any subsequence of $(n)$ we can extract a further subsequence such that the convergence in \eqref{eq:convPhi} holds along this subsequence. So without loss of generality we  only prove that if \eqref{eq:convPhiAss} holds along the sequence $(n)$ then there exists a subsequence of $(n)$ along which \eqref{eq:convPhi} holds.

    The sequence of probability measures $(\nu_n)$ is tight because $(m_n)$ and $(\sigma_n^2)$ are bounded, thus we can extract a subsequence of $(n)$, which  also be denoted as $(n)$, such that $(\nu_n)$ converges weakly to a probability measure $\nu$. Consider now the moment generating function $\Phi_{\nu_n}$ of $\nu_n$ defined on $\RR$ as follows,
    $$ \Phi_{\nu_n}(t) = \int e^{tx} d \nu_n(x) = \exp(m_n t + \sigma_n^2 t^2 / 2), \quad t\in\RR. $$
    This function can be viewed as a restriction to the real line of the following holomorphic function 
    $$ \Phi_{\nu_n}(z) = \int e^{zx} d \nu_n(x) = \exp(m_n z + \sigma_n^2 z^2 / 2), \quad z\in\CC. $$
    Notice that the sequence $(\Phi_{\nu_n})$ is uniformly bounded on compact sets of $\CC$, thus there exists a holomorphic function $\Phi$ and a subsequence of $(n)$ such that $(\Phi_{\nu_n})$ converges uniformly to $\Phi$ on compact sets. This implies the pointwise convergence of the moment generating function $(\Phi_{\nu_n}(t))$ to $\Phi(t)$ so by a convergence result in \cite[Theorem 3]{curtiss} and the uniqueness of the weak limit, we get $\Phi(t) = \Phi_\nu(t)$. The convergence of $(\Phi_{\nu_n}(t))$ to $\Phi_\nu(t)$ implies the convergence of the moments, and by \eqref{eq:convPhiAss} we get 
    \begin{equation}
    \label{eq:momentConvProof}
        \int x^k d \mu_n \tolong \int x^k d \nu,
    \end{equation}
    we also know that $\Phi_\nu$ characterizes $\nu$ \cite[Theorem 1]{curtiss}, thus $\nu$ is determined by its moments, so $(\mu_n)$ converges weakly to $\nu$. Let $\psi$ be a function as in the lemma and let $X_n$ and $X$ be random variables with distributions $\mu_n$ and $\nu$ respectively, we want to prove that $\EE[\psi(X_n)] \tolong \EE[\psi(X)]$, this follows from the convergence in distribution of $(\psi(X_n))$ to $\psi(X)$ and the uniform integrability of $(\psi(X_n))$. The latter is due the following observation 
    $$\sup_{n\in\NN}\EE\left[(\psi(X_n))^2\right] \leq C^2 \sup_{n\in\NN}\EE\left[(1+ |X_n|^m)^2\right] = C^2 \sup_{n\in\NN}\int (1+|x|^m)^2 d\mu_n(x) < \infty\,. $$
    The last inequality is due to the convergence of the moments \eqref{eq:momentConvProof}. 
\end{proof}

\begin{remark}
    Results of Lemma~\ref{lem:momentsTogeneral} can be extended to probability measures $\mu$ on $\RR^d$ by Cramér–Wold theorem, i.e. considering the push-forward probability measure $\mu_t$ by the map $x\mapsto \langle x, t \rangle$ for each $t\in \RR^d$.
\end{remark}
\begin{remark}
    We can also extend Lemma~\ref{lem:momentsTogeneral} to the case where $(\mu_n)$ and $(\nu_n)$ are sequences of random probability measure and where we replace both two convergence statements by convergence in probability formulations. The proof follows from the subsequence criterion \cite[Lemma 3.2]{kallenberg2002foundations}.
\end{remark}

\section{Polynomial approximation }
\label{app:polyApprox}
The following lemma states a basic density result of polynomial functions in the Hilbert space $L^2(\mu)$ where $\mu$ is a Gaussian measure. The polynomial approximation is shown to hold uniformly on certain sets of Gaussian measures $(\mu_{\sigma})_{\sigma\in S}$.

\begin{lemma}{\cite{HACHEM2024104276}}
\label{lem:polyApprox}
    Let $\cQ\subset \RR$ a compact set and $h:\RR \times \cQ \to \RR$ a function satisfying the following properties. (i) There exists a fixed number $L>0$ such that uniformly in $\eta\in \cQ$, 
    $$
    |h(x,\eta) - h(y,\eta)|\quad \le \quad L| x- y|\,,\qquad \forall (x,y)\in \RR^2\, .
    $$
    (ii) There exists a continuous non-decreasing function $\kappa:\RR^+\to \RR^+$ with $\kappa(0)=0$ such that $$
    |h(x,\eta) - h(x,\eta')| \quad \le \quad \kappa(|\eta- \eta'|) \left( 1+|x|\right)\,,\qquad \forall x\in \RR\,,\ \forall (\eta, \eta')\in \cQ^2\, .
    $$
    Let $0<\sigma_{\min} \leq \sigma_{\max}$ and $\varepsilon > 0$ be fixed, and $\xi\sim{\mathcal N}(0,1)$. 
    
    There exists a function $g_\varepsilon : \RR \times \cQ \to \RR$ such that for every $\eta \in \cQ$, $x\mapsto g_\varepsilon(x, \eta)$ is a polynomial, and uniformly in $\eta\in \cQ$ and $\sigma\in [\sigma_{\min}, \sigma_{\max}]$,
    \begin{equation*}
         \EE\left(h(\sigma\xi, \eta) - g_\varepsilon( \sigma\xi, \eta)\right)^2 \leq \varepsilon
        \quad \text{and} \quad \left|\EE \,\partial_x h( \sigma\xi, \eta) - \EE \,\partial_x g_\varepsilon( \sigma\xi, \eta)\right| \leq \varepsilon\ .
    \end{equation*}
\end{lemma}

\begin{proof} Let $\delta>0$ and consider a $\delta$-covering of the compact set $\cQ$ with balls centered in $\{\eta_k\}_{k\in [K]}$.
Fix $k\in [K]$ and consider the function $x\mapsto h(x, \eta_k)$. By the density of polynomials in the space $L^2(\cN(0, \sigma_{\max}^2))$, there exists a polynomial $x\mapsto g_\varepsilon(x, \eta_k)$ such that 
$$ \EE\left(h(\sigma_{\max}\xi, \eta_k) - g_\varepsilon(\sigma_{\max}\xi, \eta_k)\right)^2 \quad \leq\quad  \frac{\varepsilon}4\, .
$$
Let $\eta\in \cQ$ and $\eta_k$ such that $|\eta - \eta_k|\leq \delta$ and put $g_{\varepsilon}(x, \eta) := g_{\varepsilon}(x, \eta_k)$ for such $\eta$. By the properties of function $h$, we have
\begin{eqnarray*}
    \EE\left(h(\sigma_{\max}\xi, \eta) - g_\varepsilon(\sigma_{\max}\xi, \eta)\right)^2 &\leq& 2\EE\left(h(\sigma_{\max}\xi, \eta) - h(\sigma_{\max}\xi, \eta_k)\right)^2 \\&&\quad + 2\EE\left(h(\sigma_{\max}\xi, \eta_k) - g_\varepsilon(\sigma_{\max}\xi, \eta_k)\right)^2\,, \\
    &\leq& 2L^2 \kappa(\delta)^2 \EE\left(1+\sigma_{\max}|\xi|\right)^2 + \frac{\varepsilon}2\,.
\end{eqnarray*}
Using the properties of $\kappa$ we can choose $\delta>0$ small enough so that 
$$ 
\EE\left(h(\sigma_{\max}\xi, \eta) - g_\varepsilon(\sigma_{\max}\xi, \eta)\right)^2 \leq \varepsilon\, .
$$
Let $\sigma \in [\sigma_{\min}, \sigma_{\max}]$, denote $\varphi(x):=h(x, \eta) - g_\varepsilon(x, \eta)$. A change of variable yields 
$$ \EE\varphi(\sigma \xi)^2 \quad \leq \quad \frac{\sigma_{\max}}{\sigma_{\min}} \EE\varphi(\sigma_{\max} \xi)^2 \quad \leq \quad \frac{\sigma_{\max}}{\sigma_{\min}} \varepsilon\,.$$
By Stein's integration by parts lemma we also have 
$$
    \left|\EE\varphi^\prime(\sigma \xi)\right| \quad =\quad  \frac{1}{\sigma} \EE[\xi \varphi(\sigma \xi)]\quad 
    \leq \quad \frac{1}{\sigma_{\min}} \sqrt{\EE\varphi(\sigma \xi)^2} \quad \leq\quad  \sqrt{\frac{\sigma_{\max}}{(\sigma_{\min})^3}} \sqrt{\varepsilon}\, ,
$$
which concludes the proof.
\end{proof}

\section{Proof of Lemma~\ref{lem:Onsagerswitch}}
\label{app:Onsagerswitch}
\begin{proof}[Proof of Lemma~\ref{lem:Onsagerswitch}]
In this proof, we use the framework introduced in Section~\ref{subsec:treeStructure}.
Let us put $p_j := \partial p(\cx^t_j, \eta_j, t)$ as a simplification of the notations, the expectation can be developed as follows,

\begin{equation*}
    \begin{split}
       \mathbb{E}\left[ \left(\sum_{j\in [n]} \left(W_{ij}W_{ji} - V_{ij}\right)p_j \right)^4 \right] &= \sum_{j_1, j_2, j_3, j_4 \in [n]} \mathbb{E}\left[\left(\prod_{\ell = 1}^4\left(W_{ij_\ell}W_{j_\ell i}-V_{ij_\ell}\right)\right)p_{j_1}p_{j_2}p_{j_3}p_{j_4}\right] \\
       &:= \sum_{j_1, j_2, j_3, j_4 \in [n]} \mathbb{E}\varphi(j_1, j_2, j_3, j_4)\,, \\
    \end{split}
\end{equation*}
with $p_j$ having the following form 
$$ p_j = \sum_{\ell = 0}^{d-1} (1+\ell) \alpha_\ell (j,t) \left(\cx_j^t\right)^\ell, $$
notice now that by using Lemma~\ref{lem:xtzt}, we can easily see $p_j$ as a sum over unmarked trees with root type $j$, with depth at most $t$ and with each vertex having at most $d-1$ children, the weight of the trees (i.e. the terms $W(T)$, $\Tilde\Gamma(T)$ and $x(T)$) are the same as in Lemma~\ref{lem:xtzt}.
$$ p_j = \sum_{T \in \bar \cU_j^t} W(T)\Tilde{\Gamma}(T)x(T)\,. $$
Thus, the quantity $\varphi(j_1, j_2, j_3, j_4)$ above can be written as a sum over trees as follows: 
\begin{equation}
\label{eq:sum_trees}
\begin{split}
    \varphi(j_1, j_2, j_3, j_4) &=\sum_{\substack{(T_1, T_2, T_3, T_4)\in \\ \bar \cU_{j_1}^t \times\bar \cU_{j_2}^t \times\bar \cU_{j_3}^t \times\bar \cU_{j_4}^t}} \psi(T_1, T_2, T_3, T_4), \\
    \psi(T_1, T_2, T_3, T_4) & := \prod_{\ell=1}^4 \left(W_{ij_\ell}W_{j_\ell i} - V_{ij_\ell}\right)W(T_\ell)\Tilde\Gamma(T_\ell) x(T_\ell)\,. \\
\end{split}
\end{equation}
In the case where $j_1$, $j_2$, $j_3$ and $j_4$ are distinct, the above sum can interpreted as a sum over trees having the structure described in Figure~\ref{fig:proof_tree_structure}.
\begin{figure}[!ht]
    \centering
    \includegraphics[width=0.5\linewidth]{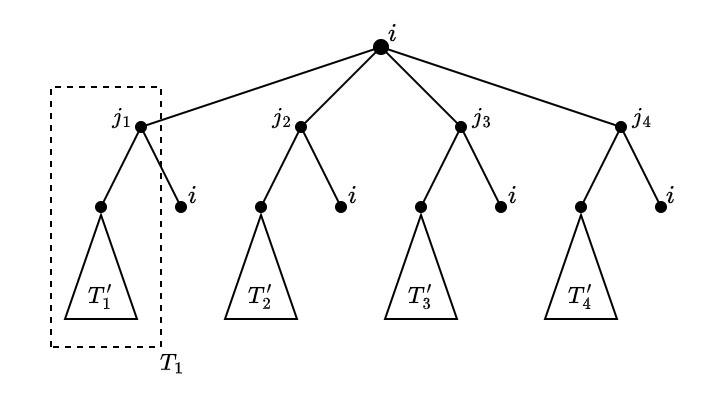}
    \caption{Tree structure.}
    \label{fig:proof_tree_structure}
\end{figure}
these are trees having a root of type $i$, this root has four children of types $j_1$, $j_2$, $j_3$ and $j_4$, each one of these four vertices has a child of type $i$ and is also the planted root of a tree of length $t-1$. Let us denote by $\cS_i$ the set of all these trees. Let $T \in \cS_i$ a tree parameterized by $(T_1, T_2, T_3, T_4)\in  \bar \cU_{j_1}^t \times\bar \cU_{j_2}^t \times\bar \cU_{j_3}^t \times\bar \cU_{j_4}^t$ and let $\mu$ be the number of edges of $T$, i.e.
$$ \mu = 8 + \sum_{\ell = 1}^4 |E(T_\ell)|. $$
Following the proof of Proposition~\ref{prop:ztzttilde}, we know that
\begin{equation}
\label{eq:bound_Epsi}
    \left|\EE\psi(T_1, T_2, T_3, T_4)\right| \leq C K_n^{-\mu/2},
\end{equation}
Let us now compute the number of non vanishing contributions in $\varphi(j_1, j_2, j_3, j_4)$. A term $\EE\psi(T_1, T_2, T_3, T_4)$ vanishes if there exists an $\ell = 1, 2, 3, 4$ such that neither the edge $(i\to j_\ell)$ nor $(j_\ell \to i)$ belongs to set of edges of the trees $T_1, \cdots, T_4$ or if there exists another edge in $T_1, \cdots, T_4$ which occurs once, in other words, if we consider the graph $G$ obtained by identifying the vertices of the same type in $T$ then $T$ has a non vanishing contribution if all the edges are covered in $G$ at least twice and the edges $\{(i,j_\ell) \ | \ \ell = 1, \cdots, 4\}$ at least three times, then:
$$ \mu \geq 2 \left(|E(G)| - 4\right) + 3\times 4 = 2 |E(G)| + 4. $$
Notice that $G$ is a connected graph (there exists a path from any vertex of $G$ to $i$), then
$$ |V(G)| \leq |E(G)| + 1 \leq \mu / 2 - 1. $$
The vertices except $\{i, j_1, j_2, j_3, j_4\}$ can have arbitrary types from a set of at most $CK_n$ types, so we get 
$$ |\EE \varphi(j_1, j_2, j_3, j_4)| \leq C K_n^{-\mu/2} K_n^{\mu/2 - 1 - 5} = C K_n^{-6}, $$
In addition, we have $\begin{pmatrix}
    K_n \\ 4
\end{pmatrix}\leq CK_n^4$ choices for quadruples $(j_1, j_2, j_3, j_4)$ with distinct elements, this means that 
$$ \sum_{\substack{j_1, j_2, j_3, j_4 \in [n] \\ \text{distinct}} } |\mathbb{E}\varphi(j_1, j_2, j_3, j_4)| \leq CK_n^{-2}. $$
A similar argument can be used to analyze the other cases where $j_1, j_2, j_3, j_4$ are not necessarily distinct.

\end{proof}

\section{Proof of Proposition~\ref{prop:spectralNormBound}}
\label{app:spectralnormbound}
We begin by decoupling the entries of our random matrix $W$ using triangular inequality twice  
$$ \left(\EE\lVert W\rVert^p\right)^{1/p} \leq \left(\EE(\lVert U\rVert + \lVert L\rVert)^p\right)^{1/p} \leq \left(\EE\lVert U\rVert^p\right)^{1/p} + \left(\EE\lVert L\rVert^p\right)^{1/p} , $$
where $U$ and $L$ are $n\times n$ triangular matrices corresponding to the upper part (including diagonal) and lower part of $W$ respectively. Notice that $U$ can be seen as an $n\times n$ random matrix with independent entries having the following variance profile 
$$ s^{u}_{ij} = \left\{
    \begin{array}{ll}
        s_{ij} & \mbox{if} \ i\leq j \\
        0 & \mbox{otherwise.}
    \end{array}
\right. $$
Following the notations of \cite{Bandeira_2016} we define
\begin{equation*}
    \sigma_1 = \max_{i} \left(\sum_{j\geq i} s_{ij}\right)^{1/2}, \quad \sigma_2 = \max_{j} \left(\sum_{i\leq j} s_{ij}\right)^{1/2}, \quad \sigma_\ast = \max_{i\leq j} \sqrt{s_{ij}}\,.
\end{equation*}
Now using the results of \cite{Bandeira_2016} we get 
\begin{equation*}
\begin{split}
    \left(\EE \lVert U\rVert^{2\log(n)}\right)^{1/2\log(n)} &\lesssim \sigma_1 + \sigma_2 + \sigma_\ast (\log(n))^{(\rho \vee 1)/2}\\
    &\lesssim 1 + \sqrt{\frac{(\log(n))^{\rho \vee 1}}{K_n}}. \\
\end{split}
\end{equation*}
Using assumption A-\ref{ass:S} we get $\left(\EE \lVert U\rVert^{2\log(n)}\right)^{1/2\log(n)} \leq C$ and with a similar treatment to $L$ we finally get $\left(\EE \lVert W\rVert^{2\log(n)}\right)^{1/2\log(n)} \leq C$. Using Markov's inequality, 
$$ \PP\left[\lVert W\rVert \geq Ce\right] \leq \frac{1}{n^2}. $$
Finally, using Borel-Cantelli's lemma we get 
$$ \PP\left[\sup_{n}\lVert W\rVert <\infty\right] = 1. $$

\section{Proof of Lemma~\ref{lem:boundVariance}}
\label{app:boundvariance}
    We  prove both results by induction on $t$. The proof of the first item is very similar to \cite[Lemma 1]{HACHEM2024104276} and thus will be omitted. Let us now prove the second item.
     For $t= 1$ we have $R_i^1(1,1) = \sum_{\ell = 1}^n s_{i\ell} \left(h(x_\ell^0, \eta_\ell, 0)\right)^2 \geq \inf_{n\in\NN}\inf_{i\in [n]} \left(h(x_i^0, \eta_i, 0)\right)^2 \sum_{\ell = 1}^n s_{i\ell},$
        using assumptions A-\ref{ass:S}, A-\ref{ass:initial_point} and A-\ref{ass:nonDegen}-(1) we get the result. Suppose now that that exists $c>0$ such that 
        $$\forall n\in \NN, \forall i\in [n], \  \sigma_i :=\sqrt{ R^t_i(t,t)}\geq c.$$
        Let $\xi \sim \cN(0, 1)$, we can write \begin{align*}
            R_i^{t+1}(t+1, t+1) &= \sum_{\ell = 1}^n s_{i\ell} \EE \left(h(Z_\ell^t, \eta_\ell, t)\right)^2 
            =  \sum_{\ell = 1}^n s_{i\ell} \EE \left(h(\sigma_\ell \xi, \eta_\ell, t)\right)^2\\
            &\geq \EE \left(h(\sigma_\star \xi, \eta_\star, t)\right)^2 \sum_{\ell = 1}^n s_{i\ell},
        \end{align*}
    where $(\sigma_\star, \eta_\star)$ is such that $\EE \left(h(\sigma_\star \xi, \eta_\star, t)\right)^2 = \min_{\ell\in [n]} \EE \left(h(\sigma_\ell \xi, \eta_\ell, t)\right)^2$. Let $D>0$ be as in A-\ref{ass:nonDegen}-(2), using the induction hypothesis and the previous result we can see that $0< c \leq \sigma_\star \leq C$, using this gives the following 
    \begin{align*}
        \EE \left(h(\sigma_\star \xi, \eta_\star, t)\right)^2 &= \frac{1}{\sigma_\star\sqrt{2\pi}}\int_\RR \left(h(x, \eta_\star, t)\right)^2 \exp(-x^2 / 2 \sigma_\star^2)dx \\
        & \geq \frac{1}{C\sqrt{2\pi}} \int_{[-D, D]} \left(h(x, \eta_\star, t)\right)^2 \exp(-x^2 / 2\sigma_\star^2) dx \\
        & \geq \frac{\exp(-D^2 / 2\sigma_\star^2)}{C\sqrt{2\pi}}\int_{[-D, D]} \left(h(x, \eta_\star, t)\right)^2 dx \\
        & \geq \frac{\exp(-D^2 / 2c^2)}{C\sqrt{2\pi}} \inf_{\eta\in \cQ_{\eta}} \int_{[-D, D]} \left(h(x, \eta, t)\right)^2 dx\,.
    \end{align*}
    Finally assumption A-\ref{ass:nonDegen}-(2) gives the result.

\section{Hölder continuity of the squared root}
\label{app:srholdercont}

\begin{lemma}
    The function $X \mapsto X^{1/2}$ is $\frac{1}{2}$-Hölder continuous on $\mathcal{S}^n_+$ (the set of symmetric positive matrices).
\end{lemma}
\begin{proof}
    Let $A, B \in \mathcal{S}^n_+$, it suffices to show the following inequality,
    $$ \lVert A-B\rVert^2 \leq \lVert A^2-B^2\rVert. $$
    Let $\lambda$ be an eigenvalue of $A-B$ such that $|\lambda| = \lVert A-B\rVert$, then there exists $u\in \mathbb{R}^n$ of norm $1$ such that
    $$ (A-B)u = \lambda u .$$
    We can write the following
    $$A^2 - B^2 = (A-B)^2 + B(A-B) + (A-B)B,$$
    taking the quadratic form of this matrix at $u$ gives 
    $$ \lVert A^2 - B^2 \rVert \geq u^\top (A^2-B^2) u =  \lambda^2 + 2\lambda u^\top B u \,.  $$
    We can assume without loss of generality that $\lambda \geq 0$, having that $u^\top B u \geq 0$ gives 
    $$\lVert A^2 - B^2 \rVert \geq  \lambda^2 + 2\lambda u^\top B u \geq \lambda^2 = \lVert A-B\rVert^2. $$
\end{proof}
This result is used in the proof of Lemma~\ref{lem:RRtilde}.


\newcommand{\etalchar}[1]{$^{#1}$}

\end{document}